\documentclass[letterpaper,11pt]{article}
\usepackage[margin=1in]{geometry}  
\usepackage{setspace}  
\usepackage{natbib}
 \bibpunct[, ]{(}{)}{,}{a}{}{,}%

\usepackage{amsfonts}
\usepackage{amsthm}
\usepackage{mathrsfs}
\usepackage{amsmath}
\usepackage{amssymb}
\usepackage{latexsym}
\usepackage{indentfirst}
\usepackage{footmisc}
\usepackage{algorithm}
\usepackage{algorithmic}
\usepackage{float}
\usepackage{graphicx}
\usepackage{epstopdf}
\usepackage{longtable}
\usepackage{enumerate}
\usepackage{enumitem}
\usepackage{authblk}
\usepackage{bbm}
\usepackage{booktabs}
\usepackage{url}

\usepackage{multirow}
\usepackage{bm}
\usepackage{array}
\usepackage{subfigure}
\usepackage{caption}
\usepackage{color}
\usepackage{adjustbox}

\usepackage{pifont}  

\newcommand{\Pre}[1]{\mathcal{P}_{\text{\uppercase\expandafter{\romannumeral#1}}}}
\newcommand{\PCS}[1]{\text{PCS}_{\text{\uppercase\expandafter{\romannumeral#1}}}}
\newcommand{\x}{\mathbf{x}}
\newcommand{\X}{\mathbf{X}}
\newcommand{\fxx}{\mathbf{f(x)}}
\newcommand{\fx}[1]{\mathbf{f}(\mathbf{x}_#1)}

\newcommand{\T}{\text{T}}
\newcommand{\one}{\text{\uppercase\expandafter{\romannumeral1}}}
\newcommand{\two}{\text{\uppercase\expandafter{\romannumeral2}}}
\newcommand{\three}{\text{\uppercase\expandafter{\romannumeral3}}}
\newcommand{\bb}{\bm{\beta}}

\newtheorem{assumption}{ASSUMPTION}
\newtheorem{lemma}{LEMMA}
\newtheorem{remark}{REMARK}
\newtheorem{theorem}{THEOREM}
\newtheorem{proposition}{PROPOSITION}

\allowdisplaybreaks

\title{Unified Precision-Guaranteed Stopping Rules for Contextual Learning}

\author[a,b,1]{Mingrui Ding}
\author[a,2]{Qiuhong Zhao}
\author[b,*]{Siyang Gao}
\author[c,3]{Jing Dong}
\affil[a]{School of Economics and Management, Beihang University, Beijing, China}
\affil[b]{Department of Systems Engineering, City University of Hong Kong, Hong Kong, China}
\affil[c]{Decision, Risk, and Operations, Columbia Business School, New York, NY, USA}

\affil[1]{\texttt{mingrding2-c@my.cityu.edu.hk}}
\affil[2]{\texttt{qhzhao@buaa.edu.cn}}
\affil[3]{\texttt{jing.dong@gsb.columbia.edu}}
\affil[*]{Corresponding author: Siyang Gao, \texttt{siyangao@cityu.edu.hk}}

\date{\vspace{-8ex}}

\begin{document}

\maketitle

\begin{abstract}
Contextual learning seeks to learn a decision policy that maps an individual's characteristics to an action through data collection. In operations management, such data may come from various sources, and a central question is when data collection can stop while still guaranteeing that the learned policy is sufficiently accurate. We study this question under two precision criteria: a context-wise criterion and an aggregate policy-value criterion. We develop unified stopping rules for contextual learning with unknown sampling variances in both unstructured and structured linear settings. Our approach is based on generalized likelihood ratio (GLR) statistics for pairwise action comparisons. To calibrate the corresponding sequential boundaries, we derive new time-uniform deviation inequalities that directly control the self-normalized GLR evidence and thus avoid the conservativeness caused by decoupling mean and variance uncertainty. Under the Gaussian sampling model, we establish finite-sample precision guarantees for both criteria. Numerical experiments on synthetic instances and two case studies demonstrate that the proposed stopping rules achieve the target precision with substantially fewer samples than benchmark methods. The proposed framework provides a practical way to determine when enough information has been collected in personalized decision problems. It applies across multiple data-collection environments, including historical datasets, simulation models, and real systems, enabling practitioners to reduce unnecessary sampling while maintaining a desired level of decision quality.

\emph{Key words:} contextual learning, stopping rules, precision guarantees, unknown variances
\end{abstract}


\section{Introduction}

\label{sec:introduction}

Personalized decision making has become increasingly prevalent in operations management (OM). By leveraging individual-level characteristics, one can deploy \emph{policies} that map a person or system state (the context) to an action, thereby improving outcomes relative to uniform decision rules. Contextual learning provides a principled framework for constructing such policies from data.

In practice, these data may come from a variety of sources. We highlight three representative settings that arise in a broad range of OM applications.

\textit{(a) Offline datasets (passive learning).} In many applications, the learner begins with a fixed logged dataset. This setting is commonly studied under offline policy learning or off-policy evaluation \citep{hadad2021confidence,zhou2023offline,zhan2024policy}. The goal is to use the logged contexts, actions, and outcomes to construct a decision policy without further interaction with the system.

\textit{(b) Offline sequential experiments in simulation (active learning).} In simulation-based decision making, the learner can actively choose which context-action pairs to sample from a simulator and adapt these choices sequentially based on observed outcomes. In the literature, this setting is typically studied under contextual ranking and selection (R\&S) \citep{Shen2021,Du2024,keslin2025ranking,li2025efficient}. This paradigm is especially useful when real-world experimentation is costly or risky, and simulation offers a controllable environment for policy learning.

\textit{(c) Online sequential experiments in real systems (partially passive learning).} In real operations, contexts arrive exogenously (e.g., patients, users, or jobs) and are therefore outside the learner's control, unlike in simulation, where contexts can often be selected. After observing the context, the learner chooses an action and receives a noisy outcome. This setting is commonly modeled as a contextual multi-armed bandit \citep{li2010contextual,pan2020online,bastani2022learning,zhalechian2022online}. When the goal is to identify a high-quality policy rather than minimize cumulative regret during learning, the problem is studied as contextual best arm identification (BAI) or PAC contextual bandits \citep{li2022instance,Simchi2024}.

Moreover, these scenarios are often intertwined in practice rather than isolated. A typical workflow may begin with a historical dataset, then use simulation to evaluate candidate policies, and finally run a field experiment for validation. More generally, data sources may be continuously blended, for example by combining logged data with ongoing experimentation, thereby producing hybrid information streams.

Despite the diversity of these settings, practitioners repeatedly face the same operational question: when is the available information sufficient to stop collecting more data while still guaranteeing that the learned policy meets a prescribed level of quality? This question is important because additional sampling can be costly (e.g., in simulation time, experimental exposure, or opportunity cost), while stopping too early can result in an unreliable deployed policy.

Formulating this question rigorously leads to two key requirements for a practically useful method. First, because learning may rely on multiple data sources, the method should be \emph{source-agnostic}, remaining valid under different, and possibly hybrid, data streams and sampling mechanisms. Second, because sampling variances are typically unknown in real systems and must be estimated from data, the method should accommodate \emph{unknown variances}.

Despite substantial progress in contextual learning, existing stopping methods remain fragmented across settings and assumptions. In simulation-based contextual R\&S, guarantees are often tied to specific sampling designs \citep{Shen2021, keslin2025ranking}, which limits their applicability outside controlled simulation environments. In contextual BAI, available stopping rules are either computationally difficult to implement or practically conservative, and they typically assume known sampling variances \citep{li2022instance, Simchi2024}. In offline policy learning, the emphasis is usually on regret or value estimation rather than on certifying whether the currently available data are sufficient to meet a prescribed precision target \citep{zhou2023offline,zhan2024policy}. More fundamentally, existing stopping guarantees are usually tied to environment-specific assumptions on how data are sampled or collected and do not readily extend to other data-collection environments. As a result, existing methods do not provide a unified answer to the stopping problem in contextual learning.

To overcome this limitation, we develop precision-guaranteed stopping rules that are valid across the data sources in (a), (b), and (c), including their hybrids, and remain applicable when sampling variances are unknown. Our starting point is a common abstraction: regardless of the data source, contextual learning can be viewed as a sequential information-collection process that generates observations adapted to a filtration, with a stopping time that is itself random. This viewpoint is natural in online experimentation, but it is equally useful for simulation, where the learner actively chooses context-action pairs, and for logged datasets, where the observed records can be interpreted as a predetermined sampling path.

Methodologically, our stopping rules are based on GLR statistics, which quantify the evidence that the currently estimated optimal action dominates its competitors. A central challenge is to calibrate corresponding evidence boundaries that remain valid uniformly over time while avoiding excessive conservativeness \citep{Garivier2016Optimal}. Existing GLR-based calibrations often rely on indirect proxy bounds, which can lead to unnecessarily large boundaries, especially when sampling variances are unknown \citep{Jourdan}. To address this issue, we develop new time-uniform deviation inequalities that directly control the relevant self-normalized GLR evidence. This yields substantially tighter stopping boundaries while preserving rigorous finite-sample guarantees.

We study two widely used precision criteria in contextual learning. The first, Weighted-PAC ($\mathcal{P}_{\one}$), requires the selected action to be near-optimal for a randomly drawn context with high probability. The second, PAC ($\mathcal{P}_{\two}$), requires the expected performance of the selected policy under the context distribution to be near-optimal with high probability. These criteria capture different operational priorities, with $\mathcal{P}_{\one}$ emphasizing context-wise reliability and $\mathcal{P}_{\two}$ emphasizing aggregate performance. We consider two modeling settings: an \emph{unstructured} setting that makes no structural assumptions on the relationship between the context-action pair and its performance, and a \emph{structured linear} setting with action-specific linear models. The unstructured formulation is most appropriate when the context set is moderate in size and one prefers to avoid structural assumptions, while the linear formulation is better suited to larger contextual spaces, where pooling information across contexts enables both statistical and computational scalability.

Our main contributions are as follows.

First, we formulate a unified stopping problem for contextual learning under multiple data sources and unknown variances. Our framework covers offline datasets, simulation-based experiments, online learning, and hybrid data streams within a single sequential perspective.

Second, we develop new time-uniform deviation inequalities that directly control the self-normalized terms underlying the plug-in GLR statistics. These inequalities lead to stopping boundaries that are substantially less conservative while preserving finite-sample guarantees.

Third, we extend the framework to structured linear contextual learning with action-specific linear models. For this setting, we derive the corresponding GLR statistics and a new mixture-martingale-based calibration that controls directional uncertainty and accommodates unknown variances. More broadly, our analysis contributes a new tool for controlling directional parameter deviations, whereas existing linear bandit methods typically focus on worst-case deviations over all directions \citep{Abbasi2011improved, jedra2020optimal}.

Finally, we characterize the performance of the proposed stopping rules both theoretically and numerically. We derive sample-size bounds under equal allocation and show that our rules improve on a strong existing benchmark for simulation-based contextual learning \citep{Shen2021}. We also demonstrate strong empirical performance across synthetic problems and case studies relative to existing benchmarks, e.g., the method in \citet{Simchi2024}, even though their method is given access to the true sampling variances, whereas ours must estimate them from the data.

\section{Literature Review}
\label{sec-literature}

Our work relates to three streams of research: contextual bandits, contextual ranking and selection, and sequential stopping via GLR tests and martingale methods.

\paragraph{Contextual bandits.}
Most contextual multi-armed bandit (MAB) research focuses on online learning algorithms that minimize cumulative regret, with applications including clinical trials, recommendation systems, and operational resource allocation \citep{li2010contextual,karimi2018news,pan2020online,bastani2022learning,zhalechian2022online,delshad2022adaptive,kinyanjui2023fast,wang2025reinforcement}. A more closely related line of work studies contextual BAI, where the objective is to output a high-quality policy with confidence guarantees rather than optimize cumulative reward during learning \citep{li2022instance,Simchi2024}. However, existing methods in this literature either rely on computationally intensive elimination procedures or use conservative empirical lower bounds, and they typically assume known sampling variances, which are often unavailable in practice. In contrast, we develop implementable stopping rules that remain valid with unknown variances and under a broader range of data-collection mechanisms, including offline, simulation-based, online, and hybrid settings.

\paragraph{Contextual ranking and selection.}
In simulation-based contextual decision problems, the contextual R\&S literature studies how to allocate simulation effort across context-action pairs in order to identify a good policy \citep{Shen2021,Du2024,keslin2025ranking,li2025efficient}. However, the associated stopping guarantees in this literature are typically tied to specific simulation designs and do not readily extend to settings in which contexts arrive exogenously, sampling rules are more generally adaptive, or data are combined across heterogeneous sources. Our framework fills this gap by developing stopping rules whose validity does not rely on a specific sampling design. Instead, the guarantees are formulated to hold for any sampling process adapted to the observed filtration, provided the conditional sampling model holds.

\paragraph{Sequential stopping via GLR tests and martingales.} A central challenge in sequential inference is to calibrate stopping rules that remain valid at arbitrary random stopping times. Time-uniform deviation inequalities provide a principled way to construct such boundaries while preserving favorable asymptotic behavior \citep{Kaufmann}. In the fixed-confidence BAI literature, a standard approach is to combine these inequalities with GLR statistics, a classical tool in sequential analysis \citep{Chernoff1959Sequential}, to certify pairwise arm comparisons \citep{Garivier2016Optimal,Qin2017,Kaufmann}. However, most existing results assume known sampling variances. Relaxing this assumption is not trivial, because plug-in GLR statistics depend jointly on estimated means and variances. Recent work by \citet{Jourdan} addresses this issue using peeling-based time-uniform arguments and related techniques \citep{howard2020time}, but the resulting boundaries can be overly conservative in practice.

We build on this GLR-based framework, but extending it to contextual learning introduces several additional difficulties. First, the object of inference is a policy over contexts rather than a single best arm, so the analysis must aggregate many context-dependent comparisons. Second, unknown variances further complicate the structure of the GLR statistics. Third, in the structured linear setting, the relevant quantities are directional, action-specific deviations rather than global parameter bounds. To address these challenges, we develop new time-uniform inequalities that directly control the self-normalized GLR evidence, leading to tighter and more practically effective stopping boundaries.

The rest of the paper is organized as follows. In
Section~\ref{sec-problem formulation}, we formulate the problem and
precision measures. Section~\ref{sec-stopping rules general} develops
stopping rules under the unstructured setting, and
Section~\ref{sec-stopping rules linear} extends them to the structured linear
setting. Numerical experiments are presented in
Section~\ref{sec-numerical experiments}, followed by conclusions and discussions in Section~\ref{sec-conclusions}.

\section{Problem Formulation}   
\label{sec-problem formulation}

We consider a finite set of actions $\mathcal{A}$ and a finite set of contexts $\mathcal{X}\subset\mathbb{R}^d$, where $d$ is the dimension of the context. Let $y(\x,a)$ denote the mean performance of action $a$ under context $\x$. We study two settings. In the unstructured setting, we make no structural assumptions on the mapping $(\x,a)\mapsto y(\x,a)$, which is suitable when the finite context set is relatively small. In the structured linear setting, for each action $a$, the mean performance is assumed to depend linearly on the context, which is useful when the context set is finite but potentially large.

For each context $\x\in\mathcal{X}$, let $\mathcal{A}(\x)\subseteq \mathcal{A}$ denote the set of feasible actions. Throughout the paper, we focus on a decomposable policy class,
\[
\Pi:=\{\pi:\mathcal{X}\to\mathcal{A}:\ \pi(\x)\in\mathcal{A}(\x),\ \forall\,\x\in\mathcal{X}\}.
\]
Thus, a policy specifies one feasible action for each context independently. If the numbers of contexts and actions are $m$ and $k$, respectively, then we have $|\Pi| \leq k^m$. The optimal policy $\pi^*$ satisfies
\[
\pi^*(\x)\in {\arg\max}_{a\in\mathcal{A}(\x)} y(\x,a),
\qquad \forall\,\x\in\mathcal{X}.
\]

Under the above decomposable policy class, identifying the optimal policy is equivalent to identifying the best action under each context. We retain the policy notation because the final output is a policy and the precision criteria below are defined at the policy level under the context distribution.

Contextual learning can be formulated as a sequential sampling process aimed at identifying the optimal policy $\pi^*$. Let $\mathcal{F}_t$ denote the $\sigma$-algebra generated by the observations up to stage $t$, i.e.,
\begin{equation*}
    \mathcal{F}_{t} = \sigma\big( (\X_1, A_1, Y_1), \dots, (\X_t, A_t, Y_t) \big).
\end{equation*}
Here, $Y_t$ is a noisy sample for the performance of action $A
_t$ under context $\X_t$ and the sampling decision $(\X_t, A_t)$ may depend on past information and is assumed to be $\mathcal{F}_{t-1}$-measurable. The sequence $\{\mathcal{F}_{t}\}_{t\geq 1}$ forms a filtration. Let $\varepsilon_t$ denote the sampling noise, then we have
\begin{equation*}
    Y_t = y(\X_t, A_t) + \varepsilon_t.
\end{equation*}
This abstraction is flexible enough to cover data collected from various sources, which will be discussed in detail later through a common OM example.

We make the following technical assumptions for our analysis. 

\begin{assumption}
    The sampling noises $\{\varepsilon_t\}_{t\geq 1}$ are independent across different time stages $t$.
\end{assumption}
\begin{assumption}
    For each time stage $t$, conditional on $(\X_t=\x, A_t=a)$, the sampling noise $\varepsilon_t$ follows a Gaussian distribution with mean zero and variance $\sigma^2(\x, a)$.
\end{assumption}

These assumptions impose independence across samples and Gaussian noise for each context-action pair, which are standard in the literature on sequential testing, BAI, and contextual learning \citep{Kaufmann, li2022instance, delshad2022adaptive, Jourdan, zhan2024policy}. The independence assumption ensures that observations do not exhibit temporal dependence beyond what is captured by the adaptive sampling rule, which simplifies the martingale-based analysis and is commonly adopted in both simulation-based and online learning settings. The Gaussian assumption allows us to derive explicit likelihood-based statistics and sharp deviation inequalities, and can often be justified either directly (e.g., when observations are averaged over multiple replications) or approximately via central limit arguments.

Let $\hat{y}_{t} (\x,a)$ denote the estimated performance of action $a$ under context $\x$ based on observations up to stage $t$. The estimated optimal policy at stage $t$ is denoted by $\hat{\pi}_{t}$ and is defined as
\begin{equation*}
    \forall ~\x\in\mathcal{X}, ~\hat{\pi}_{t}(\x) = {\arg\max}_{a\in \mathcal{A}(\x)} ~\hat{y}_{t}(\x,a).
\end{equation*}

The stopping rule is defined as a stopping time $\tau$ with respect to the filtration $\{\mathcal{F}_t\}_{t \geq 1}$. When the stopping rule is triggered, the sampling process terminates and outputs a policy $\hat{\pi}_{\tau}$, which is measurable with respect to $\mathcal{F}_{\tau}$.
In many applications, differences in mean performance below a certain threshold are practically negligible, as they do not lead to meaningfully different decisions, while reliably distinguishing them may require a disproportionate number of samples. To capture this, we introduce a parameter $\delta \geq 0$ representing the smallest performance gap that is practically relevant to detect. When $\delta > 0$, $\delta$ is commonly referred to as the indifference-zone parameter.

We consider two types of precision guarantees for the identified policy $\hat{\pi}_{\tau}$. We denote them by $\mathcal{P}_{\one}$ and $\mathcal{P}_{\two}$, defined as
\begin{align*}
    \Pre{1} &= \mathbb{E}_{\X}\Big[\mathbb{P}\Big(y\big(\X,\hat{\pi}_{\tau}(\X) \big) \geq y\big(\X,\pi^*(\X) \big) - \delta \,\Big|\, \X \Big) \Big],\\
    \Pre{2} &= \mathbb{P} \Big(\mathbb{E}_{\X}\Big[y\big(\X,\hat{\pi}_{\tau}(\X) \big)\Big] \geq \mathbb{E}_{\X}\Big[y\big(\X,\pi^*(\X) \big)\Big] - \delta \Big),
\end{align*}
where both expectations are taken with respect to the randomness of the context $\X$. The distribution of the context $\mathcal{C}$ is available to the learner prior to sampling, with $\mathbb{P} (\X = \x) = p(\x) > 0$ for each context $\x\in\mathcal{X}$. In practice, the distribution $\mathcal{C}$ can be estimated from historical data.

The two criteria capture different notions of precision and do not generally dominate each other. Measure $\mathcal{P}_{\one}$ provides a context-wise guarantee: for a randomly realized context $\X$, it quantifies the context-distribution-weighted probability that the selected action $\hat{\pi}_{\tau}(\X)$ is within $\delta$ of the context-wise optimal action $\pi^*(\X)$. Thus, $\mathcal{P}_{\one}$ emphasizes reliability across realized contexts and is appealing when one wants the learned policy to perform well broadly across the population, rather than allowing poor decisions in some contexts to be offset by gains in others. This criterion has been used in both contextual R\&S and contextual BAI \citep{Shen2021,Du2024,Simchi2024}, where it is referred to as $\mathrm{PCS}_{\mathrm{E}}$ or Weighted-PAC.
In contrast, $\mathcal{P}_{\two}$ provides an aggregate guarantee: it controls, with high probability, the expected performance of the selected policy under the context distribution. Equivalently, it requires that the average value of the deployed policy be within $\delta$ of optimal with high probability. This criterion is natural when the main objective is overall system performance, and larger errors in some low-probability contexts are acceptable as long as the total expected value remains close to optimal. In the contextual BAI literature, this criterion is often referred to simply as PAC \citep{li2022instance}.

These two guarantees correspond to different operational priorities. Measure $\mathcal{P}_{\one}$ is better aligned with settings where context-wise reliability, consistency, or fairness across segments is important. Measure $\mathcal{P}_{\two}$ is better aligned with settings where aggregate efficiency is the primary concern. For example, in a healthcare application, one may prefer $\mathcal{P}_{\one}$ if it is undesirable for the learned policy to perform poorly for a non-negligible subset of patient types, even when the average outcome remains good. By contrast, in a revenue-management or inventory application, $\mathcal{P}_{\two}$ may be the more natural objective if the decision maker mainly cares about achieving near-optimal expected profit over the overall demand mix.

Let $\tau_{\alpha,\delta}^{\one}$ and $\tau_{\alpha,\delta}^{\two}$ denote stopping rules that guarantee the target precision level $1-\alpha$ under the smallest detection parameter $\delta$, i.e., $\mathcal{P}_{\one}\ge 1-\alpha$ and $\mathcal{P}_{\two}\ge 1-\alpha$, respectively. Section \ref{sec-stopping rules general} develops their stopping rules under unstructured contexts, and Section \ref{sec-stopping rules linear} extends them to the structured linear setting. From a practical standpoint, the two formulations are suited to different problem scales. The unstructured approach is most appropriate when the context set is of moderate size, and one prefers to avoid structural assumptions on the mean reward function. In contrast, when the number of contexts is large or grows combinatorially with underlying features, context-wise certification can become computationally intensive. In such settings, the structured linear formulation is more attractive, as it pools information across contexts through a parametric model and enables more scalable inference. The two approaches should therefore be viewed as complementary. The unstructured formulation offers modeling flexibility for smaller problems, while the linear formulation provides a scalable alternative when its structural assumptions are appropriate.

\paragraph{A common OM example that generates hybrid data.}
We present a representative OM example that naturally generates a hybrid data stream, which many existing contextual learning methods cannot readily handle, namely a newsvendor-style inventory control problem with a digital twin.
Consider a retailer that must choose, for each store-week (context $\x$), an
order quantity policy $a\in\mathcal{A}(\x)$, and observes a realized
profit $Y$ after demand is realized.

The firm typically has three sources of data:
\begin{itemize}[leftmargin=2em]
\item \textit{Historical operational logs.}
Past store-week records provide predetermined observations $(\X_t,A_t,Y_t)$ under
legacy policies. These data are offline in the sense that the sequence of
$\{(\X_t,A_t)\}$ is fixed by past operations and cannot be adaptively redesigned.

\item \textit{Simulation / digital-twin experiments.}
Before changing the live replenishment policy, analysts run a calibrated demand
model (or a discrete-event simulator) and actively choose scenarios
$(\X_t,A_t)$ to evaluate (e.g., stress-testing high-variance demand regimes,
rare disruptions, or specific store segments). This is adaptive sampling because
scenario selection depends on previously observed simulation outputs.

\item \textit{Online pilots.}
The retailer then conducts a limited pilot in production: as new weeks arrive,
contexts $\X_t$ are realized by the business environment, and the firm assigns
actions $A_t$ (possibly adaptively) to a subset of stores, while
monitoring outcomes and deciding whether to stop the pilot early.
\end{itemize}

Operationally, these three stages often overlap rather than proceeding in a strictly sequential manner. Simulation runs are launched to investigate unexpected pilot results,
and additional offline slices are pulled to sanity-check segment-level effects.
This overlap produces a hybrid stream in which $(\X_t,A_t)$ is partly
predetermined (historical logs), partly chosen by the analyst (simulation), and partly
driven by exogenous arrivals with adaptive assignment (online pilots). 

Recall that $\{\mathcal{F}_t\}_{t\geq1}$ is the natural filtration generated by the
observations up to stage $t$. Under the three data sources above, the pair $(\X_t,A_t)$ is generated in different ways. For historical logs, the entire sequence $\{(\X_t,A_t)\}_{t\geq 1}$ is predetermined and can be viewed as $\mathcal{F}_0$-measurable. For simulations, the learner selects $(\X_t,A_t)$ adaptively based on past information, therefore $(\X_t,A_t)$ is $\mathcal{F}_{t-1}$. For online pilots, $\X_t$ arrives from an environment and $A_t$ is chosen after observing $\X_t$ thus $A_t$ is measurable with respect to $\sigma(\mathcal{F}_{t-1} , X_t)$-measurable. 

Therefore, regardless of how the hybrid stream overlaps, the sampling decision at stage $t$ is always made without access to the current observation $Y_t$. Under Assumptions 1 and 2, this implies that for each context-action pair $(\x,a)$, the observations satisfy the same conditional distribution:
\begin{equation*}   \label{eq-conditional-gaussian}
    Y_t \,\big|\, \mathcal{F}_{t-1},(\X_t=\x,A_t=a) \sim \mathcal{N}\!\big(y(\x,a),\sigma^2(\x,a)\big).
\end{equation*}
That is, even though the sequence $\{(\X_t,A_t)\}$ may be partly predetermined and partly adaptively selected, the conditional law of $Y_t$ given the past and the current sampled pair remains invariant. This formulation enables us to construct unified stopping rules for hybrid data streams that combine historical logs, simulation experiments, and online pilots.

\section{Stopping Rules under the Unstructured Setting}   \label{sec-stopping rules general}
In this section, we develop stopping rules for the unstructured contextual setting. Our objective is to determine, at a data-dependent time, whether the currently estimated optimal policy is sufficiently close to optimal under the prescribed precision criterion. Conceptually, the development proceeds in three steps: (i) construct evidence via GLR statistics; (ii) calibrate corresponding time-uniform boundaries; and (iii) formalize the stopping rule.

A key difficulty is that the stopping time itself is random, so classical fixed-sample critical values are invalid. The boundary must control the probability of a false declaration uniformly over time. Moreover, since sampling variances are unknown and estimated from data, the calibration must jointly control the deviations of sample means and sample variances.

The main technical contribution of this section is to derive such a time-uniform deviation inequality tailored to the plug-in GLR statistic. This allows us to construct stopping rules that are theoretically valid while remaining substantially less conservative than existing approaches. We next turn to the formal development.

Under the unstructured setting, for any action $a$ and context $\x$, the estimate $\hat{y}_{t}(\x,a)$ for the mean performance $y(\x,a)$ is the sample mean. Let $N_{t}(\x,a) = \sum_{s=1}^{t} \bm{1}\{A_s=a, \X_{s}=\x\}$ denote the sample size accumulated for action $a$ under context $\x$ up to stage $t$. Then, the sample mean takes the form $\overline{Y}_{t}(\x,a) = \frac{1}{N_{t}(\x,a)} \sum_{s=1}^{t} \bm{1} \{A_s=a, \X_s=\x\} Y_{s}$ and the sample variance takes the form $S_{t}^2(\x,a) = \frac{1}{N_{t}(\x,a)-1} \sum_{s=1}^{t} \bm{1} \{A_s=a, \X_{s}=\x\} (Y_{s} - \overline{Y}_{t}(\x,a))^2$.

\subsection{Measure $\Pre1$}   
\label{sec-Pre1}
The measure $\Pre1$ is naturally context-wise. By definition, one might attempt to directly compare policies by defining a GLR statistic 
that quantifies the evidence that policy $\pi$ outperforms policy $\pi'$ under context $\x$ by at least the smallest slack level $\delta$, and then control such evidence jointly over all contexts $\x \in \mathcal{X}$. However, this joint control idea cannot be achieved using GLR-type sequential tests. This will be discussed in detail in Appendix A.

Instead, we guarantee $\Pre1$ by controlling the error within each context
separately. For each $\x\in\mathcal{X}$, we aim to identify an action that is
$\delta$-optimal for that context. This reduces the problem to repeated
pairwise certifications between actions under the same context.

Fix a context $\x\in\mathcal{X}$ and two actions $a,a'\in\mathcal{A}(\x)$. Let $\underline{Y_{t}}(\x,a) := \big(Y_s:\ \X_s=\x,\ A_s=a,\ s\le t\big)$ denote the vector of observations collected for $(\x,a)$ up to stage $t$. Let $\mathcal{L}_{\mu(\x,a)} \left(\underline{Y_{t}}(\x,a) \right)$ denote the likelihood of these observations under the parameter $\mu(\x,a)\in\mathbb{R}$. Under the Gaussian noise model, the likelihood function is
\[
\mathcal{L}_{\mu(\x,a)}\big(\underline{Y_{t}}(\x,a)\big)
=
\prod_{s:\ \X_s=\x,\ A_s=a,\ s\le t}
\frac{1}{\sqrt{2\pi\sigma^2(\x,a)}}
\exp\!\left(
-\frac{(Y_s-\mu(\x,a))^2}{2\sigma^2(\x,a)}
\right).
\]
When the variance is known, the GLR statistic for testing
\[
H_0:\ \mu(\x,a)\le \mu(\x,a')-\eta
\quad \text{versus}\quad
H_1:\ \mu(\x,a)\ge \mu(\x,a')-\eta
\]
is defined as
\begin{equation*}
Z_{a,a'} (\x,t,\eta) 
= 
\log 
\frac{
\displaystyle
\max_{\mu(\x,a) \geq \mu(\x,a') - \eta} 
\mathcal{L}_{\mu(\x,a)} \left(\underline{Y_{t}}(\x,a) \right) 
\mathcal{L}_{\mu(\x,a')} \left(\underline{Y_{t}}(\x,a')\right)
}{
\displaystyle
\max_{\mu(\x,a) \leq \mu(\x,a') - \eta} 
\mathcal{L}_{\mu(\x,a)} \left(\underline{Y_{t}}(\x,a)\right) 
\mathcal{L}_{\mu(\x,a')} \left(\underline{Y_{t}}(\x,a')\right)}.
\end{equation*}

Since the variances $\sigma^2(\x,c)$ are unknown in practice, we replace them by the corresponding sample variances $S_t^2(\x,c)$. This yields the plug-in GLR statistic
\begin{align}
\label{equa-local-glr}
\tilde Z_{a,a'}(\x,t,\eta)
&=
\frac{1}{2}
\frac{\big(\overline Y_t(\x,a)-\overline Y_t(\x,a')+\eta\big)^2}
{\frac{S_t^2(\x,a)}{N_t(\x,a)}+\frac{S_t^2(\x,a')}{N_t(\x,a')}}.
\end{align}

Define the pairwise GLR boundary for $\Pre1$ as a mapping $\varphi_{a,a'}^{\one}: {\mathbb{N}}^{m\times k} \times (0,1) \times \mathcal{X} \mapsto \mathbb{R}^*$. Here and below, we write $\bm N_t=\big(N_t(\x,a)\big)_{a\in\mathcal A,\ \x\in\mathcal X}$ for the vector of sample sizes. Recall that $\hat\pi_t$ denotes the estimated optimal policy at stage $t$. For a given context $\x$ and any challenger $a\in\mathcal A(\x)\setminus\{\hat\pi_t(\x)\}$, the statistic $\tilde Z_{\hat\pi_t(\x),a}(\x,t,\delta)$ quantifies the evidence that the estimated optimal action $\hat\pi_t(\x)$ dominates action $a$ under context $\x$ within slack $\delta$. When this evidence exceeds the corresponding pairwise GLR boundary for all challengers and all contexts, the sampling process terminates and $\Pre{1}$ is satisfied. Formally, the stopping rule for $\Pre1$ is defined as
\begin{equation}   \label{equa-stopping rule pre1}   
    \tau_{\alpha,\delta}^{\one} = \inf \Big\{ t\in \mathbb{N}^* :~\forall\,\x \in \mathcal{X}, ~\forall\, a\in \mathcal{A}(\x) \setminus \{\hat \pi_t(\x)\}, ~\tilde{Z}_{\hat{\pi}_t(\x),a}(\x,t,\delta) > \varphi_{\hat{\pi}_t(\x),a}^{\one} (\bm{N}_t,\alpha,\x) \Big\}.   
\end{equation}

The boundary $\varphi_{a,a'}(\bm{N}_t,\alpha,\x)$ plays the role of a \emph{sequential critical value} for the plug-in GLR statistic. Unlike fixed-sample testing, however, we evaluate this evidence repeatedly over time and stop at a \emph{data-dependent} time. To maintain a valid error probability uniformly over all times, the boundary must incorporate a time-uniform correction, which leads to a slowly increasing boundary. The precise form of $\varphi_{a,a'}(\bm{N}_t,\alpha,\x)$ is derived in Section~\ref{sec-stopping thresholds general} via a time-uniform deviation inequality tailored to the plug-in GLR statistic.

\subsection{Measure $\Pre2$}   \label{sec-Pre2}
By the definition of $\Pre{2}$, one could alternatively construct a stopping rule based on a policy-level GLR statistic that directly compares the current estimate $\hat{\pi}_t$ with all competing policies in $\Pi$. While conceptually natural, such an approach requires exhaustive certification over the policy space, which quickly becomes computationally prohibitive when $|\Pi|$ grows exponentially with the number of contexts. 
Accordingly, we do not pursue this approach as an implementable procedure. Instead, we construct the stopping rule using the pairwise GLR statistics defined in~\eqref{equa-local-glr}.

For a context $\x$ and two actions $a,a'\in\mathcal{A}(\x)$, define
\begin{equation}
\label{equa-local-slack-pre2}
w_{a,a'}(\x,t)
:=
\inf\Big\{
\eta\ge 0:\ 
\tilde Z_{a,a'}(\x,t,\eta)
>
\varphi_{a,a'}^{\two}(\bm N_t,\alpha,\x)
\Big\},
\end{equation}
where $\varphi_{a,a'}^{\two}$ is the pairwise GLR boundary for $\Pre2$, sharing the same construction with $\Pre{1}$ and to be calibrated in Section~\ref{sec-stopping thresholds general}. The quantity $w_{a,a'}(\x,t)$ is the smallest slack level at which action $a$ is certified to dominate another action $a'$. By \eqref{equa-local-glr}, it admits the explicit form
\[
w_{a,a'}(\x,t)
=
\left[
\sqrt{
2\,\varphi_{a,a'}^{\two}(\bm N_t,\alpha,\x)
\left(
\frac{S_t^2(\x,a)}{N_t(\x,a)}
+
\frac{S_t^2(\x,a')}{N_t(\x,a')}
\right)
}
-
\big(\overline Y_t(\x,a)-\overline Y_t(\x,a')\big)
\right]_+,
\]
where $[z]_+=\max\{z,0\}$.

We then define the certified regret bound of the estimated optimal action $\hat{\pi}_t(\x)$ under context $\x$:
\begin{equation}
\label{equa-local-regret-pre2}
r(\x,t)
:=
\max_{a\in\mathcal A(\x)\setminus\{\hat{\pi}_t(\x)\}}w_{\hat{\pi}_t(\x),a}(\x,t).
\end{equation}
This leads to the following implementable stopping rule for $\Pre2$:
\begin{equation}
\label{equa-stopping rule pre2}
\tau_{\alpha,\delta}^{\two}
=
\inf\Big\{
t\in\mathbb N^*:\ 
\sum_{\x\in\mathcal X}p(\x)\,r(\x,t)\le \delta
\Big\}.
\end{equation}

Thus, for $\Pre2$, the pairwise GLR boundaries are not used directly at the fixed slack $\delta$. Instead, they are first inverted to construct the certified slack levels $w_{a,a'}(\x,t)$ and the context-wise regret bounds $r(\x,t)$, which are then aggregated across contexts according to the context distribution.

\subsection{Calibration of GLR Boundaries}   \label{sec-stopping thresholds general}
To implement the stopping rules \eqref{equa-stopping rule pre1} and \eqref{equa-stopping rule pre2}, it remains to calibrate the pairwise GLR boundaries $\varphi_{a,a'}^{\one}$ and $\varphi_{a,a'}^{\two}$. For $\Pre1$, the boundary is used directly at the fixed slack level $\eta=\delta$ in the pairwise GLR certifications. For $\Pre2$, the boundary is inverted to construct the certified slack levels in \eqref{equa-local-slack-pre2}. In both cases, the purpose of the boundary is the same: it should make the probability of a false pairwise certification sufficiently small so that the resulting stopping rule satisfies the target guarantee.

Fix a context $\x\in\mathcal X$ and an action $a\in\mathcal A(\x)\setminus\{\pi^*(\x)\}$. For $\Pre1$, it suffices to rule out false certification of $a$ against $\pi^*(\x)$ at slack $\delta$. For $\Pre2$, it suffices that the certified slack dominates the true gap $\Delta_a(\x):=y(\x,\pi^*(\x))-y(\x,a)$. In both cases, the key quantity is the statistic $\tilde Z_{a,\pi^*(\x)}(\x,t,\eta)$ for a slack level $\eta$ satisfying $y(\x,a)\le y(\x,\pi^*(\x))-\eta$. Its magnitude is governed by the summed self-normalized deviation
\begin{equation}   \label{equa-summed deviation term}
    U_t(\x;a,\pi^*(\x)) := \sum_{c\in\{a,\pi^*(\x)\}} N_{t}(\x,c) \frac{\left(\overline{Y}_{t}(\x,c) - y(\x,c) \right)^2}{2S_{t}^2 (\x,c)}.
\end{equation}
Accordingly, both $\varphi_{a,a'}^{\one}$ and $\varphi_{a,a'}^{\two}$ are calibrated through a time-uniform deviation inequality for the sum of two self-normalized terms.

\begin{lemma}   \label{lemma-deviation inequality}
    Let $r\in\{1,2\}$ index two sample streams with unknown means $y_r$ and let $N_{t,r}, \overline{Y}_{t,r}, S_{t,r}^2$ denote the corresponding sample size, sample mean, and sample variance at stage $t$. Define $U_t^{+} = \sum_{r=1}^{2} N_{t,r} \frac{\left(\overline{Y}_{t,r} - y_r \right)^2}{2S_{t,r}^2}$. Then, with probability greater than $1-\alpha$, for all $t\in\mathbb{N}^*$,
    \begin{equation}   \label{equa-deviation inequality}
        U_t^{+} \leq \max \left\{\frac{1}{2} \gamma\left(N_{t,1}, \alpha \sqrt{\frac{1}{N_{t,2}+1}}\right), \frac{1}{2} \gamma\left(N_{t,2}, \alpha \sqrt{\frac{1}{N_{t,1}+1}}\right) \right\},
    \end{equation}
    where $\epsilon > 0$ can be arbitrarily small and the function $\gamma:\mathbb{N}^* \times (0,1) \to (0,+\infty)$ is defined as 
    \begin{equation}   \label{equa-gamma} 
      \gamma (t,\alpha) = \frac{t^2}{\max\left\{\left(\frac{\alpha^2}{t+1} \right)^{\frac{1}{t}} (t+1) - 1, \epsilon \right\} } - t.
    \end{equation}
\end{lemma}

Lemma~\ref{lemma-deviation inequality} follows from the Gaussian mixture martingales in \citet{wang2025anytime} together with Ville's maximal inequality. Unlike approaches that first bound the two self-normalized terms separately and then apply a union bound, e.g. \cite{Jourdan}, Lemma~\ref{lemma-deviation inequality} controls their sum directly, which leads to less conservative GLR boundaries.

For a target precision level $1-\alpha$, the remaining step is to allocate the error budget across contexts and pairwise comparisons. Under $\Pre1$, context-wise failures are weighted by $p(\x)$, leading to
$\alpha^{\one}(\x) := \frac{\alpha}{\big(|\mathcal A(\x)|-1\big)mp(\x)}$. Under $\Pre2$, the pairwise certifications must hold simultaneously across contexts, leading to $\alpha^{\two}(\x) := \frac{\alpha}{\big(|\mathcal A(\x)|-1\big)m}$. Using these allocations together with Lemma~\ref{lemma-deviation inequality}, we obtain the following result.

\begin{theorem}   \label{theorem1}
    Let $r\in\{1,2\}$ index the two precision notions $\one$ and $\two$. Under the unstructured setting, for each $r\in\{1,2\}$, each context $\x\in\mathcal X$, and each pair of actions $a,a'\in\mathcal A(\x)$, let
    \begin{equation}   \label{equa-threshold unified}
      \begin{aligned}
        \varphi_{a,a'}^{(r)}(\bm N_t,\alpha,\x)
        =
        \max\Bigg\{
        &\frac{1}{2}\gamma\!\left(
        N_t(\x,a),\,
        \alpha^{(r)}(\x)\sqrt{\frac{1}{N_t(\x,a')+1}}
        \right),\\
        &\frac{1}{2}\gamma\!\left(
        N_t(\x,a'),\,
        \alpha^{(r)}(\x)\sqrt{\frac{1}{N_t(\x,a)+1}}
        \right)
        \Bigg\}.
      \end{aligned}
    \end{equation}
    Then the stopping rule $\tau_{\alpha,\delta}^{(r)}$ satisfies the corresponding target guarantee: when $r=1$, $\Pre1\ge 1-\alpha$; when $r=2$, $\Pre2\ge 1-\alpha$.
\end{theorem}

We next discuss the growth of the calibrated boundary. Define $\rho (t,\alpha) = \left(\frac{\alpha^2}{t+1} \right)^{\frac{1}{t}} (t+1) - 1$, then the boundary function $\gamma(t,\alpha)$ becomes nontrivial only when $\rho (t,\alpha) > 0$. We can solve that the length of this initial stage is on the order of $\frac{2\log(1/\alpha)}{\log\log(1/\alpha)}$. When $\alpha = 0.05$, we can numerically solve the length of this initial inactive stage is 4. To characterize the boundary after this initial stage, let
\[
\gamma_0(t,\alpha)
:=
\frac{t^2}{\left(\frac{\alpha^2}{t+1}\right)^{1/t}(t+1)-1}-t
\]
denote the untruncated version of $\gamma$.

\begin{proposition}   \label{prop-gamma-asymptotic}
    Fix $\alpha\in(0,1)$. As $t\to\infty$,
    \begin{equation*}
        \gamma_0(t,\alpha)
        =
        2\log(1/\alpha)+\log(t+1)+o(1).
    \end{equation*}
\end{proposition}

Proposition~\ref{prop-gamma-asymptotic} shows that, once active, the boundary decomposes into a precision term $2\log(1/\alpha)$ and a time-uniformity term $\log(t+1)$. Thus, our calibration yields a logarithmically growing boundary in the sampling stage.

A natural question is whether one can improve this growth to $\mathcal O(\log\log t)$, as suggested by the law of the iterated logarithm. However, such asymptotic improvement may come at a substantial finite-sample cost. To illustrate this point, we compare our calibrated boundary with the box boundary of \citet{Jourdan}, which has $\mathcal O(\log\log t)$ asymptotic dependence and performs best empirically among the boundaries studied there. Figure~\ref{fig-thresholds} plots both boundaries as functions of $\log\log t$ for $\alpha=0.5,0.05,0.005$. Across all three values of $\alpha$, the box boundary remains substantially larger over a wide practical range of horizons. In particular, when $\alpha=0.05$, it does not fall below our boundary until $t$ exceeds $10^8$. Hence, although the box boundary is asymptotically smaller, our boundary is materially less conservative at finite horizons relevant in practice. Further, we examine the local growth rate of the box boundaries by plotting their empirical slope with respect to $\log\log (t)$. Figure \ref{fig-threshold slope} indicates that the empirical $\mathcal O(\log\log t)$ regime only emerges at extremely large horizons, around $t\approx 8.6 \times 10^7$. For small to moderate ranges of $t$, the box boundaries are much larger than ours. 

\begin{figure}
\centering
\includegraphics[width=\textwidth]{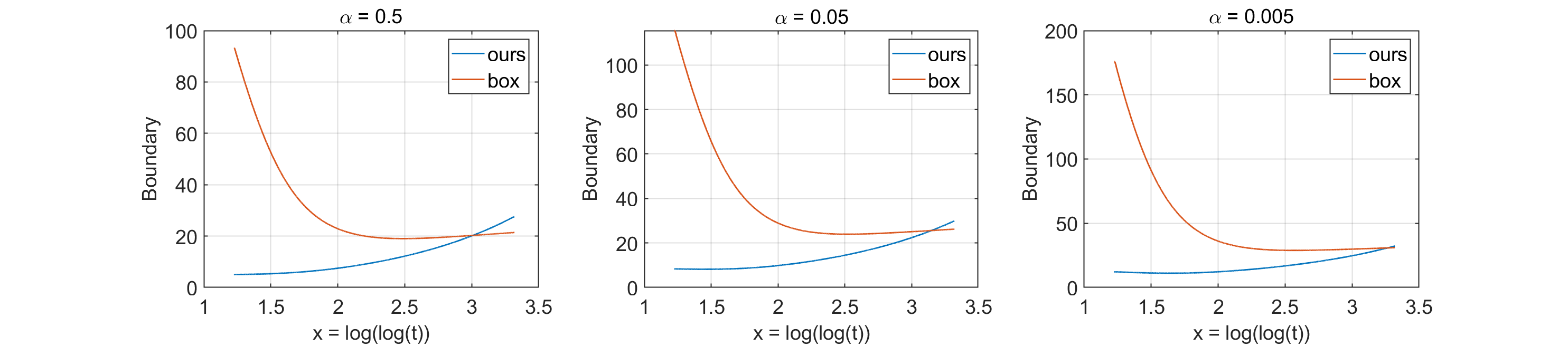}
\caption{Our boundaries and the box boundaries of \citet{Jourdan} as functions of loglog (t).}
\label{fig-thresholds}
\end{figure}

\begin{figure}
\centering
\includegraphics[width=0.34\linewidth]{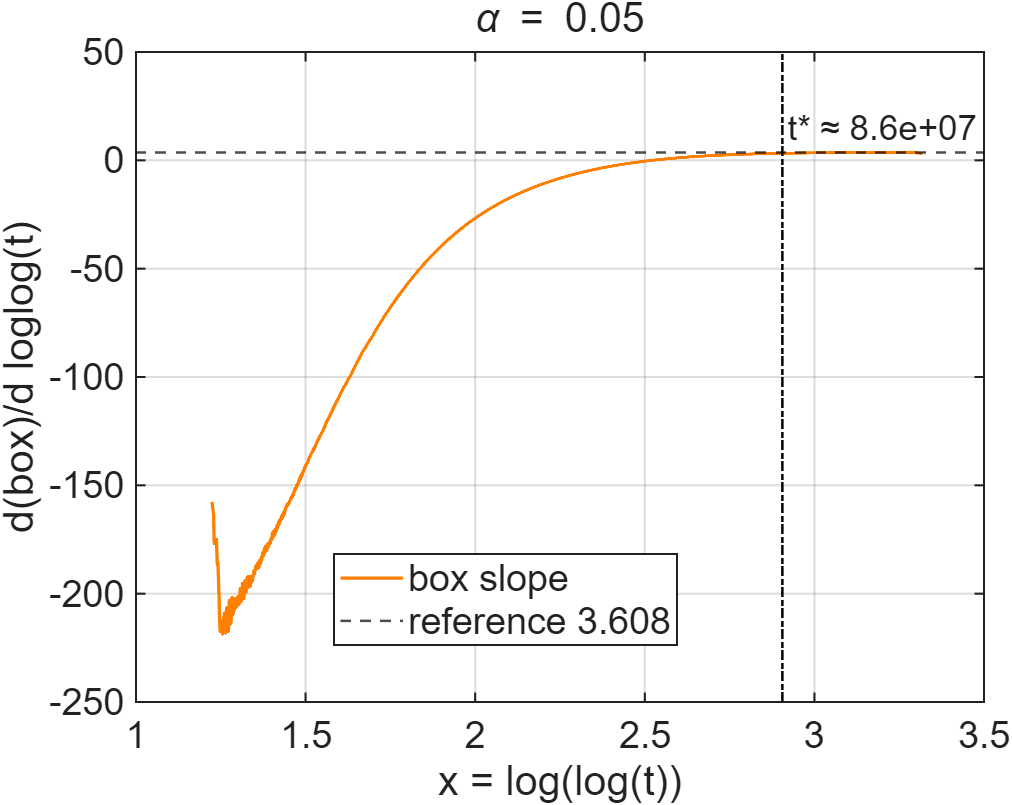}
\caption{Empirical slope of the box boundaries with respect to loglog (t).}
\label{fig-threshold slope}
\end{figure}

\section{Stopping Rules under the Structured Linear Setting}   \label{sec-stopping rules linear}

Under the linear setting, the mean performance of action $a\in\mathcal A(\x)$ under context $\x$ is $y (\x,a) = \fxx^{\text{T}} \bb(a)$, where $\bb(a) = (\beta^{1}(a), ..., \beta^{d}(a))^{\text{T}} \in \mathbb{R} ^{\text{d}}$ is a vector of unknown coefficients that need to be estimated and $\fxx = (\mathrm{f}^1(\x),...,\mathrm{f}^d(\x))^{\T}$ is a vector of known basis functions, which may be chosen to improve model fit. A common choice is to use the raw context itself, that is, $\mathrm{f}^i(\x)=\x^i$ for $i=1,\ldots,d$. The observed outcome is subject to sampling noise: $Y=y(x,a)+\epsilon$, where $\epsilon$ follows a Gaussian distribution with mean zero and variance $\sigma^2(a)$. The noise variances are unknown and may differ across actions. This action-specific linear model is standard in the contextual learning literature \citep{Shen2021, qin2022adaptivity, bastani2022learning}. 

For each action $a$, define
\[
D_t(a):=\sum_{1\le s\le t,\ A_s=a}\fx{s}\fx{s}^\T,
\qquad
\hat{\bb}_t(a):=D_t(a)^{-1}\sum_{1\le s\le t,\ A_s=a}Y_s\fx{s},
\]
whenever $D_t(a)$ is invertible, and let $\hat y_t(\x,a)=\fxx^\T\hat{\bb}_t(a)$. Since $\hat y_t(\x,a)$ is no longer a sample mean, the guarantees in Section~\ref{sec-stopping rules general} do not apply directly, and new GLR statistics and boundaries are needed.

\subsection{The GLR Statistics}   \label{sec-GLR statistics linear}
Let $\underline{Y_t}(a):=\{(Y_s,\fx{s}) : 1\le s\le t,\ A_s=a\}$
denote the observations and associated contexts collected from action $a$ up to stage $t$. Let $\mathcal{L}_{\bb(a)}^{o}\!\left(\underline{Y_t}(a)\right)$ denote the likelihood of these observations under the parameter $\bb(a)\in\mathbb{R}^d$. Under the Gaussian noise model, the likelihood function is
\[
\mathcal{L}_{\bb(a)}^{o}\!\left(\underline{Y_t}(a)\right)
=
\prod_{1\le s\le t,\ A_s=a}
\left[
\frac{1}{\sqrt{2\pi}\sigma(a)}
\exp\!\left(
-\frac{(Y_s-\fx{s}^{\T}\bb(a))^2}{2\sigma^2(a)}
\right)
\right].
\]

Using this likelihood, we define the GLR statistic for comparing any pair of actions $a,a'\in\mathcal{A}$ under context $\x\in\mathcal{X}$ as
\begin{equation}   \label{equa-GLR statistic linear}
    Z_{a,a'}^{L} (\x,t,\delta) 
    = 
    \log 
    \frac{
    \displaystyle
    \max_{\fxx^{\T} \bb(a) \geq \fxx^{\T} \bb(a') - \delta} 
    \mathcal{L}_{\bb(a)}^o \left(\underline{Y_t}(a)\right) 
    \mathcal{L}_{\bb(a')}^o \left(\underline{Y_t}(a')\right)
    }{
    \displaystyle
    \max_{\fxx^{\T} \bb(a) \leq \fxx^{\T} \bb(a') - \delta} 
    \mathcal{L}_{\bb(a)}^o \left(\underline{Y_t}(a)\right) 
    \mathcal{L}_{\bb(a')}^o \left(\underline{Y_t}(a')\right)}.
\end{equation}

It is worth noting that the Gaussian likelihood for action $a$ depends on both the regression coefficient vector $\bb(a)$ and the unknown variance $\sigma^2(a)$. For any fixed value of $\sigma^2(a)$, maximizing the likelihood with respect to $\bb(a)$ is equivalent to minimizing the residual sum of squares, and hence yields the ordinary least squares (OLS) estimator $\hat{\bb}_t(a)$. This naturally motivates a likelihood-ratio certification between two actions through the fitted values $\fxx^{\T}\hat{\bb}_t(a)$ and $\fxx^{\T}\hat{\bb}_t(a')$.

Let $\Sigma_{t}(\x,a) = \fxx^{\T} D_{t}^{-1}(a) \fxx$ for all $\x\in\mathcal{X}$ and $a\in\mathcal{A}(\x)$. If the variances $\sigma^2(a)$ were known, the constrained likelihood-ratio statistic in \eqref{equa-GLR statistic linear} would reduce to the quadratic form given in Lemma~\ref{lemma-statistics linear model}. 

\begin{lemma}   \label{lemma-statistics linear model}
    Let $t\in\mathbb{N}^*$, and assume that for all $a\in\mathcal{A}$ the matrix $D_{t}(a)$ is positive definite. For all $\x\in\mathcal{X}$ and $a,a'\in\mathcal{A}(\x)$ satisfying $\fxx^{\T}\hat{\bb}_{t}(a) \geq \fxx^{\T}\hat{\bb}_{t}(a') - \delta$, we have
    \begin{equation*}
        Z_{a,a'}^{L} (\x,t,\delta) = \frac{\left(\fxx^{\T} \hat{\bb}_{t}(a) - \fxx^{\T} \hat{\bb}_{t}(a') + \delta \right)^2} {2\left(\sigma^2(a) \Sigma_{t}(\x,a) + \sigma^2(a') \Sigma_{t}(\x,a') \right)}.
    \end{equation*}
    Moreover, $Z_{a',a}^{L} (\x,t,\delta) = - Z_{a,a'}^{L} (\x,t,\delta)$.
\end{lemma}

In practice, however, the variances are unknown. We therefore adopt a feasible version of the statistic by replacing $\sigma^2(a)$ with the residual variance estimator $S_t^2(a)$, leading to
\begin{equation}
\label{equa-GLR statistic linear feasible}
\tilde Z_{a,a'}^{L}(\x,t,\delta)
=
\frac{\big(\fxx^\T\hat{\bb}_t(a)-\fxx^\T\hat{\bb}_t(a')+\delta\big)^2}
{2\big(S_t^2(a)\Sigma_t(\x,a)+S_t^2(a')\Sigma_t(\x,a')\big)}.
\end{equation}
Our stopping rules and theoretical guarantees are formulated directly in terms of \eqref{equa-GLR statistic linear feasible}.

Let $t_0:=\inf\{t\in\mathbb N^*: D_t(a)\succ0,\ \forall a\in\mathcal A\}$ so that the OLS estimators are well defined for all actions. Then the stopping rule for $\Pre1$ under the structured linear setting is defined as
\begin{equation}
\label{equa-stopping rule pre1 linear}
\tau_{\alpha,\delta}^{\one,L}
=
\inf\Big\{
t\ge t_0:
\forall\,\x\in\mathcal X,\ 
\forall\,a\in\mathcal A(\x)\setminus\{\hat\pi_t(\x)\},\
\tilde Z_{\hat\pi_t(\x),a}^{L}(\x,t,\delta)
>
\varphi_{\hat\pi_t(\x),a}^{\one,L}(\bm N_t,\alpha,\x)
\Big\}.
\end{equation}

Similarly, for $\Pre2$, the stopping rule under the structured linear setting is defined as
\begin{equation}   \label{equa-stopping rule pre2 linear}
    \tau_{\alpha,\delta}^{\two,L} = \inf\bigg\{ t\ge t_0: \sum_{\x\in\mathcal X}p(\x)\,r^L(\x,t)\le \delta \bigg\},
\end{equation}
where $r^L(\x,t) := \max_{a\in\mathcal A(\x)\setminus\{\hat{\pi}_t(\x)\}}w_{\hat{\pi}_t(\x),a}^L(\x,t)$, and, for every context $\x$ and actions $a,a'\in\mathcal{A}(\x)$,
$$
w_{a,a'}^L(\x,t)
:=
\left[
\sqrt{
2\,\varphi_{a,a'}^{\two,L}(\bm N_t,\alpha,\x)
\left(
S_t^2(a)\,\Sigma_{t}(\x,a)
+
S_t^2(a')\,\Sigma_{t}(\x,a')
\right)
}
-
\Big(\fxx^{\T} \hat{\bb}_{t}(a) - \fxx^{\T} \hat{\bb}_{t}(a')\Big)
\right]_+.$$

\subsection{Calibration of GLR Boundaries}   \label{sec-stopping thresholds linear}

The calibration parallels Section~\ref{sec-stopping thresholds general}, but the relevant deviation now involves OLS estimators rather than sample means. For $\x\in\mathcal X$ and $a\in\mathcal A(\x)\setminus\{\pi^*(\x)\}$, define
\begin{equation}   
\label{equa-summed deviation term linear}
    U_t^L(\x;a,\pi^*(\x))
    :=
    \sum_{c\in\{a,\pi^*(\x)\}}
    \frac{\big[\fxx^\T\big(\hat{\bb}_t(c)-\bb(c)\big)\big]^2}
    {2S_t^2(c)\Sigma_t(\x,c)}.
\end{equation}
Existing bounds in linear bandits typically control $\|\hat{\bb}_t(c)-\bb(c)\|_{D_t(c)}^2$ under known variances, and do not directly yield tight bounds for the directional, variance-estimated quantity in \eqref{equa-summed deviation term linear}. We therefore calibrate the boundaries by controlling \eqref{equa-summed deviation term linear} directly.

\begin{lemma}   \label{lemma-deviation inequality linear}
    Let $r\in\{1,2\}$ index two sample streams to estimate linear models with unknown parameters $\bb_r\in\mathbb{R}^d$ and let $N_{t,r}, D_{t,r}, \hat{\bb}_{t,r}, S_{t,r}^2$ denote the corresponding sample size, design matrix, OLS estimator, and sample variance of the noise at stage $t$. For an arbitrary vector $\bm{f}\in\mathbb{R}^d$, define $\Sigma_{t,r} = \bm{f}^{\T} D_{t,r}^{-1} \bm{f}$ and $U_t^{L,+} = \sum_{r=1}^{2} \frac{\left(\bm{f}^{\T} \hat{\bb}_{t,r} - \bm{f}^{\T} \bb_r \right)^2}{2S_{t,r}^2 \Sigma_{t,r}}$. Then, with probability greater than $1-\alpha$, for all $t\in\mathbb{N}^*$,
    \begin{equation}   \label{equa-deviation inequality linear}
        U_t^{L,+} \leq \max \left\{\frac{1}{2} \gamma^{L} \left(N_{t,1}, \Sigma_{t,1}^{-1}, \alpha \sqrt{\frac{1}{\Sigma_{t,2}^{-1}+1}}\right), 
        \frac{1}{2} \gamma^{L} \left(N_{t,2}, \Sigma_{t,2}^{-1}, \alpha \sqrt{\frac{1}{\Sigma_{t,1}^{-1}+1}}\right) \right\},
    \end{equation}
    where $\epsilon > 0$ can be arbitrarily small and the function $\gamma^{L}:\mathbb{N}^* \times \mathbb{R}^* \times (0,1) \to (0,+\infty)$ is defined as 
    \begin{equation}   \label{equa-gamma-L} 
      \gamma ^{L}(t_1, t_2,\alpha) = \frac{(t_1 - d)t_2}{\max \left\{\left(\frac{\alpha^2}{t_2+1} \right)^{\frac{1}{t_1-d+1}} (t_2+1) - 1, \epsilon \right\} } - (t_1 - d).
    \end{equation}
\end{lemma}

Lemma~\ref{lemma-deviation inequality linear} is derived through a new martingale construction tailored to the linear directional deviation $\bm{f}^{\T} \hat{\bb}_{t,r} - \bm{f}^{\T} \bb_r$, which requires decomposing the linear model into a scalar projection along $\bm{f}$ and an orthogonal complement. Then we marginalize out the nuisance parameters associated with the orthogonal subspace using a Gaussian prior and obtain a martingale that depends only on the directional projection. The bound in \eqref{equa-deviation inequality linear} then follows by combining two such martingales via Ville’s maximal inequality.

Using the same error budget allocation as in Section \ref{sec-stopping rules general}, we obtain the following result.
\begin{theorem}   
\label{theorem2}
    Let $r\in\{1,2\}$ index the two precision notions $\one$ and $\two$. Under the linear setting, for each $r\in\{1,2\}$, each context $\x\in\mathcal X$, and each pair of actions $a,a'\in\mathcal A(\x)$, let
    \begin{equation}
    \label{equa-threshold-unified-linear}
      \begin{aligned}
        \varphi_{a,a'}^{(r),L}(\bm N_t,\alpha,\x)
        =
        \max\Bigg\{
        &\frac{1}{2}\gamma^L\!\left(
        N_t(a),\Sigma_t^{-1}(\x,a),
        \alpha^{(r)}(\x)\sqrt{\frac{1}{\Sigma_t^{-1}(\x,a')+1}}
        \right),\\
        &\frac{1}{2}\gamma^L\!\left(
        N_t(a'),\Sigma_t^{-1}(\x,a'),
        \alpha^{(r)}(\x)\sqrt{\frac{1}{\Sigma_t^{-1}(\x,a)+1}}
        \right)
        \Bigg\}.
      \end{aligned}
    \end{equation}
    Then the stopping rule $\tau_{\alpha,\delta}^{(r),L}$ satisfies the corresponding target guarantee: when $r=1$, $\Pre1\ge 1-\alpha$; when $r=2$, $\Pre2\ge 1-\alpha$.
\end{theorem}

Similar to the unstructured setting, define $\rho^L (t_1,t_2,\alpha) = \left(\frac{\alpha^2}{t_2+1} \right)^{\frac{1}{t_1-d+1}} (t_2+1) - 1$, and then the boundary function $\gamma^L (t_1,t_2,\alpha)$ becomes nontrivial only when $\rho^L (t_1,t_2,\alpha) > 0$. Given the condition that $t_1$ and $t_2$ grow on the same order, the length of this initial inactive stage under the structured linear setting is also on the order of $\frac{2\log(1/\alpha)}{\log\log(1/\alpha)}$. The following proposition further characterizes the asymptotic behavior of the boundary function $\gamma^L (t_1,t_2,\alpha)$ after the initial stage.
 
\begin{proposition}\label{prop-gamma-L-asymptotic}
Fix $\alpha\in(0,1)$, and let $\gamma_0^L (t_1,t_2,\alpha)$ denote the version of $\gamma^L$ without the $\epsilon$ truncation. Suppose that $t_1,t_2\to\infty$ and that there exist constants $0<r<R<\infty$ such that $r \leq t_1 / t_2 \leq R$. Then,
\begin{equation}
\label{eq-gamma-L-asymptotic}
\gamma_0^L(t_1,t_2,\alpha)
=
2\log(1/\alpha) + \log(t_2+1) + o(1).
\end{equation}
\end{proposition}

Proposition~\ref{prop-gamma-L-asymptotic} shows that, after the initial stage, the boundary decomposes into a precision term $2\log(1/\alpha)$ and a time-uniformity term $\log(t_2+1)$. In Theorem~\ref{theorem2}, $t_2$ is instantiated as $\Sigma_t^{-1}(\x,a)$, so the calibrated boundary grows with the accumulated directional information relevant to comparing actions under context $\x$, rather than merely the raw sample number.

\subsection{Expected Sample Sizes}   \label{sec-sample complexity}

The expected sample size of a stopping rule is strongly influenced by the sampling strategy it is paired with. This is because different sampling strategies govern how quickly information accumulates across actions, and hence can substantially accelerate or delay the time at which the stopping condition is met. In this section, we focus on combining our stopping rule with the equal allocation sampling strategy, which allocates samples uniformly across all actions. Equal allocation is a particularly simple but inefficient strategy and is therefore commonly used as a baseline. We prove that, even under this naive and typically poor-performing sampling strategy, our stopping rule $\tau_{\alpha,\delta}^{L}$ achieves a smaller expected sample size than the state-of-the-art \textit{two-stage procedure (TS)} \citep{Shen2021}. Notably, \textit{TS} is a recently proposed and influential method in the contextual R\&S literature. We begin by introducing the following technical assumptions.

\begin{assumption}
\label{ass:linear-design}
Let $P^s$ denote the distribution of the sampled context at each stage, and define
\[
\Sigma := \mathbb{E}_{\X\sim P^s}\!\left[\mathbf{f(\X)} \mathbf{f(\X)}^{\T}\right].
\]
Assume that $\Sigma$ is positive definite and that $\fxx\neq 0$ for all $\x\in\mathcal{X}$.
\end{assumption}

\begin{assumption}
\label{ass:bounded context}
There exist constants $0<L\le U<\infty$ such that 
$$
L \;\leq\; \|\fxx\|_2^2 \;\leq\; U, \qquad \forall\; \x\in\mathcal{X}.
$$
\end{assumption}

Assumption~\ref{ass:linear-design} requires that the context distribution provide sufficient excitation in all directions of the feature space, so that the linear models are identifiable and the design matrices accumulate information at a linear rate. In the real system environment, this condition is determined by the context distribution induced by the underlying environment. In a simulation environment, it depends on the sampling strategy used to generate contexts. This is a standard type of nondegeneracy condition in linear regression and contextual learning problems. Assumption~\ref{ass:bounded context} imposes bounded context vectors, which is a common assumption in the contextual learning literature \citep{li2010contextual}.

Under the equal allocation, the sample size allocated for each action at stage $t$ is $N_t(a) = t/k$ for all $a\in\mathcal{A}$. Let $T_{\one} := \mathbb{E}[\tau_{\alpha,\delta}^{\one,L}]$ and
$T_{\two} := \mathbb{E}[\tau_{\alpha,\delta}^{\two,L}]$ denote the expected sample sizes of the two proposed rules under the equal allocation strategy.

\begin{theorem}   \label{theorem3}
Suppose Assumptions 3 and 4 hold, $\delta>0$ and $\mathcal{A}(\x) = \mathcal{A}, \forall\,\x\in\mathcal{X}$. Then, as $k\to\infty$, 
$$
T_{\one} = \mathcal{O} (k\log k)
\quad\text{and}\quad
T_{\two} = \mathcal{O} (k\log k).
$$
Moreover, as $\alpha\to 0$,
$$
T_{\one} = \mathcal{O} \big(\log(1/\alpha)\big)
\quad\text{and}\quad
T_{\two} = \mathcal{O} \big(\log(1/\alpha)\big).
$$
\end{theorem}

Theorem~\ref{theorem3} shows that, as the number of actions $k \to \infty$, the expected sample sizes of $\tau_{\alpha,\delta}^{\one,L}$ and $\tau_{\alpha,\delta}^{\two,L}$ both scale as $\mathcal{O}(k\log k)$, whereas \textit{TS} scales as $\mathcal{O}\left(k^{1+\frac{2}{n_0-1}}\right)$. 
In addition, as the target precision level $1-\alpha \to 1$ (i.e., $\alpha \to 0$), our stopping rules achieve a sample complexity of order $\mathcal{O}(\log(1/\alpha))$, while \textit{TS} requires $\mathcal{O}\left(\alpha^{-\frac{2}{n_0 q - d}}\right)$. Here, $n_0$ denotes the number of samples allocated to each context-action pair in the first stage of \textit{TS}, and $q$ is the number of design points. This performance gap is driven by the different deviation bounds underlying the two approaches. Our stopping rules rely on time-uniform deviation inequalities, which yield a logarithmic dependence on the allocated error probability $\alpha/(k-1)$ for each action. In contrast, \textit{TS} controls the first-stage evidence using a fixed-sample deviation bound, and the overall sample size is dictated by these initial estimates, leading to a polynomial dependence on $\alpha/(k-1)$. Consequently, our method achieves strictly better scaling in expected sample size, both as $k \to \infty$ and as $\alpha \to 0$.

\begin{remark}
    Theorem \ref{theorem3} shows that the stopping time obtained by our stopping rules grows on the order of $\log(1/\alpha)$ as $\alpha \to 0$. In contrast, as established earlier, the length of the initial inactive stage induced by boundaries grows only on the order of $\frac{\log(1/\alpha)}{\log\log(1/\alpha)}$. Therefore, the influence of this initial inactive stage is asymptotically negligible relative to the overall stopping time.
\end{remark}
 
\subsection{Computational Complexity}
\label{sec-computational-complexity}

In this subsection, we discuss the computational complexity of the proposed stopping rules under both the unstructured and structured linear settings.

\paragraph{Unstructured setting.}
In the unstructured formulation, the stopping rule is constructed from pairwise GLR statistics across context--action pairs. At each stage $t$, for each context $\x \in \mathcal{X}$, the rule compares the currently estimated best action $\hat{a}_t(\x)$ with all alternative actions $a \in \mathcal{A}(\x)\setminus\{\hat{a}_t(\x)\}$. As a result, the number of pairwise comparisons required at stage $t$ is on the order of
$\sum_{\x \in \mathcal{X}} \bigl(|\mathcal{A}(\x)| - 1\bigr)$,
which reduces to $O(mk)$ when each of the $m$ contexts admits $k$ feasible actions.
Each pairwise GLR statistic in~\eqref{equa-local-glr} admits a closed-form expression and can be computed in constant time using sufficient statistics (sample means, variances, and counts). Thus, the per-stage computational cost of evaluating the stopping condition scales linearly with the number of context-action pairs.
While this is significantly more efficient than exhaustive policy-level comparisons (which would scale with $|\Pi|$, typically exponential in $|\mathcal{X}|$), the cost can still become substantial when the number of contexts is large. In particular, when $|\mathcal{X}|$ grows combinatorially with underlying features, context-wise certification may become computationally burdensome.

\paragraph{Structured linear setting.}
In the structured linear formulation, the mean reward is modeled using action-specific linear models, which allows information to be pooled across contexts. The GLR statistics are constructed based on lower-dimensional parameter estimates.
Let $d$ denote the feature dimension. The main computational cost at each stage arises from updating the least-squares estimates and the associated covariance matrices for each action. Using standard recursive updates, these operations require $O(d^2)$ time per observation, or $O(kd^2)$ per stage when accounting for all actions. The evaluation of the GLR statistics and stopping boundaries then depends only on these parameter estimates and can be performed with negligible additional cost relative to the estimation step.
Consequently, the overall per-stage complexity of the structured approach scales polynomially in the feature dimension $d$ and the number of actions $k$, but is independent of the number of contexts. This makes the structured formulation significantly more scalable in settings where the context space is large or high-dimensional.

In both settings, the use of pairwise GLR statistics and corresponding boundaries is critical for tractability and require no equation solving. By avoiding exhaustive policy-level comparisons, the proposed stopping rules remain implementable even when the policy space is large.

\section{Numerical Experiments}   \label{sec-numerical experiments}
In this section, we conduct numerical experiments to test the performance of our proposed stopping rules. We combine the stopping rules with available sampling strategies to evaluate the sample sizes required to identify the optimal policies with guaranteed $\Pre{1}$ or $\Pre{2}$.

The sampling strategies that will be implemented in our experiments are summarized as follows:
\begin{itemize}
    \item \textit{Contextual optimal computing budget allocation (\textit{C-OCBA}, \cite{gao2019selecting}).} \textit{C-OCBA} is an adaptive sample strategy for contextual R\&S under the unstructured setting, which allocates samples across context--action pairs to asymptotically achieve the optimal ratios for $\Pre{1}$.
    \item \textit{Contextual track and stop design (\textit{CTSD}, \cite{Simchi2024}).} \textit{CTSD} extends Track-and-Stop \citep{Garivier2016Optimal} to the contextual learning setting. We use its sampling component (with forced exploration), referred to as \textit{CTD}, and combine it with our stopping rule.
    \item \textit{Contextual optimal computing budget allocation for linear structure (\textit{C-OCBA-L}, \cite{Du2024}).} \textit{C-OCBA-L} is a sample strategy for the structured linear setting. It targets asymptotically optimal allocation ratios for $\Pre{1}$ across actions and a set of fixed context design points.
    \item \textit{EA} allocates samples uniformly across context--action pairs and serves as a simple benchmark.
\end{itemize}
 We compare the sample sizes required to attain the target precision guarantee with the following methods:
\begin{itemize}
    \item \textit{Stopping rule for unknown variances with box boundary (\textit{JDK}, \cite{Jourdan}).} \cite{Jourdan} propose several theoretically grounded stopping rules for BAI with unknown variances. These rules can be combined with arbitrary sampling schemes and extended to our unstructured setting. We use their EV-GLR stopping rule with box boundary, denoted by \textit{JDK}, which showed the best empirical performance in their study.
    
    \item \textit{KN procedure (\textit{KN}, \cite{keslin2025ranking}).} \cite{keslin2025ranking} decompose the contextual R\&S problem into a collection of independent R\&S subproblems, each solved using the classical \textit{KN} procedure \citep{kim2001fully}. \textit{KN} is a fully sequential selection procedure designed to achieve a prespecified probability of correct selection (PCS) within an indifference-zone under each context. We compare with \textit{KN} under the unstructured setting.
    
    \item \textit{Two-stage procedure (\textit{TS}, \cite{Shen2021}).} \textit{TS} is a two-stage procedure for contextual R\&S under the structured linear setting. It relies on an indifference-zone formulation and a fixed set of context design points. In the first stage, a small initial sample is allocated to each context--action pair to estimate the variances. In the second stage, the remaining sample sizes are determined based on these variance estimates, and the best action for each context is selected according to the sample means. \textit{TS} is designed to guarantee $\Pre{1}$ and can be adapted to $\Pre{2}$. Since it is the only existing method for the structured linear setting, we compare it extensively with our stopping rule.
\end{itemize}

\subsection{Synthetic Data}   \label{sec-systhetic data}

\subsubsection{Unstructured Setting}   
Synthetic data in the unstructured setting are generated according to three benchmark functions, which will be presented in Appendix C. We compare the sample sizes required for stopping under the following methods: \textit{(1) $\tau_{\alpha,\delta}$ with C-OCBA}, \textit{(2) JDK with C-OCBA} and \textit{(3) KN}. For both \textit{JDK} and \textit{KN}, we allocate the precision level for each context to guarantee $\Pre{1}$ or $\Pre{2}$ analogous to our stopping rule, as described in Section \ref{sec-stopping thresholds general}. Let the initial samples allocated for each action-context pair to be $n_0 = 20$ and the precision level $1 - \alpha = 0.95$. The tolerance difference is set as $\delta = 0.1$. We conducted 1000 macro-replications to calculate the averaged sample size (Avg. SSize) and the sample standard deviation (Std) for each case. The results are summarized in Table 1. All methods achieve the empirical precision levels evaluated through 1000 replications.

\begin{table}[h]   
\caption{Comparison of averaged sample sizes on benchmark functions under unstructured setting}
\label{tab-synthetic unstructured}
\centering
{
\fontsize{10pt}{12pt}\selectfont
\begin{tabular}{@{}ccclcclcc@{}}
\toprule
Target $\Pre{1}\geq 0.95$ & \multicolumn{2}{c}{\textit{$\tau_{\alpha,\delta}^{\one}$ with C-OCBA}} &  & \multicolumn{2}{c}{\textit{JDK with C-OCBA}} &  & \multicolumn{2}{c}{\textit{KN}} \\ \cmidrule(lr){2-3} \cmidrule(lr){5-6} \cmidrule(l){8-9} 
Case   & Avg. SSize            & (Std)            &  & Avg. SSize     & (Std)    &  & Avg. SSize  & (Std)  \\ \midrule 
1      & 6870.55              & (655.88)                 &  & 67427.26       & (4457.88)        &  & 10212.40     & (739.80)      \\
2      & 1019.73               & (135.78)                &  & 11107.53      & (1306.83)       &  & 1364.35    & (182.07)      \\
3      & 16722.46              & (2709.81)                &  & 149371.98         & (9183.75)       &  & 18064.57   & (1145.63)      \\ \midrule
Target $\Pre{2}\geq 0.95$ & \multicolumn{2}{c}{\textit{$\tau_{\alpha,\delta}^{\two}$ with C-OCBA}} &  & \multicolumn{2}{c}{\textit{JDK with C-OCBA}} &  & \multicolumn{2}{c}{\textit{KN}} \\ \cmidrule(lr){2-3} \cmidrule(lr){5-6} \cmidrule(l){8-9} 
Case   & Avg. SSize            & (Std)            &  & Avg. SSize     & (Std)    &  & Avg. SSize  & (Std)  \\ \midrule
1      & 10338.95               & (860.22)                 &  & 108845.11      & (7185.18)        &  & 17538.32    & (1249.05)      \\
2      & 980.81              & (120.14)                &  & 13461.67       & (1511.02)       &  & 1879.42    & (265.74)      \\
3      & 20278.89              & (2303.68)                &  & 225527.50      & (15030.70)      &  & 37295.35   & (2238.20)      \\ \bottomrule
\end{tabular}
}
\end{table}

Table \ref{tab-synthetic unstructured} highlights the superior performance of our proposed stopping rules, $\tau_{\alpha,\delta}^{\one}$ and $\tau_{\alpha,\delta}^{\two}$. When combined with an efficient sampling strategy, our stopping rules require substantially fewer samples to identify the correct policy with precision guarantee compared with \textit{JDK} and \textit{KN}. While \textit{JDK} shares the same form with our stopping rules, the calibration of their boundaries differs, resulting in greater conservativeness, as mentioned in Section \ref{sec-stopping thresholds general}. For comparison with \textit{KN}, the efficiency gain arises because the computation of our stopping rules is independent of the sampling strategy, which allows us to fully utilize the efficient sample allocation (if any) of the sampling strategy to compare different actions more effectively. 

\subsubsection{Linear Setting}   \label{sec-synthetic linear}
We compare the performance of the following methods on synthetic data under the structured linear setting: \textit{(1) $\tau_{\alpha,\delta}^{L}$ with C-OCBA-L}, \textit{(2) $\tau_{\alpha,\delta}$ with EA} and \textit{(3) TS}. Since \textit{TS} is designed to guarantee $\Pre{1}$, for $\Pre{2}$, we calculate the value of $h$ in TS as
\begin{equation*}
    \min_{\x\in\mathcal{X}} \left\{\int_{0}^{\infty} \left[\int_{0}^{\infty} \Phi\left(\frac{h}{\sqrt{(n_0 m - d)(t^{-1} + s^{-1})\fxx^{\T}(\mathbf{F}^{\T} \mathbf{F})\fxx } } \right) \right] \right\} = 1 - \frac{\alpha}{m},
\end{equation*}
where $\mathbf{F}$ is the design matrix of contexts and $\eta(\cdot)$ is the
probability density function of the chi-square distribution with $(n_0 m - d)$ degrees of freedom. 

To make a comprehensive comparison between our stopping rules and \textit{TS}, we first test them on a standard case. Let the dimension of contexts be $d=3$. Suppose that $\mathbf{X} = (1, X_2, X_3)^{\T}$ and $X_2, X_3$ are i.i.d. random variables uniformly distributed over $\{0,0.2,0.4,0.6,0.8,1\}$. There are total $m=36$ contexts. We set each except the first entry of a context design point to be 0 or 1, so there are $p = 2^{d-1}$ design points in total. For each action $a^i,i \in\{1,...,k\}$, we set $\beta_0(a^i) = 0.5(i-1)$ and $\beta_1(a^i) = \beta_2(a^i) = 1 + 0.5(i-1)$. Let the sampling variance $\sigma^2(a^i) = 1, i\in\{1,...,k\}$, $\delta = 0.5$ and $n_0 = 10$. We evaluate the averaged sample size required by our stopping rules and by \textit{TS} under different numbers of actions $k$ and precision levels $\alpha$. Each case is conducted using 1000 replications and the results are shown in Table 2.
\begin{table}[h]   
\caption{Comparison of averaged sample sizes on standard cases under structured linear setting}
\label{tab-synthetic linear}
\centering
{
\fontsize{10pt}{12pt}\selectfont
\begin{tabular}{@{}cccccccccc@{}}
\toprule
\multicolumn{2}{c}{Target $\Pre{1}\geq 1-\alpha$}      & \multicolumn{2}{c}{\textit{$\tau_{\alpha,\delta}^{\one,L}$ with C-OCBA-L}} &  & \multicolumn{2}{c}{\textit{$\tau_{\alpha,\delta}^{\one,L}$ with EA}}      &  & \multicolumn{2}{c}{\textit{TS}}      \\ \cmidrule(lr){3-4} \cmidrule(lr){6-7} \cmidrule(l){9-10} 
$\alpha$                  & $k$  & Avg. SSize    & (Std)        &  & Avg. SSize & (Std)       &  & Avg. SSize & (Std)       \\ \midrule
\multirow[t]{3}{*}{0.05}  & 10 & 504.33             & (95.45)   &  & 1199.48         & ( 519.73)  &  & 1001.85         & (73.52)   \\
                       & 20 & 935.56            & (126.11)   &  & 2522.88          & (1009.79) &  & 2489.38         & (128.86)  \\
                       & 50 & 2181.87            & (164.10)   &  & 6937.20           & (2718.35) &  & 7856.34         & (251.58)  \\
\multirow[t]{3}{*}{0.01}  & 10 & 554.87             & (123.46)   &  & 1410.88          & (541.76)  &  & 1561.82         & (113.98)  \\
                       & 20 & 1007.01            & (164.70)   &  & 2990.16         & (1117.90) &  & 3689.17         & (191.77)  \\
                       & 50 & 2252.49            & (197.04)   &  & 8216.40         & (3107.83) &  & 11124.83        & (367.05)  \\
\multirow[t]{3}{*}{0.001} & 10 & 701.32            & (228.82)   &  & 2989.68         & (852.28)  &  & 2490.09         & (181.75)  \\
                       & 20 & 1161.44            & (249.96)   &  & 3592.08         & (1232.99) &  & 5634.03         & (298.19)  \\
                       & 50 & 2435.58            & (273.19)   &  & 9601.80         & (3286.39) &  & 16243.63        & (533.36)  \\ \midrule
\multicolumn{2}{c}{Target $\Pre{2}\geq 1-\alpha$}      & \multicolumn{2}{c}{\textit{$\tau_{\alpha,\delta}^{\two,L}$ with C-OCBA-L}} &  & \multicolumn{2}{c}{\textit{$\tau_{\alpha,\delta}^{\two,L}$ with EA}}      &  & \multicolumn{2}{c}{\textit{TS}}      \\ \cmidrule(lr){3-4} \cmidrule(lr){6-7} \cmidrule(l){9-10} 
$\alpha$                  & $k$  & Avg. SSize    & (Std)        &  & Avg. SSize & (Std)       &  & Avg. SSize & (Std)       \\ \midrule 
\multirow[t]{3}{*}{0.05}  & 10 & 429.88             & (26.13)   &  & 551.16         & (104.93)  &  & 3488.22         & (251.87)  \\
                       & 20 & 837.04            & (30.15)   &  & 1154.80         & (240.36) &  & 7782.83         & (377.29)  \\
                       & 50 & 2045.43            & (30.63)   &  & 3065.20         & (635.92) &  & 22122.33        & (729.17)  \\
\multirow[t]{3}{*}{0.01}  & 10 & 444.89            & (30.28)   &  & 612.08          & (127.15)  &  & 4379.18         & (324.55)  \\
                       & 20 & 854.74            & (32.50)   &  & 1302.64         & (268.74) &  & 9595.27         & (493.12)  \\
                       & 50 & 2064.57            & (33.16)   &  & 3446.00         & (715.37) &  & 26919.92        & (878.88)  \\
\multirow[t]{3}{*}{0.001} & 10 & 471.49            & (36.93)   &  & 721.20         & (144.02)  &  & 5775.23         & (429.07)  \\
                       & 20 & 882.34            & (40.15)   &  & 1497.60         & (291.46) &  & 12461.45        & (641.60)  \\
                       & 50 & 2093.63            & (41.95)   &  & 3921.20         & (761.67) &  & 34301.37        & (1167.45) \\ \bottomrule
\end{tabular}%
}
\end{table}

From Table 2, we observe that \textit{$\tau_{\alpha,\delta}^{L}$ with C-OCBA-L} consistently outperforms \textit{TS} in both scenarios targeting $\Pre{1}$ or $\Pre{2}$. As the number of actions $k$ or the precision level $1-\alpha$ increases, \textit{$\tau_{\alpha,\delta}^{L}$ with C-OCBA-L} tends to save much more samples than \textit{TS} while \textit{$\tau_{\alpha,\delta}^{L}$ with EA} tends to perform comparably to \textit{TS} or even better as $k$ or $1-\alpha$ increases. These results indicate that our stopping rule is less sensitive to $k$ and $\alpha$ in terms of the expected total sample size. In the scenario targeting $\Pre{2}$, \textit{$\tau_{\alpha,\delta}^{L}$ with EA} outperforms \textit{TS}. The reason is that \textit{TS} guarantees $\Pre{2}$ by imposing the stronger requirement that the deviation at every context be bounded by $\delta$. By contrast, our stopping rule $\tau_{\alpha,\delta}^{L}$ guarantees $\Pre{2}$ through an aggregate criterion that pools deviations across contexts. This leads to less conservative stopping and, hence, a smaller sample size, even when our rule is combined with \textit{EA}.

\begin{figure}[h]
    \centering
    \includegraphics[width=0.8\textwidth]{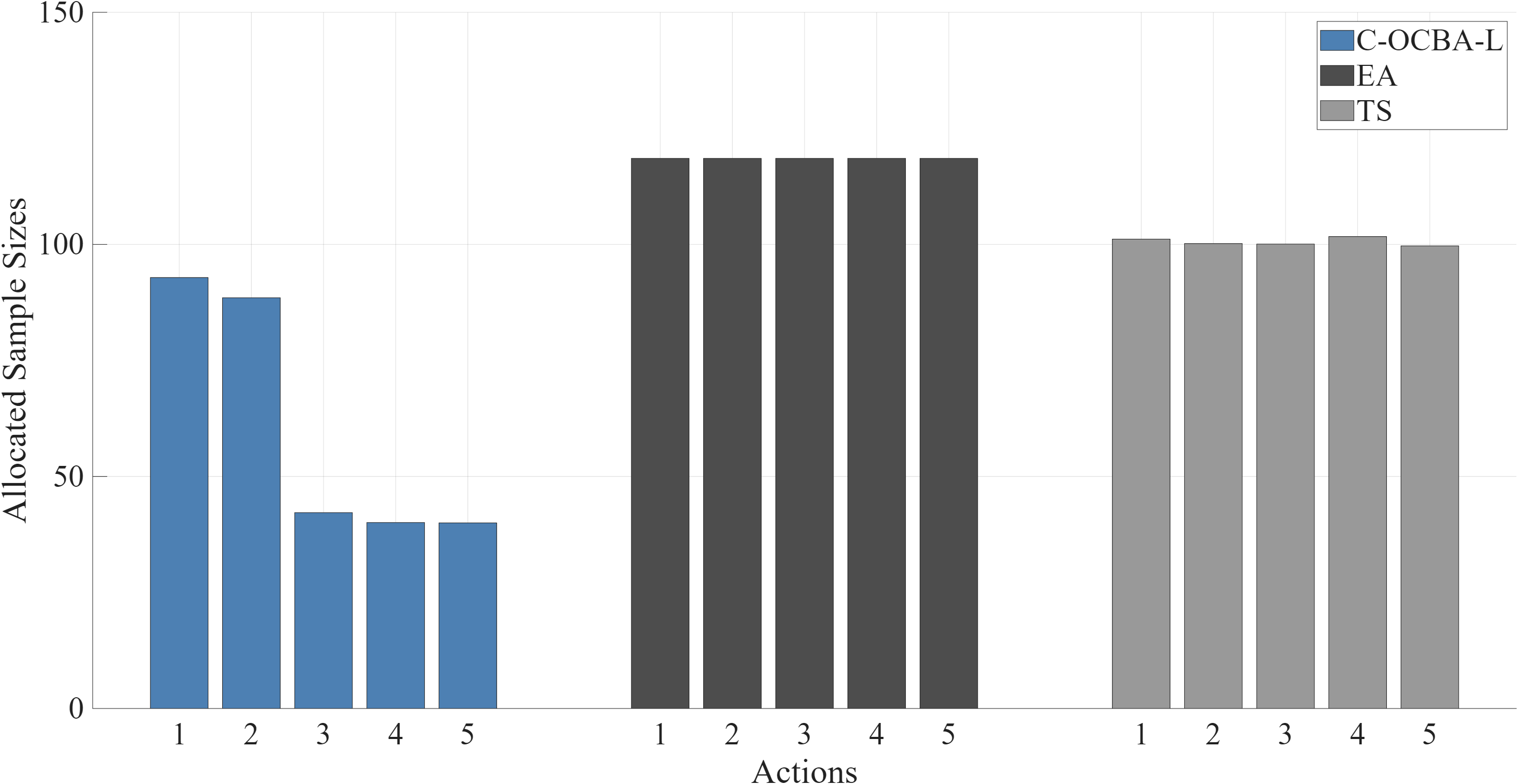}
    \caption{Allocated Sample Sizes for the 1-5th Action in standard case with $\alpha=0.05, k=10$}
    \label{fig-sample allocation}
\end{figure}

The reason for the superior performance of \textit{$\tau_{\alpha,\delta}^{\one,L}$ with C-OCBA-L} is that the sampling strategy \textit{C-OCBA-L} can intelligently allocate the samples across context-action pairs, enhancing the sampling efficiency. In Figure~\ref{fig-sample allocation}, we illustrate the sample allocations for the first five actions under the three compared methods in the standard case with $\alpha=0.05, k=10$. We observe that \textit{$\tau_{\alpha,\delta}^{L}$ with C-OCBA-L} allocates most of the samples to actions 1 and 2, which are the most difficult to distinguish. In contrast, \textit{TS} allocates almost the same number of samples across 5 actions due to its dependence on first-stage sample variance estimates for stopping. This causes the inefficient sample usage. \textit{$\tau_{\alpha,\delta}^{\one,L}$ with EA} shows greater conservativeness than \textit{TS}. This conservativeness arises because our stopping rules track both the uncertainty of sample means and sample variances, whereas \textit{TS} only tracks the uncertainty of sample variances, substituting the uncertainty of sample means with a fixed indifference-zone parameter $\delta$. Note that setting $\delta$ less than the practical smallest difference may lead to severe conservativeness.

Next, we test our stopping rules and \textit{TS} on five randomly generated cases of varying $k$, $d$, $m$ and $p$. The scale settings and details of generated values are listed in Appendix C. Table 3 presents the results, demonstrating the consistently superior performance of our proposed stopping rules when combined with an efficient sampling strategy. Although our stopping rules require significantly fewer samples than \textit{TS}, the standard deviation of the total sample size is larger. This is partly because our stopping rules decide when to stop adaptively, based on sequentially collected information, whereas TS determines the required sample size only once after the initial sampling stage, leading to more stable performance. In addition, the variability of the sampling strategy can also lead to the fluctuation in the total size of samples used.

\begin{table}[h]   \label{tab-synthetic linear random}
\caption{Comparison of averaged sample sizes on random cases under structured linear setting}
\centering
{
\fontsize{10pt}{12pt}\selectfont
\begin{tabular}{ccccccccc}
\toprule
Target $\Pre{1}\geq 0.95$  & \multicolumn{2}{c}{\textit{$\tau_{\alpha,\delta}^{\one,L}$ with C-OCBA-L}} &  & \multicolumn{2}{c}{\textit{$\tau_{\alpha,\delta}^{\one,L}$ with EA}}       &  & \multicolumn{2}{c}{\textit{TS}}       \\ \cmidrule(lr){2-3} \cmidrule(lr){5-6} \cmidrule(l){8-9}
Case & Avg. SSize   & (Std)         &  & Avg. SSize & (Std)        &  & Avg. SSize & (Std)        \\ \midrule
1    & 3271.37           & (751.46)    &  & 10902.32        & (3828.83)  &  & 79875.09        & (7016.84)  \\
2    & 6933.47          & (2708.98)   &  & 15674.48        & (5932.94)  &  & 19865.84        & (2348.18)  \\
3    & 7603.27          & (2168.50)   &  & 49719.28        & (16260.06) &  & 42117.30        & (1962.54)  \\
4    & 2234.95           & (1185.27)   &  & 6316.91          & (3622.96)  &  & 14532.93        & (2540.88)  \\
5    & 8719.18             & (5202.13)   &  & 40612.12        & (18237.26) &  & 32306.84        & (2885.93)  \\ \midrule
Target $\Pre{2}\geq 0.95$   & \multicolumn{2}{c}{\textit{$\tau_{\alpha,\delta}^{\two,L}$ with C-OCBA-L}} &  & \multicolumn{2}{c}{\textit{$\tau_{\alpha,\delta}^{\two,L}$ with EA}}       &  & \multicolumn{2}{c}{\textit{TS}}       \\ \cmidrule(lr){2-3} \cmidrule(lr){5-6} \cmidrule(l){8-9}
Case & Avg. SSize   & (Std)         &  & Avg. SSize & (Std)        &  & Avg. SSize & (Std)        \\ \midrule
1    & 3723.81            & (791.65)    &  & 12423.36        & (4009.22)  &  & 173923.86       & (15284.58) \\
2    & 6320.02           & (2909.48)   &  & 21321.42        & (6817.55)  &  & 96007.76        & (11559.22) \\
3    & 2526.73          & (817.73)  &  & 4089.28       & (1144.88) &  & 155681.30       & (7302.91)  \\
4    & 2347.22           & (1204.30)   &  & 978.06        & (375.00)  &  & 34499.50         & (6380.58)  \\
5    & 1022.41          & (1243.18)   &  & 1877.44         & (574.92) &  & 114153.78       & (10277.02) \\ \bottomrule
\end{tabular}
}
\end{table}

From Table 3, we observe that \textit{$\tau_{\alpha,\delta}^{\one,L}$ with C-OCBA-L} uses significantly fewer samples than \textit{TS} in Case 1, where the number of actions $k=20$. This is because \textit{TS} can only leverage information from sample variances to allocate samples while \textit{C-OCBA-L} can leverage both sample variances and differences in sample means to allocate samples more efficiently. The intuition is consistent with the results observed in the standard cases. In Case 3 target for $\Pre{1}$ with a large context space, \textit{$\tau_{\alpha,\delta}^{\one,L}$ with C-OCBA-L} slightly outperforms \textit{TS}. Unlike \textit{TS}, which guarantees $\Pre{1}$ by jointly controlling the errors across all contexts, the sequential nature of $\tau_{\alpha,\delta}^{\one,L}$ requires controlling the error for each individual context using the Bonferroni inequality. This induces conservativeness, particularly when the context space is large. Cases 4 and 5 demonstrate the superiority of our stopping rules under scenarios with varying distributions.

\subsection{Case Studies}   \label{sec-case study}
We further demonstrate the practical applicability and effectiveness of our proposed stopping rules through two contextual learning case studies. The first involves personalized movie recommendations under the unstructured setting with random context arrivals. The second focuses on personalized treatment decisions in precision medicine under the structured linear setting using simulation models.

\subsubsection{Personalized Movie Recommendations}
In this case study, we use a public movie recommendations dataset collected by GroupLens Research to simulate the sampling process for learning an effective recommendation policy. This dataset has been widely used in the literature as a benchmark for recommendation algorithms \citep{harper2015movielens, bastani2022learning}. It contains over 20 million user ratings on 27,000 movies from 138,000 users. We adopt a random sample of 100,000 ratings provided by MovieLens, from 671 users over 9,066 movies. Ratings are made on a scale of one to five, with an average of 3.65. 

Following \cite{bastani2022learning}, we extract latent features of users and movies using low-rank matrix factorization on the rating data, where a rank of five provides a good fit. We them employ a Gaussian mixture model (GMM) to cluster user features into 8 groups, each characterized by its mean feature vector and population proportion. This clustering approach is also consistent with the method in \cite{Li2024b}, which leverages context clustering to enhance sampling efficiency. Figure~\ref{fig-radar} shows a radar chart of the mean features for the eight user groups, which shows their heterogeneous movie preferences.
\begin{figure}[h]
    \centering
\includegraphics[width=0.6\textwidth]{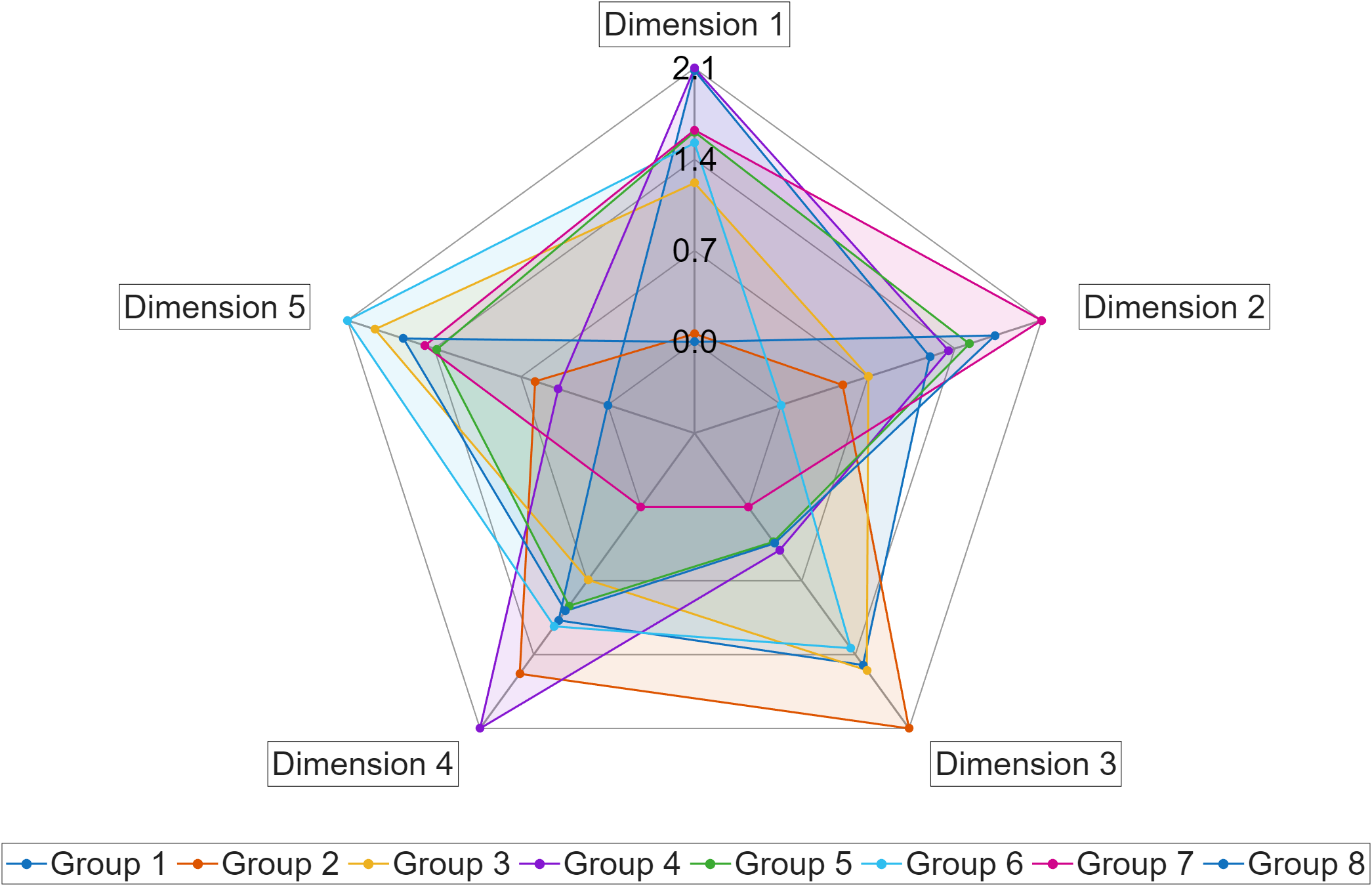}
    \caption{Radar chart of mean user features across eight groups.}\label{fig-radar}
\end{figure}

We use the group proportions as the probabilities of user contexts and generate random users accordingly. Twenty movies are randomly selected from the dataset as candidate actions. The latent rating score $y(\x,a)$ are calculated as the inner product of the corresponding context and movie feature vectors. At each sampling stage $t$, an action is assigned to the arriving user and a noisy rating sample $Y_t = y(\X_t,A_t)+\varepsilon_t$ is observed, where the noise $\varepsilon_t$ follows the distribution $\mathcal{N}(0,0.1^2)$. We adopt \textit{CTD} as the sampling strategy. The stopping rules $\tau_{\alpha,\delta}^{\one}$ and $\tau_{\alpha,\delta}^{\two}$ are used to obtain the $\Pre{1}$- and $\Pre{2}$-guaranteed recommendation policies, respectively. We compare the required sample sizes of our stopping rules with those of \textit{JDK}. We also compare with the complete \textit{CTSD} to compare their stopping rule with ours. Their stopping rules are given with the information of sampling variances. The \textit{KN} procedure is not included since it is inapplicable to random contexts. We set $\alpha = 0.05$, $\delta = 0.1$ and conduct 1000 replications. The results are shown in Table 4, which indicate that our stopping rules require substantially fewer samples than \textit{JDK}. Though \textit{CTSD} is applied with the additional information of known sampling variances, our stopping rules still require fewer samples than it. The case study demonstrates the efficiency and applicability of our stopping rules to contextual learning problems with random contexts.

\begin{table}[h]
\centering
\caption{Comparison of averaged sample sizes on the case study under unstructured setting}
\label{tab-case unstructured}
{
\fontsize{10pt}{12pt}\selectfont
\begin{tabular}{@{}ccclcclcc@{}}
\toprule
   & \multicolumn{2}{c}{\textit{$\tau_{\alpha,\delta}$ with CTD}} &  & \multicolumn{2}{c}{\textit{JDK with CTD}} &  & \multicolumn{2}{c}{\textit{CTSD}} \\ \cmidrule(lr){2-3} \cmidrule(l){5-6} \cmidrule(l){8-9}
   & Avg. SSize        & (Std)      &  & Avg. SSize         & (Std)  &  & Avg. SSize         & (Std)       \\ \midrule
Target $\Pre{1}\geq 0.95$ & 6389.78     & (1770.49)   &  & 32159.63    & (7470.07) &  & 9336.79 & (2958.18)  \\
Target $\Pre{2}\geq 0.95$ & 4350.20    & (259.10)   &  & 10310.19    & (411.82) &  & 4820.02 & (281.33)    \\ \bottomrule
\end{tabular}%
}
\end{table}

\subsubsection{Personalized Treatments for Chronic Obstructive Pulmonary Disease}

Chronic Obstructive Pulmonary Disease (COPD) is the fourth leading cause of death worldwide, causing 3.5 million deaths in 2021, approximately 5\% of all global deaths, according to the World Health Organization \citep{WHO2025copd}. In this case study, we use a simulation model for COPD to learn personalized treatment decisions with precision guarantee. This
example has also been considered in \cite{Du2024}, where they utilize the case to study the efficiency of sampling algorithms with a fixed simulation budget. Characterized by progressive airflow limitation, symptoms of COPD include long-term breathlessness, cough, and sputum production. As a chronic condition, patients may experience three types of adverse events: exacerbation, pneumonia, and death. These events occur stochastically and depend on the patient’s current health state. Even after recovery from an adverse event, recurrence remains possible.

Currently, COPD remains incurable, which makes effective health management especially important. Four treatment methods can be adopted to improve the patients’ quality of life \citep{hoogendoorn2019broadening, corro2020address}: reducing the decline rate in lung function by 30\%, increasing the time to exacerbation by 30\%, improving the physical activity level by 3 points, and reducing the probability of having cough/sputum by 30\%. For simplicity, we consider the initial age of developing into COPD, the number of packs smoked each year and the gender as patient characteristics that determine the effectiveness of a treatment regimen. More specifically, the context vector is defined as $\mathbf{X} = (1, X_2, X_3, X_4)^{\T}$, where $X_2$, $X_3$, and $X_4$ represent the initial age, smoking level, and gender, respectively. Here, the first dimension of contexts is fixed at one to include an intercept term in the linear models. According to \cite{corro2020address}, smoking levels are categorized into six groups: 0 (corresponds to nonsmokers), 1-19, 20-29, 30-39, 40-49, and 50-59 packs per year. Age is divided into six five-year intervals: 40-44, 45-49, 50-54, 55-59, 60-64 and 65-69. The gender is binary with male and female. In total, there are 72 contexts. 

\begin{figure}[h]
    \centering
    \includegraphics[width=0.8\textwidth]{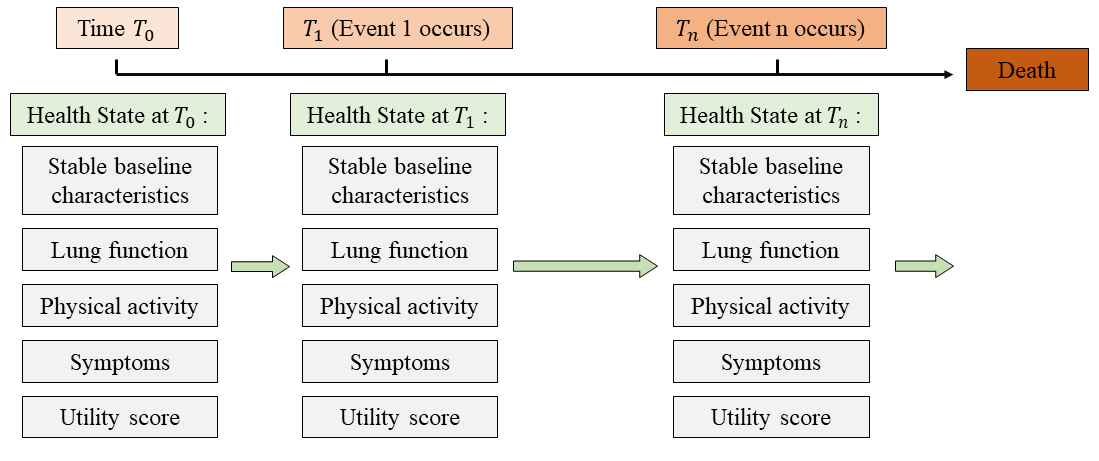}
    \caption{ Simulation model for the chronic obstructive pulmonary disease}
    \label{fig-copd}
\end{figure}

For chronic diseases, the expected quality-adjusted life years (QALYs) of patients serve as a common measure of treatments' effectiveness. We use the simulation model in \cite{hoogendoorn2019broadening} to estimate the QALYs of each treatment regimen across different patient categories, which is illustrated in Figure~\ref{fig-copd}. The state transition probabilities are estimated using historical patient data. For each fixed patient context-treatment pair, we model the replicated QALY response as a simulation output that is conditionally Gaussian around its mean, with context and treatment (action) dependent variance. Since the contexts are controllable in the simulation, we compare the total sample sizes required to identify the optimal treatment for each context using the same procedures as in the synthetic linear experiment. The design points are chosen as two endpoints value of each dimension. Let $n_0 = 50$, $\alpha = 0.05$ and $\delta = 1/6$, which corresponds to a two-month QALY difference. We conduct 300 macro-replications to evaluate each procedure and the results are shown in Table 5. When combined with \textit{C-OCBA-L}, our stopping rules consistently require fewer samples than \textit{TS}. The standard deviations of the total sample size of all three methods are relatively large since the sampling variance of the simulation model is high.

\begin{table}[h]
\centering
\caption{Comparison of averaged sample sizes on the case study under structured linear setting}
\label{tab-case linear}
{
\fontsize{10pt}{16pt}\selectfont
\begin{tabular}{ccclcclcc}
\hline
   & \multicolumn{2}{c}{\textit{$\tau_{\alpha,\delta}^{L}$ with C-OCBA-L}} &  & \multicolumn{2}{c}{\textit{$\tau_{\alpha,\delta}^{L}$ with EA}} &  & \multicolumn{2}{c}{\textit{TS}} \\ \cline{2-3} \cline{5-6} \cline{8-9} 
   & Avg. SSize         & (Std)         &  & Avg. SSize       & (Std)        &  & Avg. SSize       & (Std)        \\ \hline
Target $\Pre{1}\geq 0.95$ & 35233.97    & (14675.83)   &  & 94788.27  & (33967.94) &  & 73364.29  & (1495.24)  \\
Target $\Pre{2}\geq 0.95$ & 18823.41    & (10099.61)   &  & 17805.47  & (4413.07) &  & 253448.4  & (5355.48)  \\ \hline
\end{tabular}
}
\end{table}

\section{Conclusions}  
\label{sec-conclusions}

This paper studies a fundamental deployment question in contextual learning, that when can one stop collecting data and certify, at a user-specified tolerance and confidence level, that the learned policy is good enough to implement. We propose precision-guaranteed sequential stopping rules built around plug-in generalized likelihood ratio evidence and designed to remain valid when sampling variances are unknown. The framework covers both unstructured settings and structured linear models, and it yields implementable procedures that can be paired with a wide range of sampling strategies. Numerical experiments and case studies illustrate that the resulting rules can substantially reduce the amount of data required to achieve the target precision while maintaining rigorous finite-sample guarantees.

The paper provides an operationally meaningful certification mechanism for contextual decisions that is compatible with the hybrid evidence streams common in practice. This allows evidence to be accumulated coherently over time without re-deriving source-specific tests. A key enabling technique for this research is a new way to calibrate GLR boundaries by controlling the GLR-type evidence directly via time-uniform deviation inequalities, rather than relying on KL-proxy bounds or loose union-bound assemblies. This approach yields non-asymptotic guarantees with substantially reduced conservativeness and, more importantly, avoids the common split between a provably valid but impractical rule and a heuristic rule used in empirical work. As a result, it delivers stopping rules that are both theoretically certified and readily usable in practice for data-driven operations.

\bibliographystyle{poms}
\bibliography{BIB}

@article{Abbasi2011Improved,
  title={Improved algorithms for linear stochastic bandits},
  author={Abbasi-Yadkori, Yasin and P{\'a}l, D{\'a}vid and Szepesv{\'a}ri, Csaba},
  journal={Advances in neural information processing systems},
  volume={24},
  year={2011}
}

@article{Chernoff1959Sequential,
  title = {Sequential Design of Experiments},
  author = {Chernoff, Herman},
  year = {1959},
  journal = {The Annals of Mathematical Statistics},
  volume = {30},
  number = {3},
  pages = {755--770},
  publisher = {Institute of Mathematical Statistics},
  langid = {english}
}

@article{Du2024,
  title={A contextual ranking and selection method for personalized medicine},
  author={Du, Jianzhong and Gao, Siyang and Chen, Chun-Hung},
  journal={Manufacturing \& Service Operations Management},
  volume={26},
  number={1},
  pages={167--181},
  year={2024},
  publisher={INFORMS}
}

@inproceedings{gao2019selecting,
  author={Gao, Siyang and Du, Jianzhong and Chen, Chun-Hung},
  booktitle={2019 IEEE 15th International Conference on Automation Science and Engineering (CASE)}, 
  title={Selecting the Optimal System Design under Covariates}, 
  year={2019},
  volume={},
  number={},
  pages={547-552},
  publisher = {IEEE},
  address = {Piscataway, New Jersey, USA}
  }

@inproceedings{Garivier2016Optimal,
  title = {Optimal Best Arm Identification with Fixed Confidence},
  booktitle = {29th Annual Conference on Learning Theory (COLT)},
  author = {Garivier, Aur{\'e}lien and Kaufmann, Emilie},
  editor = 	 {Feldman, Vitaly and Rakhlin, Alexander and Shamir, Ohad},
  year = {2016},
  month = jun,
  series = {Proceedings of Machine Learning Research},
  volume = {49},
  pages = {998--1027},
  publisher = {PMLR},
  address = {New York, New York, USA},
  langid = {english}
}

@article{howard2020time,
  title = {Time-Uniform Chernoff Bounds via Nonnegative Supermartingales},
  author = {Howard, Steven R and Ramdas, Aaditya and McAuliffe, Jon and Sekhon, Jasjeet},
  year = {2020},
  journal = {Probability Surveys},
  volume = {17},
  pages = {257--317},
  langid = {english}
}

@InProceedings{Jourdan,
  title = 	 {Dealing with Unknown Variances in Best-Arm Identification},
  author =       {Jourdan, Marc and R{\'e}my, Degenne and Emilie, Kaufmann},
  booktitle = 	 {34th International Conference on Algorithmic Learning Theory (ALT)},
  pages = 	 {776--849},
  year = 	 {2023},
  series = {Proceedings of Machine Learning Research},
  editor = 	 {Agrawal, Shipra and Orabona, Francesco},
  publisher =    {PMLR},
  address = {Singapore}
}

@article{Kaufmann,
  title = {Mixture Martingales Revisited with Applications to Sequential Tests and Confidence Intervals},
  author = {Kaufmann, Emilie and Koolen, Wouter M},
  journal = {Journal of Machine Learning Research},
  year    = {2021},
  volume  = {22},
  number  = {246},
  pages   = {1--44}
}

@article{Li2024b,
  title = {Efficient Learning for Clustering and Optimizing Context-Dependent Designs},
  author = {Li, Haidong and Lam, Henry and Peng, Yijie},
  year = {2024},
  journal = {Operations Research},
  volume = {72},
  number = {2},
  pages = {617--638}
}

@article{Qin2017,
  title={Improving the expected improvement algorithm},
  author={Qin, Chao and Klabjan, Diego and Russo, Daniel},
  journal={Advances in Neural Information Processing Systems},
  volume={30},
  year={2017}
}

@article{Shen2021,
  title={Ranking and selection with covariates for personalized decision making},
  author={Shen, Haihui and Hong, L Jeff and Zhang, Xiaowei},
  journal={INFORMS Journal on Computing},
  volume={33},
  number={4},
  pages={1500--1519},
  year={2021},
  publisher={INFORMS}
}

@article{Simchi2024,
  title = {On Experimentation With Heterogeneous Subgroups: An Asymptotic Optimal $\delta$-Weighted-PAC Design},
  author = {Simchi-Levi, David and Wang, Chonghuan and Xu, Jiamin},
  year = {2024},
  journal = {SSRN preprint SSRN:4721755},
}

@article{li2022instance,
  title={Instance-optimal pac algorithms for contextual bandits},
  author={Li, Zhaoqi and Ratliff, Lillian and Jamieson, Kevin G and Jain, Lalit and others},
  journal={Advances in Neural Information Processing Systems},
  volume={35},
  pages={37590--37603},
  year={2022}
}

@article{qin2022adaptivity,
  title={Adaptivity and confounding in multi-armed bandit experiments},
  author={Qin, Chao and Russo, Daniel},
  journal={arXiv preprint arXiv:2202.09036},
  year={2022}
}

@article{wang2025anytime,
  title={Anytime-valid t-tests and confidence sequences for Gaussian means with unknown variance},
  author={Wang, Hongjian and Ramdas, Aaditya},
  journal={Sequential Analysis},
  volume={44},
  number={1},
  pages={56--110},
  year={2025},
  publisher={Taylor \& Francis}
}

@article{jedra2020optimal,
  title={Optimal best-arm identification in linear bandits},
  author={Jedra, Yassir and Proutiere, Alexandre},
  journal={Advances in Neural Information Processing Systems},
  volume={33},
  pages={10007--10017},
  year={2020}
}

@inproceedings{li2010contextual,
  title={A contextual-bandit approach to personalized news article recommendation},
  author={Li, Lihong and Chu, Wei and Langford, John and Schapire, Robert E},
  booktitle={Proceedings of the 19th International Conference on World Wide Web (WWW)},
  editor = "Rappa, Michael and Jones, Paul and Freire, Juliana and Chakrabarti, Soumen",
  publisher = "ACM",
  address = {New York, NY, USA},
  pages={661--670},
  year={2010}
}

@article{delshad2022adaptive,
  title={Adaptive Design of Personalized Dose-Finding Clinical Trials},
  author={Delshad, Saeid and Khademi, Amin},
  journal={Service Science},
  volume={14},
  number={4},
  pages={273--291},
  year={2022},
  publisher={INFORMS}
}

@article{kinyanjui2023fast,
  title={Fast treatment personalization with latent bandits in fixed-confidence pure exploration},
  author={Kinyanjui, Newton Mwai and Carlsson, Emil and Johansson, Fredrik D},
  journal={Transactions on Machine Learning Research},
  issn={2835-8856},
  year={2023}
}

@article{karimi2018news,
  title={News recommender systems--Survey and roads ahead},
  author={Karimi, Mozhgan and Jannach, Dietmar and Jugovac, Michael},
  journal={Information Processing \& Management},
  volume={54},
  number={6},
  pages={1203--1227},
  year={2018},
  publisher={Elsevier}
}

@article{zhou2023offline,
  title={Offline multi-action policy learning: Generalization and optimization},
  author={Zhou, Zhengyuan and Athey, Susan and Wager, Stefan},
  journal={Operations Research},
  volume={71},
  number={1},
  pages={148--183},
  year={2023},
  publisher={INFORMS}
}

@article{zhan2024policy,
  title={Policy learning with adaptively collected data},
  author={Zhan, Ruohan and Ren, Zhimei and Athey, Susan and Zhou, Zhengyuan},
  journal={Management Science},
  volume={70},
  number={8},
  pages={5270--5297},
  year={2024},
  publisher={INFORMS}
}

@article{bastani2022learning,
  title={Learning personalized product recommendations with customer disengagement},
  author={Bastani, Hamsa and Harsha, Pavithra and Perakis, Georgia and Singhvi, Divya},
  journal={Manufacturing \& Service Operations Management},
  volume={24},
  number={4},
  pages={2010--2028},
  year={2022},
  publisher={INFORMS}
}

@article{ding2021frisch,
  title={The {F}risch-{W}augh-{L}ovell theorem for standard errors},
  author={Ding, Peng},
  journal={Statistics \& Probability Letters},
  volume={168},
  pages={108945},
  year={2021},
  publisher={Elsevier}
}

@article{keslin2025ranking,
  title={Ranking and contextual selection},
  author={Keslin, Gregory and Nelson, Barry L and Pagnoncelli, Bernardo and Plumlee, Matthew and Rahimian, Hamed},
  journal={Operations Research},
  volume={73},
  number={5},
  pages={2695--2707},
  year={2025},
  publisher={INFORMS}
}

@article{kim2001fully,
  title={A fully sequential procedure for indifference-zone selection in simulation},
  author={Kim, Seong-Hee and Nelson, Barry L},
  journal={ACM Transactions on Modeling and Computer Simulation (TOMACS)},
  volume={11},
  number={3},
  pages={251--273},
  year={2001},
  publisher={ACM New York, NY, USA}
}

@article{harper2015movielens,
  title={The movielens datasets: History and context},
  author={Harper, F Maxwell and Konstan, Joseph A},
  journal={Acm Transactions on Interactive Intelligent Systems},
  volume={5},
  number={4},
  pages={1--19},
  year={2015},
  publisher={Acm New York, NY, USA}
}

@misc{WHO2025copd,
  author       = {{World Health Organization}},
  title        = {Chronic obstructive pulmonary disease ({COPD})},
  year         = {2024},
  note         = {Fact sheet. Accessed November 10, 2025},
  howpublished = {\url{https://www.who.int/news-room/fact-sheets/detail/chronic-obstructive-pulmonary-disease-(copd)}}
}

@article{hoogendoorn2019broadening,
  title={Broadening the perspective of cost-effectiveness modeling in chronic obstructive pulmonary disease: a new patient-level simulation model suitable to evaluate stratified medicine},
  author={Hoogendoorn, Martine and Ramos, Isaac Corro and Baldwin, Michael and Guix, Nuria Gonzalez-Rojas and Rutten-van M{\"o}lken, Maureen PMH},
  journal={Value in Health},
  volume={22},
  number={3},
  pages={313--321},
  year={2019},
  publisher={Elsevier}
}

@article{corro2020address,
  title={How to address uncertainty in health economic discrete-event simulation models: an illustration for chronic obstructive pulmonary disease},
  author={Corro Ramos, Isaac and Hoogendoorn, Martine and Rutten-van M{\"o}lken, Maureen PMH},
  journal={Medical Decision Making},
  volume={40},
  number={5},
  pages={619--632},
  year={2020},
  publisher={SAGE Publications Sage CA: Los Angeles, CA}
}

@article{li2025efficient,
  title={Efficient simulation budget allocation for contextual ranking and selection with quadratic models},
  author={Li, Dongyang and Chew, Ek Peng and Li, Haobin and Y{\"u}cesan, Enver and Chen, Chun-Hung},
  journal={European Journal of Operational Research},
  volume = {328},
  number = {3},
  pages = {862-876},
  year = {2026}
}

@article{wang2025reinforcement,
title = {Reinforcement learning algorithm for reusable resource allocation with unknown rental time distribution},
journal = {European Journal of Operational Research},
volume = {331},
number = {1},
pages = {186-199},
year = {2026},
author = {Ziwei Wang and Jie Song and Yixuan Liu and Jingtong Zhao},
}

@article{pan2020online,
  title={Online contextual learning with perishable resources allocation},
  author={Pan, Xin and Song, Jie and Zhao, Jingtong and Truong, Van-Anh},
  journal={IISE Transactions},
  volume={52},
  number={12},
  pages={1343--1357},
  year={2020},
  publisher={Taylor \& Francis}
}

@article{zhalechian2022online,
  title={Online resource allocation with personalized learning},
  author={Zhalechian, Mohammad and Keyvanshokooh, Esmaeil and Shi, Cong and Van Oyen, Mark P},
  journal={Operations Research},
  volume={70},
  number={4},
  pages={2138--2161},
  year={2022},
  publisher={INFORMS}
}

@article{hadad2021confidence,
  title={Confidence intervals for policy evaluation in adaptive experiments},
  author={Hadad, Vitor and Hirshberg, David A and Zhan, Ruohan and Wager, Stefan and Athey, Susan},
  journal={Proceedings of the National Academy of Sciences},
  volume={118},
  number={15},
  pages={e2014602118},
  year={2021},
  publisher={National Academy of Sciences}
}

\newpage
\appendix
{\noindent \LARGE \textbf{Appendix}}

\vspace{15pt}
This document provides further discussion of the idea of joint error control for $\Pre1$, proofs of the theoretical claims in the main paper, and additional details on the numerical experiments.

\section{Discussion on Joint Error Control for $\Pre1$}
In this appendix, we discuss why the GLR test cannot be directly used to $\Pre1$ by jointly controlling the errors across contexts. When developing stopping rules, the idea of jointly controlling errors across contexts is common (e.g., as in \cite{Simchi2024}). This idea is also natural in our setting. For some contexts $\x$ where the optimal action can be easily identified (i.e., substantial evidence is gathered quickly), the error probability under $\x$ will be naturally small. This suggests that one may borrow error budget from such easy contexts and allocate more to harder ones. In particular, for contexts $\x'$ that are harder to distinguish, we could require less evidence to identify the optimal action, granting a larger error budget, as long as the averaged error over all contexts remains below $\alpha$.

However, this joint control idea cannot be applied to the GLR test to guarantee $\Pre1$. In sequential learning targeting a fixed precision, there are two sources of randomness: the randomness of the samples and the randomness of the stopping time. The GLR statistic provides information on the concentration of randomness in the sample at a fixed stopping time. To account for the random stopping time, we require a time-uniform boundary for the GLR statistic to control the error probability, ensuring that the error allocation is uniform over time. This uniform error allocation prevents the idea of setting a time-varying allocation for the error probability across contexts. Thus, we can only allocate the error $\alpha$ to each context before the learning process begins.
    
\section{Proofs of Theoretical Claims}
\subsection{Proof of Lemma 1}
\begin{lemma}[Ville's maximal inequality]   \label{lemma-ville}
Let $(L_t)_{t\ge 0}$ be a nonnegative supermartingale with $L_0$ finite. Then, for any $a>0$,
\begin{equation*} \label{eq:ville}
    \mathbb{P} \left(\sup_{t\geq 0} L_t \geq a \right)
    \;\leq\; \frac{\mathbb{E}[L_0]}{a}.
\end{equation*}
\end{lemma}

\begin{lemma}[Gaussian mixture martingale, \cite{wang2025anytime}]   \label{lemma-mixture martingale}
    Let $Y_1, ..., Y_t$ be $t$ i.i.d. samples from the Gaussian distribution with mean $\mu$. Define $R_t = \sum_{i=1}^{t} Y_i$ and $W_t = \sum_{i=1}^{t} Y_i^2$. For any $s>0$, the process $\{G_t\}_{t\geq 1}$ defined by 
    \begin{equation}   \label{equa-concentration bound general}
        G_t = \sqrt{\frac{s^2}{t+s^2}} \left(1 - \frac{(R_t - t\mu)^2}{(t+s^2)(W_t - 2R_t\mu + t\mu^2)} \right)^{-\frac{t}{2}}
    \end{equation}
    is a martingale with initial value one.
\end{lemma}

For $r\in\{1,2\}$, let $Y_{1,r},Y_{2,r},\ldots$ be i.i.d.\ Gaussian with unknown mean $y_r$ and unknown variance $\sigma_r^2\in(0,\infty)$. Let $N_{t,r}$, $\overline{Y}_{t,r}$, and $S_{t,r}^2$ denote, respectively, the sample size, sample mean, and sample variance computed from the first $N_{t,r}$ observations of stream $r$ by stage $t$.  Define the self-normalized deviations
$$
V_{t,r}\;:=\; N_{t,r}\,\frac{(\overline{Y}_{t,r}-y_r)^2}{2S_{t,r}^2},
\qquad r\in\{1,2\},
$$
such that $U_t^{+}=V_{t,1}+V_{t,2}$.

Next we find a time-uniform boundary for $V_{t,r}$. Fix a stream $r$ and let $\{G_t^{(r)}\}_{t\geq 0}$ be the martingale of Lemma~\ref{lemma-mixture martingale} specialized to $s=1$ and $\mu=y_r$. Define the embedded process
$$
M_{t,r}\;:=\;G_{N_{t,r}}^{(r)},\qquad t\geq 0.
$$
A test martingale is defined as a non-negative martingale with initial value one. Then $\{M_{t,r}\}_{t\geq 0}$ is a test martingale with respect to the natural filtration of stream $r$. Further we have $M_{t,r}$ can be denoted as a function of $V_{t,r}$:
\begin{align*}
    M_{t,r} &= \sqrt{\frac{1}{N_{t,r}+1}} \left(1 - \frac{N_{t,r}^2 (\overline{Y}_{t,r}-y_r)^2}{(N_{t,r}+1)(N_{t,r} (\overline{Y}_{t,r}-y_r)^2 + N_{t,r} S_{t,r}^2)} \right)^{-\frac{N_{t,r}}{2}}   \\
    &= \sqrt{\frac{1}{N_{t,r}+1}} \left(\frac{(N_{t,r} (\overline{Y}_{t,r}-y_r)^2 + N_{t,r} S_{t,r}^2) + N_{t,r}^2 S_{t,r}^2}{(N_{t,r}+1)(N_{t,r} (\overline{Y}_{t,r}-y_r)^2 + N_{t,r} S_{t,r}^2)} \right)^{-\frac{N_{t,r}}{2}}   \\
    &= \sqrt{\frac{1}{N_{t,r}+1}} \left(\frac{(V_{t,r} + N_{t,r}) + N_{t,r}^2}{(N_{t,r}+1)(V_{t,r} + N_{t,r})} \right)^{-\frac{N_{t,r}}{2}}.
\end{align*}

Therefore, by Ville's maximal inequality (Lemma \ref{lemma-ville}), for any $\alpha\in(0,1)$, we have the following time-uniform concentration inequality for $M_{t,r}$
\begin{align*}
    \mathbb{P} \left(\exists \,t\in\mathbb{N}^*, M_{t,r}\geq \frac{1}{\alpha} \right) &= \mathbb{P} \left(\exists \,t\in\mathbb{N}^*, \frac{N_{t,r}^2}{(N_{t,r}+1)(V_{t,r}+N_{t,r})}\leq \left(\frac{N_{t,r}+1}{\alpha^2 } \right)^{-\frac{1}{N_{t,r}}} - \frac{1}{N_{t,r}+1} \right) \\
    &= \mathbb{P} \left(\exists \,t\in\mathbb{N}^*, V_{t,r}\geq \frac12 \gamma\,(N_{t,r},\alpha) \right) \leq \alpha,
\end{align*}
where let $\epsilon > 0$ be arbitrarily small and the function $\gamma:\mathbb{N}^* \times (0,1) \to (0,+\infty)$ is defined as 
\begin{equation}   \label{equa-gamma-appendix} 
  \gamma (t,\alpha) = \frac{t^2}{\max\left\{\left(\frac{\alpha^2}{t+1} \right)^{\frac{1}{t}} (t+1) - 1, \epsilon \right\} } - t.
\end{equation}

Next we consider the boundary for the two-summed deviation term. Let $\{\mathcal{F}_t\}_{t\ge 0}$ be the natural filtration generated by the sampling process up to stage $t$. Define the product process
$$
M_t^*\;:=\;M_{t,1}\,M_{t,2},\qquad t\ge 0.
$$
Since at each stage at most one stream receives a new sample, we have
\begin{align*}
    \mathbb{E}[M_{t+1}^*\mid\mathcal{F}_t] &=
    \mathbb{E}[G_{N_{t,1}+1}^{(1)} G_{N_{t,2}}^{(2)}\mid\mathcal{F}_t]\cdot\mathbf{1}\{A_t=1\}
    + \mathbb{E}[G_{N_{t,1}}^{(1)} G_{N_{t,2}+1}^{(2)}\mid\mathcal{F}_t]\cdot\mathbf{1}\{A_t=2\}
    \\&\qquad + \mathbb{E}[G_{N_{t,1}}^{(1)}G_{N_{t,2}}^{(2)}\mid\mathcal{F}_t]\cdot\mathbf{1}\{A_t\notin\{1,2\}\} = M_{t,1}M_{t,2}\ =\ M_t^*.
\end{align*}
Hence $\{M_t^*\}_{t\geq 0}$ is a test martingale with initial value one. Applying Ville's inequality gives
\begin{equation}\label{eq:ville-prod}
\mathbb{P}\left(\exists\, t\in\mathbb{N}^*:\ M_t^*\geq \frac{1}{\alpha}\right)\ \leq\ \alpha.
\end{equation}

For the martingale $G_{N_{t,r}}^{(r)}$, let $N_{t,r} = n$ and $V_{t,r} = V$, then it can be written as
$$
G_{n}^{(r)}(V)
=
\sqrt{\frac{1}{n+1}}
\left(
\frac{n^2 + n + V}{(n+1)(n+V)}
\right)^{-n/2},
\qquad n\ge 1,\ V\ge 0,
$$
A direct calculation shows that for all $n\geq 1$, $\log G_{n}^{(r)}(V)$ is strictly concave in $V$ on $[0,\infty)$:
$$
\frac{d^2}{d\,V^2}\log G_n^{(r)}(V)
=
- \frac{n^3 (2V + n^2 +2n)}{2(V+n)^2 (V+n^2+n)^2} \ <\ 0,
$$
Fix $n_1,n_2\geq 1$ and a total deviation $v\geq 0$.  Define
$$
h(v_1)\;:=\;\log G_{n_1}^{(1)}(v_1)+\log G_{n_2}^{(2)}(v-v_1),\qquad v_1\in[0,v].
$$
Since $\log G_{n_1}^{(1)}$ is concave and $v_1\mapsto \log G_{n_2}^{(2)}(v-v_1)$ is also concave, $h$ is concave on $[0,v]$. Therefore, its minimum over $[0,v]$ is attained at an endpoint, implying that
\begin{equation}\label{equa-endpoint min}
\min_{\substack{v_1,v_2\ge 0\\ v_1+v_2=v}}
G_{n_1}^{(1)} (v_1)\,G_{n_2}^{(2)} (v_2)
=
\min\big\{G_{n_1}^{(1)} (v)\,G_{n_2}^{(2)} (0),\ G_{n_1}^{(1)} (0)\,G_{n_2}^{(2)} (v)\big\}.
\end{equation}
Applying \eqref{equa-endpoint min} with $n_r=N_{t,r}$, $v_r=V_{t,r}$ and
$v=U_t^{+}$, we have, for each $t$,
\begin{equation}\label{equa-prod-lower}
M_t^*
=
G_{N_{t,1}}^{(1)} (V_{t,1})\cdot G_{N_{t,2}}^{(2)} (V_{t,2})
\ \ge\
\min\Big\{G_{N_{t,1}}^{(1)} (U_t^{+})\,G_{N_{t,2}}^{(2)} (0),\ G_{N_{t,1}}^{(1)} (0)\,G_{N_{t,2}}^{(2)} (U_t^{+})\Big\}.
\end{equation}

For each $t$, define
$$
b_{t,1}
:=\frac12\,\gamma\!\left(N_{t,1},\ \alpha \sqrt{\frac{1}{N_{t,2}+1}}\right),
\qquad
b_{t,2}
:=\frac12\,\gamma\!\left(N_{t,2},\ \alpha \sqrt{\frac{1}{N_{t,1}+1}}\right),
\qquad
B_t:=\max\{b_{t,1},b_{t,2}\}.
$$
Suppose $U_t^{+}>B_t$. Then $U_t^{+}>b_{t,1}$ and $U_t^{+}>b_{t,2}$. By the definition of the function $\tilde{\gamma}$, we have, for $r\in\{1,2\}$,
\begin{equation*}
    \left\{G_{N_{t,r}}^{(r)} \geq \frac{1}{\beta}\right\}\ \equiv \ 
    \left\{V_{t,r} \geq \frac12 \gamma\,(N_{t,r},\beta)\right\}.
\end{equation*}
Applying this with $\beta = \alpha \sqrt{1/(N_{t,2}+1)}$ for stream 1 and $\beta = \alpha \sqrt{1/(N_{t,1}+1)}$ for stream 2, and using $U_t^{+}>b_{t,1}, b_{t,2}$, yields
$$
G_{N_{t,1}}^{(1)} (U_t^{+})\ \ge\ \left(\alpha \sqrt{\frac{1}{N_{t,2}+1}}\right)^{-1},
\qquad
G_{N_{t,2}}^{(2)} (U_t^{+})\ \ge\ \left(\alpha \sqrt{\frac{1}{N_{t,1}+1}}\right)^{-1}.
$$
Since $G_{n_r}^{(r)}(0) = \sqrt{1/(n_r+1)}$, we obtain
$$
G_{N_{t,1}}^{(1)} (U_t^{+})\,G_{N_{t,2}}^{(2)} (0)\ \ge\ \frac{1}{\alpha},
\qquad
G_{N_{t,1}}^{(1)} (0)\,G_{N_{t,2}}^{(2)} (U_t^{+})\ \ge\ \frac{1}{\alpha}.
$$

Taking the minimum of the two products and applying \eqref{equa-prod-lower}, we have
$$
U_t^{+}>B_t\quad\Longrightarrow\quad M_t\ge \frac{1}{\alpha}.
$$
Therefore,
$$
\mathbb{P}\left(\exists\, t\in\mathbb{N}^*:\ U_t^{+}>B_t\right)
\ \le\
\mathbb{P}\left(\exists\, t\in\mathbb{N}^*:\ M_t\ge \frac{1}{\alpha}\right)
\ \le\ \alpha,
$$
where the last inequality follows from \eqref{eq:ville-prod}.  Equivalently,
with probability at least $1-\alpha$, for all $t\in\mathbb{N}^*$,
$$
U_t^{+}
\le
\max \left\{\frac{1}{2} \gamma\left(N_{t,1}, \alpha \sqrt{\frac{1}{N_{t,2}+1}}\right),\ 
\frac{1}{2} \gamma\left(N_{t,2}, \alpha \sqrt{\frac{1}{N_{t,1}+1}}\right) \right\},
$$
which is exactly equation (7). This completes the proof.  \qed

\subsection{Proof of Theorem 1}

By Lemma~1 and the definitions of the boundaries in equation (9), we have
\begin{equation}
\label{appendix-theorem1-pre1-bound}
\mathbb P\!\left(
\exists\, t\in\mathbb N^*:\ 
U_t(\x;a,\pi^*(\x))
>
\varphi_{a,\pi^*(\x)}^{\one}(\bm N_t,\alpha,\x)
\right)
\le
\frac{\alpha}{\big(|\mathcal A(\x)|-1\big)mp(\x)},
\end{equation}
and
\begin{equation}
\label{appendix-theorem1-pre2-bound}
\mathbb P\!\left(
\exists\, t\in\mathbb N^*:\ 
U_t(\x;a,\pi^*(\x))
>
\varphi_{a,\pi^*(\x)}^{\two}(\bm N_t,\alpha,\x)
\right)
\le
\frac{\alpha}{\big(|\mathcal A(\x)|-1\big)m}.
\end{equation}

Moreover, whenever $y(\x,a)\le y(\x,\pi^*(\x))-\eta$ for some $\eta\ge 0$, we have
\begin{equation}
\label{appendix-theorem1-glr-domination}
\tilde Z_{a,\pi^*(\x)}(\x,t,\eta)
\le
U_t(\x;a,\pi^*(\x)),
\qquad
\forall\, t\in\mathbb N^*.
\end{equation}

We first prove the result for $\Pre1$. Let $\mathcal E(\x) = \big\{ \hat\pi_{\tau_{\alpha,\delta}^{\one}}(\x)\in\mathcal A_\delta^*(\x) \big\}$ denote the event that the selected action at the stopping time is $\delta$-optimal under context $\x$. For any action $a\in\mathcal A(\x)\setminus\mathcal A_\delta^*(\x)$, we have
\[
y(\x,a)\le y(\x,\pi^*(\x))-\delta.
\]
Therefore, by \eqref{appendix-theorem1-glr-domination} with $\eta=\delta$,
\[
\tilde Z_{a,\pi^*(\x)}(\x,t,\delta)
\le
U_t(\x;a,\pi^*(\x)),
\qquad
\forall\, t\in\mathbb N^*.
\]
Define the event
\[
\mathcal G_{\x}^{\one}
:=
\bigcap_{a\in\mathcal A(\x)\setminus\mathcal A_\delta^*(\x)}
\left\{
\forall\, t\in\mathbb N^*:\ 
\tilde Z_{a,\pi^*(\x)}(\x,t,\delta)
\le
\varphi_{a,\pi^*(\x)}^{\one}(\bm N_t,\alpha,\x)
\right\}.
\]
By the definition of the stopping rule in (2), this implies that no inferior action can be selected at the stopping time. Hence $\mathcal G_{\x}^{\one}\subseteq \mathcal E(\x)$. Using \eqref{appendix-theorem1-pre1-bound} and the union bound,
\begin{align*}
\mathbb P\big(\mathcal G_{\x}^{\one}\big)
&\ge
1-
\sum_{a\in\mathcal A(\x)\setminus\mathcal A_\delta^*(\x)}
\mathbb P\!\left(
\exists\, t\in\mathbb N^*:\ 
\tilde Z_{a,\pi^*(\x)}(\x,t,\delta)
>
\varphi_{a,\pi^*(\x)}^{\one}(\bm N_t,\alpha,\x)
\right) \\
&\ge
1-
\sum_{a\in\mathcal A(\x)\setminus\mathcal A_\delta^*(\x)}
\mathbb P\!\left(
\exists\, t\in\mathbb N^*:\ 
U_t(\x;a,\pi^*(\x))
>
\varphi_{a,\pi^*(\x)}^{\one}(\bm N_t,\alpha,\x)
\right) \\
&\ge
1-
\sum_{a\in\mathcal A(\x)\setminus\mathcal A_\delta^*(\x)}
\frac{\alpha}{\big(|\mathcal A(\x)|-1\big)mp(\x)} \\
&\ge
1-\frac{\alpha}{mp(\x)}.
\end{align*}
Since $\mathcal G_{\x}^{\one}\subseteq \mathcal E(\x)$, we obtain
\[
\mathbb P\big(\mathcal E(\x)\big)\ge 1-\frac{\alpha}{mp(\x)}.
\]
Therefore,
\begin{align*}
\Pre1 =
\sum_{\x\in\mathcal X}p(\x)\,\mathbb P\big(\mathcal E(\x)\big) \ge 1-\alpha.
\end{align*}

We next prove the result for $\Pre2$. Define
\[
\mathcal G^{\two}
:=
\bigcap_{\x\in\mathcal X}
\bigcap_{a\in\mathcal A(\x)\setminus\{\pi^*(\x)\}}
\left\{
\forall\, t\in\mathbb N^*:\ 
\tilde Z_{a,\pi^*(\x)}\big(\x,t,\Delta(\x,a)\big)
\le
\varphi_{a,\pi^*(\x)}^{\two}(\bm N_t,\alpha,\x)
\right\}.
\]
Using \eqref{appendix-theorem1-pre2-bound} and the union bound,
\begin{align*}
\mathbb P\big(\mathcal G^{\two}\big)
&\ge
1-
\sum_{\x\in\mathcal X}\sum_{a\in\mathcal A(\x)\setminus\{\pi^*(\x)\}}
\mathbb P\!\left(
\exists\, t\in\mathbb N^*:\ 
\tilde Z_{a,\pi^*(\x)}\big(\x,t,\Delta(\x,a)\big)
>
\varphi_{a,\pi^*(\x)}^{\two}(\bm N_t,\alpha,\x)
\right) \\
&\ge
1-
\sum_{\x\in\mathcal X}\sum_{a\in\mathcal A(\x)\setminus\{\pi^*(\x)\}}
\mathbb P\!\left(
\exists\, t\in\mathbb N^*:\ 
U_t(\x;a,\pi^*(\x))
>
\varphi_{a,\pi^*(\x)}^{\two}(\bm N_t,\alpha,\x)
\right) \\
&\ge
1-
\sum_{\x\in\mathcal X}\sum_{a\in\mathcal A(\x)\setminus\{\pi^*(\x)\}}
\frac{\alpha}{\big(|\mathcal A(\x)|-1\big)m} \\
&\ge
1-\alpha.
\end{align*}

We claim that on the event $\mathcal G^{\two}$,
\begin{equation}
\label{appendix-theorem1-local-regret}
y(\x,\pi^*(\x))-y(\x,\hat\pi_t(\x))
\le
r_t(\x),
\qquad
\forall\,\x\in\mathcal X,\ \forall\, t\in\mathbb N^*.
\end{equation}
To see this, fix $\x\in\mathcal X$ and $t\in\mathbb N^*$. If $\hat\pi_t(\x)=\pi^*(\x)$, then the left-hand side of \eqref{appendix-theorem1-local-regret} is zero, so the claim is immediate. Otherwise, let $a=\hat\pi_t(\x)\neq \pi^*(\x)$. Since $a$ is the empirically optimal action under context $\x$ at stage $t$, the function $\eta\mapsto \tilde Z_{a,\pi^*(\x)}(\x,t,\eta)$ is nondecreasing on $[0,\infty)$. On the event $\mathcal G^{\two}$, we have
\[
\tilde Z_{a,\pi^*(\x)}\big(\x,t,\Delta(\x,a)\big)
\le
\varphi_{a,\pi^*(\x)}^{\two}(\bm N_t,\alpha,\x).
\]
By the definition of the certified slack level,
\[
w_{a,\pi^*(\x)}(\x,t)
=
\inf\Big\{
\eta\ge 0:\ 
\tilde Z_{a,\pi^*(\x)}(\x,t,\eta)
>
\varphi_{a,\pi^*(\x)}^{\two}(\bm N_t,\alpha,\x)
\Big\},
\]
it follows that $w_{a,\pi^*(\x)}(\x,t)\ge \Delta_{a}(\x)$.

Since $r(\x,t)=\max_{b\in\mathcal A(\x)\setminus\{\hat{\pi}_t(\x)\}}w_{\hat{\pi}_t(\x),b}(\x,t)$, and $\pi^*(\x)\in\mathcal A(\x)\setminus\{\hat\pi_t(\x)\}$, we obtain
\[
r_t(\x)\ge u_t\big(\x,\pi^*(\x)\big)\ge \Delta(\x,a)
=
y(\x,\pi^*(\x))-y(\x,\hat\pi_t(\x)).
\]
This proves \eqref{appendix-theorem1-local-regret}.

Finally, on the event $\mathcal G^{\two}$, using \eqref{appendix-theorem1-local-regret} and the definition of the stopping rule (5), we have
\begin{align*}
V(\pi^*)-V(\hat\pi_{\tau_{\alpha,\delta}^{\two}})
&:=
\sum_{\x\in\mathcal X}
p(\x)\Big(y(\x,\pi^*(\x))-y(\x,\hat\pi_{\tau_{\alpha,\delta}^{\two}}(\x))\Big) \\
&\le
\sum_{\x\in\mathcal X}p(\x)\,r_{\tau_{\alpha,\delta}^{\two}}(\x) \le
\delta.
\end{align*}
Therefore, we have $\Pre2
=
\mathbb P\big(V(\hat\pi_{\tau_{\alpha,\delta}^{\two}})\ge V(\pi^*)-\delta\big)
\ge
\mathbb P\big(\mathcal G^{\two}\big)
\ge
1-\alpha.
$
\qed

\subsection{Proof of Lemma 2}
We analyze the numerator and denominator of $Z_{a,a'}^{L}$ in (10) separately. Let $\bb^{*1}(a)$ and $\bb^{*1}(a')$ denote the optimal parameters of the numerator maximization problem $$\max_{\fxx^{\T} \bb(a) \geq \fxx^{\T} \bb(a') - \delta} \mathcal{L}_{\bb(a)}^o \left(\underline{Y_{t}}(\x,a)\right) \mathcal{L}_{\bb(a')}^o \left(\underline{Y_{t}}(\x,a')\right)$$ and $\bb^{*2}(a)$, $\bb^{*2}(a')$ denote those of the denominator maximization problem $$\max_{\fxx^{\T} \bb(a) \leq \fxx^{\T} \bb(a') - \delta} \mathcal{L}_{\bb(a)}^o \left(\underline{Y_{t}}(\x,a)\right) \mathcal{L}_{\bb(a')}^o \left(\underline{Y_{t}}(\x,a')\right).$$

Suppose that at stage $t$, we have $\fxx^{\T}\hat{\bb}_{t}(a) \geq \fxx^{\T}\hat{\bb}_{t}(a')$. Then the constraint in the numerator maximization problem becomes inactive, the optimal solution $\left(\bb^{*1}(a), \bb^{*1}(a') \right)$ equals to the maximum likelihood estimates (MLE) $\left(\hat{\bb}_{t}(a), \hat{\bb}_{t}(a') \right)$. Since the log-likelihood function is strictly concave and the exponential function is monotonic, then we can solve the denominator maximization problem by solving the following convex programming problem:
\begin{align*}
    \min_{\bb(a),\bb(a')} & \quad \sum_{1\leq s\leq t,a_s=a} \frac{\left(Y_{s} - \fx{s}^{\T}\bb(a) \right)^2}{2\sigma^2(a)} - \sum_{1\leq s\leq t,a_s=a'} \frac{\left(Y_{s} - \fx{s}^{\T}\bb(a') \right)^2}{2\sigma^2(a')} \\
    s.t. & \quad \fxx^{\T} \bb(a) \leq \fxx^{\T} \bb(a'). 
\end{align*}

The KKT optimality conditions are 
\begin{align*}
    \frac{1}{\sigma^{2}(a)} \left(\sum_{1\leq s\leq t,a_s=a} \fx{s} \fx{s}^{\T} \right) \bb(a) - \frac{1}{\sigma^{2}(a)} \left(\sum_{1\leq s\leq t,a_s=a} Y_{s}\fx{s} \right) + \lambda \fxx &= 0,   \\
    \frac{1}{\sigma^{2}(a')} \left(\sum_{1\leq s\leq t,a_s=a'} \fx{s} \fx{s}^{\T} \right) \bb(a') - \frac{1}{\sigma^{2}(a')} \left(\sum_{1\leq s\leq t,a_s=a'} Y_{s} \fx{s} \right) - \lambda \fxx &= 0,   \\
    \fxx^{\T}\big(\bb(a') - \bb(a)\big) + \delta &\geq 0,   \\
    \lambda \left(\fxx^{\T}\big(\bb(a') - \bb(a)\big) + \delta \right) &= 0,   \\
    \lambda &\geq 0,
\end{align*}
where $\lambda$ is the Lagrange multiplier. The optimal solution $\left(\bb^{*2}(a), \bb^{*2}(a') \right)$ is obtained when $\lambda > 0$ and $\fxx^{\T}\big(\bb(a') - \bb(a)\big) + \delta = 0$, then we have
\begin{align*}
    \bb^{*2}(a) &= \hat{\bb}_{t}(a) - \frac{\left(\fxx^{\T} \hat{\bb}_{t}(a) - \fxx^{\T} \hat{\bb}_{t}(a') + \delta \right) \sigma^2(a) D_{t}^{-1}(a) \fxx }{\sigma^2(a) \fxx^{\T} D_{t}^{-1}(a) \fxx + \sigma^2(a') \fxx^{\T} D_{t}^{-1}(a') \fxx },   \\
    \bb^{*2}(a') &= \hat{\bb}_{t}(a') + \frac{\left(\fxx^{\T} \hat{\bb}_{t}(a) - \fxx^{\T} \hat{\bb}_{t}(a') + \delta \right) \sigma^2(a') D_{t}^{-1}(a') \fxx }{\sigma^2(a) \fxx^{\T} D_{t}^{-1}(a) \fxx + \sigma^2(a') \fxx^{\T} D_{t}^{-1}(a') \fxx }.
\end{align*}

Therefore, substituting $\bb^{*1}(a)$, $\bb^{*1}(a')$, $\bb^{*2}(a)$ and $\bb^{*2}(a')$ into the expression of $Z_{a,a'}^{L}$, we have
\begin{align*}
    Z_{a,a'}^{L} (\x,t,\delta) &= - \sum_{1\leq s\leq t,a_s=a} \frac{\left(Y_{s} - \fx{s}^{\T}\bb^{*1}(a)\right)^2}{2\sigma^2(a)} + \sum_{1\leq s\leq t,a_s=a'} \frac{\left(Y_{s} - \fx{s}^{\T}\bb^{*1}(a') \right)^2}{2\sigma^2(a')} \\
    & \qquad + \sum_{1\leq s\leq t,a_s=a} \frac{\left(Y_{s} - \fx{s}^{\T}\bb^{*2}(a)\right)^2}{2\sigma^2(a)} - \sum_{1\leq s\leq t,a_s=a'} \frac{\left(Y_{s} - \fx{s}^{\T}\bb^{*2}(a')\right)^2}{2\sigma^2(a')}  \\
    &= \frac{1}{2} \sum_{1\leq s\leq t,a_s=a} \frac{\left(2Y_{s} - \fx{s}^{\T}\hat{\bb}_{t}(a) - \fx{s}^{\T}\bb^{*2}(a) \right)}{\sigma(a)} \cdot \frac{\left(\fx{s}^{\T}\bb^{*2}(a) - \fx{s}^{\T}\hat{\bb}_{t}(a)\right)}{\sigma(a)} \\
    & \qquad + \frac{1}{2} \sum_{1\leq s\leq t,a_s=a'} \frac{\left(2Y_{s} - \fx{s}^{\T}\hat{\bb}_{t}(a') - \fx{s}^{\T}\bb^{*2}(a')\right)}{\sigma(a')} \cdot \frac{\left(\fx{s}^{\T}\bb^{*2}(a') - \fx{s}^{\T}\hat{\bb}_{t}(a') \right)}{\sigma(a')} \\
    &= \frac{1}{2\sigma^2(a)} (\hat{\bb}_{t}(a) - \bb^{*2}(a)) D_{t}(a) (\hat{\bb}_{t}(a) - \bb^{*2}(a))) \\
    & \qquad + \frac{1}{2\sigma^2(a')} (\hat{\bb}_{t}(a') - \bb^{*2}(a')) D_{t}(a') (\hat{\bb}_{t}(a') - \bb^{*2}(a'))   \\
    &= \frac{\left(\fxx^{\T} \hat{\bb}_{t}(a) - \fxx^{\T} \hat{\bb}_{t}(a') + \delta \right)^2} {2\left(\sigma^2(a) \fxx^{\T} D_{t}^{-1}(a) \fxx + \sigma^2(a') \fxx^{\T} D_{t}^{-1}(a') \fxx \right)}.
\end{align*}
In addition, it is straightforward to see that
\begin{align*}
    Z_{a',a}^{L} (\x,t,\delta) 
    &= 
    - \log 
    \frac{
    \displaystyle
    \max_{\fxx^{\text{T}} \bb(a) \geq \fxx^{\text{T}} \bb(a') - \delta}
    \mathcal{L}_{\bb(a)}^{o} \mathcal{L}_{\bb(a')}^{o} 
    }{
    \displaystyle
    \max_{\fxx^{\text{T}} \bb(a) \leq \fxx^{\text{T}} \bb(a') - \delta} \mathcal{L}_{\bb(a)}^{o} \mathcal{L}_{\bb(a')}^{o} } 
    = 
    - Z_{a,a'}^{L} (\x,t,\delta). \qed
\end{align*}    

\subsection{Proof of Lemma 3}
\begin{lemma}[Gaussian mixture martingale for OLS estimator]   \label{lemma-mixture martingale linear}
    Let $\hat{\bb}_{t}$ denote the OLS estimator of $\bb$ based on observations up to stage $t$, and let $S_t^2$ denote the sample variance of the noise, i.e.,
    \begin{equation*}
        S_t^2 = \frac{1}{t-d} \sum_{i=1}^{t} \left(Y_i - \fx{i}^{\T} \hat{\bb}_t \right)^2.
    \end{equation*}
    Define $D_t = \sum_{i=1}^{t} \fx{i}^{\T} \fx{i}$ and $\Sigma_{t} = \bm{f}^{\T} D_{t}^{-1} \bm{f}$, where $\bm{f}\in\mathbb{R}^d$ denotes an arbitrary vector. For any $s>0$, the process $\{G_t^L\}_{t\geq 1}$ defined by 
    \begin{equation}   \label{equa-test martingale linear}
        G_t^L 
        = 
        \sqrt{\frac{s^2}{s^2 + \Sigma_{t}^{-1}}} 
        \left(
        \frac{
        \big(s^2 + \Sigma_{t}^{-1} \big) (t - d) S_t^2
        +
        s^2 \big(\bm{f}^{\T} \hat{\bb}_{t} - \bm{f}^{\T} \bb \big)^2
        \Sigma_{t}^{-1}
        }{
        \big(s^2 + \Sigma_{t}^{-1} \big) (t - d) S_t^2
        +
        \big(s^2 + \Sigma_{t}^{-1} \big)
        \Sigma_{t}^{-1}
        \big(\bm{f}^{\T} \hat{\bb}_{t} - \bm{f}^{\T} \bb \big)^2
        } 
        \right)^{-\frac{t - d + 1}{2}}
    \end{equation}
    is a martingale with initial value one.
\end{lemma}

The proof of Lemma \ref{lemma-mixture martingale linear} is provided in Section \ref{sec-proof of lemma3} below. For $r\in\{1,2\}$, define the self-normalized deviations
$
V_{t,r}^{L}\;:=\; \frac{(\bm{f}^T \hat{\bb}_t - \bm{f}^T \bb)^2}{2 S_t^2 \Sigma_t},
$
such that $U_t^{L,+}=V_{t,1}^{L} + V_{t,2}^{L}$.

Next we find a time-uniform boundary for $V_{t,r}^{L}$. Fix a stream $r$ and let $\{G_t^{L(r)}\}_{t\geq 0}$ be the martingale of Lemma~\ref{lemma-mixture martingale linear} specialized to $s=1$ and $\bb=\bb_r$. Define the embedded process
$$
M_{t,r}^{L}\;:=\;G_{N_{t,r}}^{L(r)},\qquad t\geq 0.
$$
Then $\{M_{t,r}^{L}\}_{t\geq 0}$ is a test martingale with respect to the natural filtration of stream $r$. Further we have $M_{t,r}^{L}$ can be denoted as a function of $V_{t,r}^{L}$:
\begin{align*}
    M_{t,r}^{L} = \sqrt{\frac{1}{1 + \Sigma_{t,r}^{-1}}} 
        \left(
        \frac{
        \big(N_{t,r} - d + V_{t,r}^L\big) + (N_{t,r} - d) \Sigma_{t,r}^{-1}
        }{
        \big(N_{t,r} - d + V_{t,r}^L\big) \big(1 + \Sigma_{t,r}^{-1}\big)
        } 
        \right)^{-\frac{N_{t,r} - d + 1}{2}}.
\end{align*}

Then using Ville's maximal inequality, for any $\alpha\in(0,1)$, we have the following time-uniform deviation inequality for $M_{t,r}^L$,
\begin{align*}
    \mathbb{P} \left(\exists \,t\in\mathbb{N}^*, M_{t,r}^L\geq \frac{1}{\alpha} \right) &= \mathbb{P} \left(\exists \,t\in\mathbb{N}^*, \frac{(N_{t,r}-d) \Sigma_{t,r}^{-1}}{(\Sigma_{t,r}^{-1}+1)(V_{t,r}^L+N_{t,r}-d)}\leq \left(\frac{\Sigma_{t,r}^{-1}+1}{\alpha^2} \right)^{-\frac{1}{N_{t,r}-d+1}} - \frac{1}{\Sigma_{t,r}^{-1}+1} \right) \\
    &= \mathbb{P} \left(\exists \,t\in\mathbb{N}^*, V_{t,r}^L\geq \frac{1}{2} \gamma^L(N_{t,r},\Sigma_{t,r}^{-1},\alpha) \right) \leq \alpha,
\end{align*}
where let $\epsilon > 0$ be arbitrarily small and the function $\gamma^L: \mathbb{N}^* \times \mathbb{R}^* \times (0,1) \to (0,+\infty)$ is defined as 
\begin{equation}   \label{equa-gamma-L-appendix} 
    \gamma^{L} (t_1, t_2,\alpha) = \frac{(t_1 - d)t_2}{\max \left\{\left(\frac{\alpha^2}{t_2+1} \right)^{\frac{1}{t_1-d+1}} (t_2+1) - 1, \epsilon \right\} } - (t_1 - d).
\end{equation}

Next we consider the boundary for the two-summed deviation term. Similar to the proof for Lemma 1, define the product process
$$
M_t^{L,*}\;:=\;M_{t,1}^{L}\,M_{t,2}^{L},\qquad t\ge 0.
$$
Then the process $\{M_t^{L,*}\}_{t\geq 0}$ is a test martingale. Applying Ville's inequality gives
\begin{equation}\label{eq:ville-prod-L}
\mathbb{P}\left(\exists\, t\in\mathbb{N}^*:\ M_t^{L,*} \geq \frac{1}{\alpha}\right)\ \leq\ \alpha.
\end{equation}

For the martingale $G_{N_{t,r}}^{L(r)}$, let $N_{t,r} = n$, $\Sigma_{t,r}^{-1} = \lambda$ and $V_{t,r}^L = V$, then it can be written as
$$
G_{n}^{L(r)}(V)
=
\sqrt{\frac{1}{\lambda+1}}
\left(
\frac{(n-d)(\lambda+1) + V}{(\lambda+1)(n-d+V)}
\right)^{- \frac{n-d+1}{2}},
\qquad n\ge 1,\ V\ge 0,
$$
A direct calculation shows that for all $n\geq d+1$, $\log G_{n}^{L(r)}(V)$ is strictly concave in $V$ on $[0,\infty)$:
$$
\frac{d^2}{dV^2}\log G_{n}^{L(r)}(V)
=
-\frac{n-d+1}{2}\cdot
\frac{(n-d)\lambda}{\big(V+(n-d)(\lambda+1)\big)^2}
\left(
\frac{2}{V+(n-d)}+\frac{(n-d)\lambda}{\big(V+(n-d)\big)^2}
\right) \ <\ 0,
$$
Therefore, we have that
\begin{equation}\label{equa-endpoint min linear}
\min_{\substack{v_1,v_2\ge 0\\ v_1+v_2=v}}
G_{n_1}^{L(1)} (v_1)\,G_{n_2}^{L(2)} (v_2)
=
\min\big\{G_{n_1}^{L(1)} (v)\,G_{n_2}^{L(2)} (0),\ G_{n_1}^{L(1)} (0)\,G_{n_2}^{L(2)} (v)\big\}.
\end{equation}
Applying \eqref{equa-endpoint min linear} with $n_r=N_{t,r}$, $v_r=V_{t,r}$, $\lambda_r = \Sigma_{t,r}^{-1}$ and
$v=U_t^{+}$, we have, for each $t$,
\begin{equation*}\label{equa-prod lower linear}
M_t^{L,*}
=
G_{N_{t,1}}^{L(1)} (V_{t,1})\cdot G_{N_{t,2}}^{L(2)} (V_{t,2})
\ \ge\
\min\Big\{G_{N_{t,1}}^{L(1)} (U_t^{+})\,G_{N_{t,2}}^{L(2)} (0),\ G_{N_{t,1}}^{L(1)} (0)\,G_{N_{t,2}}^{L(2)} (U_t^{+})\Big\}.
\end{equation*}

For each $t$, define
$$
b_{t,1}^L
:=\frac12\,\gamma^L\!\left(N_{t,1}, \Sigma_{t,1}^{-1}, \alpha \sqrt{\frac{1}{N_{t,2}+1}}\right),
\quad
b_{t,2}^L
:=\frac12\,\gamma^L\!\left(N_{t,2}, \Sigma_{t,2}^{-1}, \alpha \sqrt{\frac{1}{N_{t,1}+1}}\right),
\quad
B_t:=\max\{b_{t,1}^L, b_{t,2}^L\}.
$$
Following the same calculations with the proof of Lemma 1, we have
$$
U_t^{L,+}>B_t^L \quad\Longrightarrow\quad M_t^{L,*} \ge \frac{1}{\alpha}.
$$
Therefore,
$$
\mathbb{P}\left(\exists\, t\in\mathbb{N}^*:\ U_t^{L,+}>B_t^L\right)
\ \le\
\mathbb{P}\left(\exists\, t\in\mathbb{N}^*:\ M_t^{L,*}\ge \frac{1}{\alpha}\right)
\ \le\ \alpha,
$$
where the last inequality follows from \eqref{eq:ville-prod-L}.  Equivalently,
with probability at least $1-\alpha$, for all $t\in\mathbb{N}^*$,
$$
U_t^{L,+}
\leq
\max \left\{\frac{1}{2} \gamma^{L} \left(N_{t,1}, \Sigma_{t,1}^{-1}, \alpha \sqrt{\frac{1}{\Sigma_{t,2}^{-1}+1}}\right), 
\frac{1}{2} \gamma^{L} \left(N_{t,2}, \Sigma_{t,2}^{-1}, \alpha \sqrt{\frac{1}{\Sigma_{t,1}^{-1}+1}}\right) \right\},
$$
which is exactly (15). This completes the proof.  \qed

\subsection{Proof of Lemma \ref{lemma-mixture martingale linear}}   \label{sec-proof of lemma3}
We first present the \textit{scale-invariant} technique that will be used to handle the unknown variance parameter, following \cite{wang2025anytime}.

A function $h : \mathbb{R}^t \to \mathbb{R}$ is called \emph{scale-invariant} if it is measurable and, for any $x_1, \ldots, x_t \in \mathbb{R}$ and $\lambda > 0$,
$$
h(x_1, \ldots, x_t) = h(\lambda x_1, \ldots, \lambda x_t).
$$
For any $t \ge 1$, define
$$
\mathcal{F}_t^* = \sigma\big(h(X_1, \ldots, X_t)\big),
$$
which we call the \textit{scale-invariant filtration} of data $X_1, \ldots, X_t$.  
If $X_1 \neq 0$, then equivalently,
$$
\mathcal{F}_t^* = \sigma\!\left(\frac{X_1}{|X_1|}, \frac{X_2}{|X_1|}, \ldots, \frac{X_t}{|X_1|}\right).
$$

Let $g_{\mu, \sigma^2}$ denote the probability density function of a distribution on $\mathbb{R}$ parameterized by $\mu \in \mathbb{R}$ and $\sigma > 0$. The following lemma provides a fundamental tool for constructing martingales for hypothesis testing.

\begin{lemma} (Lemma 4.2. in \cite{wang2025anytime})
For any $\theta\in\mathbb{R}$ and $\theta_0\in\mathbb{R}$, the process
$$
M_t(\theta, \theta_0) = 
\frac{\displaystyle \int_{0}^{\infty} \left\{\prod_{s=1}^{t} g_{\omega\theta, \omega^2}(X_s)\right\} \, \frac{d\omega}{\omega} } 
{\displaystyle \int_{0}^{\infty} \left\{\prod_{s=1}^{t} g_{\omega\theta_0, \omega^2}(X_s)\right\} \, \frac{d\omega}{\omega}}
$$
is a martingale with respect to $\{\mathcal{F}_t^*\}_{t \ge 1}$ under all probability measures induced by $\{g_{\mu_0, \sigma_0^2} : \mu_0 / \sigma_0 = \theta_0\}$.
\end{lemma}

This lemma relates the traditional likelihood ratio martingales with respect to $\mathcal{F}_t$ to those with respect to scale-invariant filtration $\mathcal{F}_t^*$. Under this reduction, the resulting martingale depends on the parameters only through the ratio $\mu / \sigma$.

Next we consider the linear models in our setting. Consider a context-action pair $(\x,a)$, to track the time-uniform behavior of $\fxx^{\T} \hat{\bb}_{t}(a)$ instead of $\hat{\bb}_{t}(a)$, we propose to separate the scalar quantity $\fxx^{\T} \bb(a)$ from the linear model. At stage $t$, the observed sample satisfies the standard linear model
\begin{equation}   \label{equa-yt linear assumption}
    Y_t = \fx{t}^{\T} \bb(a_t) + \varepsilon_t.
\end{equation}

Let $\bm{f}\in\mathbb{R}^d$ denote an arbitrary vector. For each action $a\in\mathcal{K}$, we decompose $\bb(a)$ into two components: one along the direction of $\bm{f}$ and the other in the subspace orthogonal to $\bm{f}$. Let $C\in\mathbb{R}^{d\times(d-1)}$ be an orthogonal complement of $\bm{f}$ such that $C^{\T}\bm{f} = 0$, $\operatorname{rank}(C) = d-1$ and $C^{\T}C = I_{d-1}$. Such a matrix can be easily obtained computationally, for example, using the "null" function in MATLAB. Define the transformation matrix $T = [\bm{f}~C]^{\T} \in \mathbb{R}^{d\times d}$. It is straightforward to verify that the inverse of $T$ is given by $T^{-1} = \left[\frac{\bm{f}}{\bm{f}^{\T}\bm{f}} ~ C \right]$. We omit the subscript "$a_t$" afterwards, since the result holds for all $a\in\mathcal{K}$. Hence, the model in (\ref{equa-yt linear assumption}) is equivalently to
\begin{equation}   \label{equa-yt linear assumption new}
    Y_t = p_{t} \eta + \mathbf{q}_{t}^{\T} \bm{\zeta} + \varepsilon_t,
\end{equation}
where $p_{t} = \frac{\bm{f}^{\T} \fx{t}}{\bm{f}^{\T}\bm{f}}$, $\eta = \bm{f}^{\T} \bb$, $\mathbf{q}_{t} = C^{\T} \fx{t}$ and $\bm{\zeta} = C^{\T} \bb$. 
This transformation makes the scalar projection $\eta=\bm f^{\top}\bb$ explicit and enables us to construct a test martingale that directly tracks the deviation of $\bm f^{\top}\hat{\bb}_t$.

We consider the matrix form. For the $t$-th sample, the collected $t$ samples can be written in vector as
\begin{equation}   \label{equa-yt linear vector}
    \bm{Y}_t = P_t^{\T} \eta + Q_t^{\T} \bm{\zeta} + \bm{\varepsilon}_t, 
\end{equation}
where $P_t \in \mathbb{R}^{1\times t}$, $Q_t \in \mathbb{R}^{(d-1) \times t}$, and $\bm{\varepsilon}_t \in \mathbb{R}^{1\times t}$.

Let $L_{\eta,\sigma^2}(\bm{Y}_t)$ denote the likelihood of samples $Y_1, \ldots, Y_t$ after integrating out the parameter $\bm{\zeta}$. The following lemma provides the closed form of this marginal likelihood.

\begin{lemma}
Under the linear model (\ref{equa-yt linear vector}), we have
$$
L_{\eta,\sigma^2}(\bm{Y}_t)
= 
\frac{1}{\sqrt{(2\pi\sigma^2)^{t-d+1} \det(Q_t Q_t^{\T})}}
\exp\!\left(
    -\frac{1}{2\sigma^2}
    \big\|
        R_{Q_t}
        (\bm{Y}_t - P_t^{\T}\eta)
    \big\|^2
\right),
$$
where $R_{Q_t} = (I_t - H_{Q_t})$, $H_{Q_t} = Q_t^{\T}(Q_t Q_t^{\T})^{-1}Q_t$ and $I_t$ is the identity matrix of dimension $t$.
\end{lemma}

\textbf{Proof.}  
We take a flat prior to marginalize $\bm{\zeta}$. Since $\varepsilon_1, \ldots, \varepsilon_t$ are i.i.d. from the distribution $\mathcal{N}(0, \sigma^2)$, we have
\begin{align*}
    L_{\eta,\sigma^2}(\bm{Y}_t) 
    &= 
    \int_{\mathbb{R}^{d-1}} 
    \left\{ 
    \prod_{s=1}^t \frac{1}{\sqrt{2\pi}\sigma} 
    \exp\!\left(-\frac{(Y_s - P_s^T\eta - Q_s^T \bm{\zeta})}{2\sigma^2}\right)
    \right\} 
    \,d\bm{\zeta}  \\
    &= 
    \frac{1}{(2\pi \sigma^2)^{t/2}} 
    \int_{\mathbb{R}^{d-1}} 
    \exp\!\left(
    -\frac{1}{2\sigma^2}\| Q_t^{\T} \xi - (\bm{Y}_t - P_t^T\eta)\|^2\right) \,d\bm{\zeta}.
\end{align*}

Let $V_t = \bm{Y}_t - P_t^T \eta$. Since
\begin{align*}
    &(Q_t^T \xi - V_t)^T (Q_t^T \bm{\zeta} - V_t)   \\
    = &\bm{\zeta}^T Q_t Q_t^T \bm{\zeta} 
    - 2 V_t^T Q_t^T \bm{\zeta} 
    + V_t^T Q_t^T (Q_t Q_t^T)^{-1} Q_t V_t
    + V_t^T V_t
    - V_t^T Q_t^T (Q_t Q_t^T)^{-1} Q_t V_t \\
    = &(\bm{\zeta} - (Q_t Q_t^T)^{-1} Q_t V_t)^T 
    Q_t^T Q_t 
    (\bm{\zeta} - (Q_t Q_t^T)^{-1} Q_t V_t) 
    + V_t^T R_{Q_t} V_t,
\end{align*}
we have
\begin{align*}
L_{\eta,\sigma^2}(\bm{Y}_t) 
&= \frac{1}{(2\pi \sigma^2)^{t/2}} 
\exp\!
\left( -\frac{1}{2\sigma^2} V_t^T R_{Q_t} V_t \right)  \\
&\quad \times \int_{\mathbb{R}^{d-1}} 
\exp\!
\left( -\frac{1}{2\sigma^2} 
(\bm{\zeta} - (Q_t Q_t^T)^{-1} Q_t V_t)^T 
Q_t Q_t^T 
(\bm{\zeta} - (Q_t Q_t^T)^{-1} Q_t V_t) \right) 
\,d\bm{\zeta}   \\
&= \frac{1}{(2\pi \sigma^2)^{t/2}} 
\exp\!
\left(-\frac{1}{2\sigma^2} V_t^T R_{Q_t} V_t \right) 
(2\pi)^{\frac{d-1}{2}} 
\det\left((Q_t Q_t^T)^{-1}\right)^{\frac{1}{2}}   \\
&= \frac{1}{\sqrt{(2\pi \sigma^2)^{t-d+1} \det(Q_t Q_t^T)}}
\exp\!
\left(-\frac{1}{2\sigma^2} \Big\|R_{Q_t} (\bm{Y}_t - P_t^T \eta) \Big\|^2 \right).   \qed
\end{align*}

The next lemma gives an explicit expression for $M_t(\theta,\theta_0)$ when $\theta_0 = 0$.

\begin{lemma}
    For any $\theta \in \mathbb{R}$, the process
\begin{equation}   \label{equa-appendix M_t}
M_t(\theta,0) 
= \frac{\exp\!\left(-\frac{1}{2}\theta^2 P_t R_{Q_t} P_t^T \right)}
{\Gamma\left(\frac{t-d+1}{2}\right)} 
\int_{0}^{\infty} u^{\frac{t-d-1}{2}} \exp\!\left( -u + \theta \bm{Y}_t^T R_{Q_t} P_t^T 
\sqrt{\frac{2u}{\bm{Y}_t^T R_{Q_t} \bm{Y}_t}} \right) \,du.
\end{equation}
\end{lemma}

\textbf{Proof.} First we have
$$
M_t(\theta,0) = \frac{\displaystyle \int_{0}^{\infty} L_{\omega\theta,\omega^2}(\bm{Y}_t) \frac{d\omega}{\omega}}
{\displaystyle \int_{0}^{\infty} L_{0,\omega^2}(\bm{Y}_t) \frac{d\omega}{\omega}}.
$$

Substituting $L_{\omega\theta,\omega^2}(\bm{Y}_t)$ and $\theta=\eta / \sigma$, we have
\begin{align*}
&\text{the numerator} \\
&= \int_{0}^{\infty} 
\frac{\displaystyle\exp\!\left(- \frac{1}{2\omega^2} 
\left\|R_{Q_t} \left(\bm{Y}_t - \frac{1}{\sigma} P_t^T \omega \eta \right) \right\|^2 \right)
}{
\displaystyle(2\pi\omega^2)^{\frac{n-d+1}{2}} \det(Q_t Q_t^T)^{\frac{1}{2}}}
\,\frac{d\omega}{\omega}   \\
&= \frac{1}{(2\pi)^{\frac{t-d+1}{2}} \det(Q_t Q_t^T)^{\frac{1}{2}}} 
\int_{0}^{\infty} \frac{1}{\omega^{t-d+2}}  
\exp\!\left( -\frac{1}{2\omega^2} \bm{Y}_t^T R_{Q_t} \bm{Y}_t 
+ \frac{\theta}{\omega} \bm{Y}_t^T R_{Q_t} P_t 
- \frac{\theta^2}{2} P_t R_{Q_t} P_t^T \right) \,d\omega   \\
&= \frac{\exp\!\left(-\tfrac{1}{2}\theta^2 P_t R_{Q_t} P_t^T\right)}{(2\pi)^{\frac{t-d+1}{2}} \det(Q_t Q_t^T)^{1/2}}
\int_{0}^{\infty} \frac{1}{\omega^{t-d+2}} 
\exp\!\left(-\tfrac{1}{2\omega^2} \bm{Y}_t^T R_{Q_t} \bm{Y}_t 
+ \tfrac{\theta}{\omega} \bm{Y}_t^T R_{Q_t} P_t\right) \,d\omega.   
\end{align*}

Similarly, substituting $L_{0,\omega^2}(\bm{Y}_t)$, we have
\begin{align*}
\text{the denominator} &= \frac{1}{(2\pi)^{\frac{t-d+1}{2}} \det(Q_t Q_t^T)^{\frac{1}{2}}} 
\int_{0}^{\infty} \frac{1}{\omega^{t-d+2}}  
\exp\!\left(-\frac{1}{2\omega^2} \bm{Y}_t^T R_{Q_t} \bm{Y}_t \right) \,d\omega   \\
&= \frac{1}{(2\pi)^{\frac{t-d+1}{2}} \det(Q_t Q_t^T)^{\frac{1}{2}}} 
\int_{\infty}^{0}
\frac{e^{-u}}{\left(\frac{\bm{Y}_t^T R_{Q_t} \bm{Y}_t}{2u} \right)^{t-d+2}}
\sqrt{\frac{\bm{Y}_t^T R_{Q_t} \bm{Y}_t}{2}} 
\left(-\frac{1}{2} u^{- \frac{3}{2}} \right) \,du   \\
&= \frac{1}{(2\pi)^{\frac{t-d+1}{2}} \det(Q_t Q_t^T)^{1/2}} 
\cdot \frac{1}{2}\left(\frac{\bm{Y}_t^T R_{Q_t} \bm{Y}_t}{2}\right)^{-\frac{t-d+1}{2}}
\Gamma\left(\frac{t-d+1}{2}\right),
\end{align*}
where the second equality is obtained by letting $\omega = \sqrt{\frac{\bm{Y}_t^T R_{Q_t} \bm{Y}_t}{2u}}$

Finally, we have
\begin{align*}
&M_t(\theta,0) \\
&= \frac{\exp\!\left(-\frac{1}{2}\theta^2 P_t R_{Q_t} P_t^T\right)
}{
\frac{1}{2}\left(\tfrac{\bm{Y}_t^T R_{Q_t} \bm{Y}_t}{2}\right)^{-\frac{t-d+1}{2}} 
\Gamma(\frac{t-d+1}{2})}
\int_{0}^{\infty} \frac{1}{\omega^{t-d+2}}  
\exp\left(-\frac{1}{2\omega^2}\bm{Y}_t^T R_{Q_t} \bm{Y}_t 
+ \frac{\theta}{\omega} \bm{Y}_t^T R_{Q_t} P_t^T\right) \,d\omega.   \\
&= \frac{\exp\!\left(-\frac{1}{2}\theta^2 P_t R_{Q_t} P_t^T\right)
}{
\frac{1}{2}\left(\frac{\bm{Y}_t^T R_{Q_t} \bm{Y}_t}{2}\right)^{-\frac{t-d+1}{2}} 
\Gamma(\frac{t-d+1}{2})}
\int_{0}^{\infty} \frac{u^{\frac{t-d-1}{2}}
}{
2\left(\frac{\bm{Y}_t^T R_{Q_t} \bm{Y}_t}{2}\right)^{\frac{t-d+1}{2}} 
}
\exp\!\left(-u 
+ \theta \bm{Y}_t^T R_{Q_t} P_t^T \sqrt{\frac{2u}{\bm{Y}_t^T R_{Q_t} \bm{Y}_t}} \right) \,du   \\
&= \frac{\exp\!\left(-\frac{1}{2}\theta^2 P_t R_{Q_t} P_t^T\right)
}{\Gamma(\frac{t-d+1}{2})}
\int_{0}^{\infty} u^{\frac{t-d-1}{2}}
\exp\!\left(-u 
+ \theta \bm{Y}_t^T R_{Q_t} P_t^T \sqrt{\frac{2u}{\bm{Y}_t^T R_{Q_t} \bm{Y}_t}}  \right).   \qed
\end{align*}

Let $\hat{\eta}_t$ be the OLS estimator of $\eta$ and $S_t^2$ the sample variance of $Y_1,\ldots,Y_t$. We have the following lemma.
\begin{lemma}
    For any $s>0$, the process
    \begin{equation}   \label{equa-appendix martingale G_t}
        G_t^{(s)} = \sqrt{\frac{s^2}{P_t R_{Q_t} P_t^T + s^2}} 
        \exp\!\left(\frac{(s^2 + P_t R_{Q_t} P_t^T)(t-d) S_t^2
        + s^2 \hat{\eta}_t P_t R_{Q_t} P_t^T
        }{
        (s^2 + P_t R_{Q_t}P_t^T)(t-d) S_t^2
        + (s^2 + P_t R_{Q_t}P_t^T) \hat{\eta}_t P_t R_{Q_t}P_t^T
        }   \right)
    \end{equation}
is a non-negative martingale with respect to $\{\mathcal{F}_t^*\}_{t \ge 1}$ under distribution $\mathcal{N}_{\eta=0}$.
\end{lemma}

\textbf{Proof.} Take a Gaussian prior $\mathcal{N}\left(0, \frac{1}{s^2} \right)$ on $\theta = \frac{\eta}{\sigma}$ such that 
$$
\frac{d \pi(\theta)}{d \theta} = \frac{s}{\sqrt{2\pi}} e^{-\frac{1}{2} s^2\theta^2},
$$
then define the martingale $G_t^{(s)}$ as 
$$
G_t^{(s)} = \int_{\mathbb{R}} M_t (\theta,0) \,\pi(d \theta).
$$

We will show that has an expression in (\ref{equa-appendix martingale G_t}). Substituting $M_t (\theta,0)$ by (\ref{equa-appendix M_t}), we have
\begin{align*}
G_t^{(s)} &= \int_{\theta\in\mathbb{R}}
\frac{\exp\!\left(-\frac{1}{2}\theta^2 P_t R_{Q_t} P_t^T \right)}
{\Gamma\left(\frac{t-d+1}{2}\right)} 
\int_{u>0} u^{\frac{t-d-1}{2}} e^{-u} 
\exp\!\left(\theta \bm{Y}_t^T R_{Q_t} P_t^T 
\sqrt{\frac{2u}{\bm{Y}_t^T R_{Q_t} \bm{Y}_t}} \right) \,du
\cdot \frac{s}{\sqrt{2\pi}} e^{-\frac{1}{2} s^2\theta^2}
\,d\theta   \\
&= \frac{s}{\sqrt{2\pi} \Gamma\left(\frac{t-d+1}{2}\right)}
\int_{u>0} u^{\frac{t-d-1}{2}} e^{-u}
\int_{\theta\in\mathbb{R}}
\exp\!\left(-\frac{1}{2}\theta^2 (P_t R_{Q_t} P_t^T + s^2)
+ \theta \bm{Y}_t^T R_{Q_t} P_t^T \sqrt{\frac{2u}{\bm{Y}_t^T R_{Q_t} \bm{Y}_t}} \right)
\,d\theta \,du.
\end{align*}

Since 
\begin{align*}
&\exp\!\left(-\frac{1}{2}\theta^2 (P_t R_{Q_t} P_t^T + s^2)
+ \theta \bm{Y}_t^T R_{Q_t} P_t^T \sqrt{\frac{2u}{\bm{Y}_t^T R_{Q_t} \bm{Y}_t}} \right)   \\
=&
\exp\!\left(- \frac{\left(\theta 
- \frac{\bm{Y}_t^T R_{Q_t} P_t^T \sqrt{\frac{2u}{\bm{Y}_t^T R_{Q_t} \bm{Y}_t}}
}{
P_t R_{Q_t} P_t^T + s^2} \right)^2
}{2(P_t R_{Q_t} P_t^T + s^2)^{-1}}
+ \frac{
\left(\bm{Y}_t^T R_{Q_t} P_t^T \sqrt{\frac{2u}{\bm{Y}_t^T R_{Q_t} \bm{Y}_t}} \right)^2
}{
2(P_t R_{Q_t} P_t^T + s^2)
} \right),
\end{align*}
we have
\begin{align*}
G_t^{(s)} &= 
\frac{s}
{\sqrt{2\pi} \Gamma\left(\frac{t-d+1}{2}\right)}
\int_{u>0} u^{\frac{t-d-1}{2}} e^{-u}
\sqrt{\frac{2\pi}{P_t R_{Q_t} P_t^T + s^2}}
\exp\!\left(\frac{
\left(\bm{Y}_t^T R_{Q_t} P_t^T \sqrt{\frac{2u}{\bm{Y}_t^T R_{Q_t} \bm{Y}_t}} \right)^2
}{
2(P_t R_{Q_t} P_t^T + s^2)
} \right) \,du   \\
&= \sqrt{\frac{s^2}{P_t R_{Q_t} P_t^T + s^2}}
\left(1 - \frac{(\bm{Y}_t^T R_{Q_t} P_t^T)^2}{
(P_t R_{Q_t} P_t^T + s^2) \bm{Y}_t^T R_{Q_t} P_t^T
}
\right)^{- \frac{t-d+1}{2}}.
\end{align*}

Since $\hat{\eta}_t = (P_t R_{Q_t} P_t^T)^{-1} P_t R_{Q_t} \bm{Y}_t$ and
\begin{align*}
    S_t^2 &= \frac{1}{t-d} 
    \bm{Y}_t^T (I_t - H_{[Q_t^T,P_t^T]}) \bm{Y}_t   \\
    &= \frac{1}{t-d} \left(\bm{Y}_t^T R_{Q_t} \bm{Y}_t
    - \frac{(\bm{Y}_t^T R_{Q_t} P_t^T)^2}{P_t R_{Q_t} P_t^T} \right),
\end{align*}
we have
\begin{align*}
G_t^{(s)} &= \sqrt{\frac{s^2}{P_t R_{Q_t} P_t^T + s^2}}
\left(\frac{(P_t R_{Q_t} P_t^T + s^2) \bm{Y}_t^T R_{Q_t} \bm{Y}_t
- (\bm{Y}_t^T R_{Q_t} P_t^T)^2}{
(P_t R_{Q_t} P_t^T + s^2) \bm{Y}_t^T R_{Q_t} \bm{Y}_t
}
\right)^{- \frac{t-d+1}{2}}   \\
&= \sqrt{\frac{s^2}{P_t R_{Q_t} P_t^T + s^2}}
\left(\frac{\frac{(t-d)S_t^2}{
(t-d)S_t^2 + \frac{\hat{\eta}_t}{P_t R_{Q_t} P_t^T}}
+ \frac{s^2}{P_t R_{Q_t} P_t^T}
}{
1 + \frac{s^2}{P_t R_{Q_t} P_t^T}
}
\right)^{- \frac{t-d+1}{2}}   \\
&= \sqrt{\frac{s^2}{P_t R_{Q_t} P_t^T + s^2}}
\left(\frac{
(s^2 + P_t R_{Q_t} P_t^T) (t-d) S_t^2 + s^2 \hat{\eta}_t^2 P_t R_{Q_t} P_t^T
}{
(s^2 + P_t R_{Q_t} P_t^T) (t-d) S_t^2 + (s^2 + P_t R_{Q_t} P_t^T) \hat{\eta}_t^2
} 
\right)^{- \frac{t-d+1}{2}}.   \qed
\end{align*}

Let us take a shifting argument $Y_t \rightarrow Y_t - p_{t} \eta_0$. Then we obtain the martingale $G_t^{(s,\eta_0)}$ defined by
\begin{equation}   \label{equa-appendix martingale G_t shift}
G_t^{(s,\eta_0)} = \sqrt{\frac{s^2}{P_t R_{Q_t} P_t^T + s^2}}
\left(\frac{
(s^2 + P_t R_{Q_t} P_t^T) (t-d) S_t^2 + s^2 (\hat{\eta}_t - \eta_0)^2 P_t R_{Q_t} P_t^T
}{
(s^2 + P_t R_{Q_t} P_t^T) (t-d) S_t^2 + (s^2 + P_t R_{Q_t} P_t^T) (\hat{\eta}_t - \eta_0)^2
} 
\right)^{- \frac{t-d+1}{2}}.
\end{equation}

\begin{lemma} (Frisch-Waugh-Lovell (FWL) Theorem, \cite{ding2021frisch}) \label{lemma-FWL}
    The coefficient of $X_2$ in the full ordinary least squares (OLS) fit of $Y$ on $X = (X_1, X_2)$ equals the coefficient of $\tilde{X}_2$ in the partial OLS fit of $\tilde{Y}$ on $\tilde{X}_2$, where $\tilde{Y}$ and $\tilde{X}_2$ are the residuals from the OLS fits of $Y$ and $X_2$ on $X_1$, respectively.
\end{lemma}

Using Lemma~\ref{lemma-FWL}, we have $\hat{\eta}_t = \bm{f}^T \hat{\bb}_t$. Since the variance of $\hat{\eta}_t$, given by
$$
\text{Var} (\hat{\eta}_t) = \frac{\sigma^2}{P_t R_{Q_t} P_t^T}
$$
equals to the variance of $\bm{f}^T \hat{\bb}_t$, given by
$$
\text{Var} (\hat{\eta}_t) = \sigma^2 \bm{f}^T D_t^{-1} \bm{f},
$$
we have 
$$
P_t R_{Q_t} P_t^T = \frac{1}{\bm{f}^T D_t^{-1} \bm{f}}.
$$

Substituting $\eta_0 = \bm{f}^T \bb_0$, $\hat{\eta}_t$ and $P_t R_{Q_t} P_t^T$ into (\ref{equa-appendix martingale G_t shift}), we obtain 
\begin{equation*}
G_t^{L} = \sqrt{\frac{s^2}{s^2 + (\bm{f}^T D_t^{-1} \bm{f})^{-1}}}
\left(\frac{
(s^2 + (\bm{f}^T D_t^{-1} \bm{f})^{-1}) (t-d) S_t^2 + s^2 (\bm{f}^T \hat{\bb}_t - \bm{f}^T \bb_0)^2 (\bm{f}^T D_t^{-1} \bm{f})^{-1}
}{
(s^2 + (\bm{f}^T D_t^{-1} \bm{f})^{-1}) (t-d) S_t^2 + (s^2 + (\bm{f}^T D_t^{-1} \bm{f})^{-1}) (\bm{f}^T \hat{\bb}_t - \bm{f}^T \bb_0)^2
} 
\right)^{- \frac{t-d+1}{2}},
\end{equation*}
which is a test martingale for distribution $\mathcal{N}_{\bm{f}^T \bb_ = \bm{f}^T \bb_0}$.   \qed

\subsection{Proof of Theorem 2}

The proof is identical to that of Theorem~1 after replacing the unstructured pairwise quantities by their linear counterparts. By Lemma~3 and the definitions of $\varphi_{a,a'}^{\one,L}$ and $\varphi_{a,a'}^{\two,L}$, we have, for every $\x\in\mathcal X$ and every $a\in\mathcal A(\x)\setminus\{\pi^*(\x)\}$,
\[
\mathbb P\!\left(
\exists\, t\in\mathbb N^*:\ 
U_t^{L}(\x;a,\pi^*(\x))
>
\varphi_{a,\pi^*(\x)}^{\one,L}(\bm N_t,\alpha,\x)
\right)
\le
\frac{\alpha}{\big(|\mathcal A(\x)|-1\big)mp(\x)},
\]
and
\[
\mathbb P\!\left(
\exists\, t\in\mathbb N^*:\ 
U_t^{L}(\x;a,\pi^*(\x))
>
\varphi_{a,\pi^*(\x)}^{\two,L}(\bm N_t,\alpha,\x)
\right)
\le
\frac{\alpha}{\big(|\mathcal A(\x)|-1\big)m}.
\]
Moreover, whenever $y(\x,a)\le y(\x,\pi^*(\x))-\eta$, we have
$\tilde Z_{a,\pi^*(\x)}^{L}(\x,t,\eta) \le U_t^{L}(\x;a,\pi^*(\x))$ for all stages at which the statistic is well defined. Therefore, $\Pre{1}$ follows from exactly the same argument as in the proof of Theorem~1 for $\Pre{1}$, yielding $\Pre1\ge 1-\alpha$.

For $\Pre{2}$, for all $a\in\mathcal A(\x)\setminus\{\pi^*(\x)\}$, let $\Delta^{L}(\x,a):=y(\x,\pi^*(\x))-y(\x,a)$, and define
\[
\mathcal G^{\two,L}
:=
\bigcap_{\x\in\mathcal X}
\bigcap_{a\in\mathcal A(\x)\setminus\{\pi^*(\x)\}}
\left\{
\forall\, t\in\mathbb N^*:\ 
\tilde Z_{a,\pi^*(\x)}^{L}\big(\x,t,\Delta^{L}(\x,a)\big)
\le
\varphi_{a,\pi^*(\x)}^{\two,L}(\bm N_t,\alpha,\x)
\right\}.
\]
Using a union bound, we have $\mathbb P(\mathcal G^{\two,L})\ge 1-\alpha$. On the event $\mathcal G^{\two,L}$, the same monotonicity argument as in Theorem~1 for $\Pre{2}$ yields
\[
y(\x,\pi^*(\x))-y(\x,\hat\pi_t(\x))
\le
r^{L}_t (\x),
\qquad
\forall\,\x\in\mathcal X,\ \forall\, t\in\mathbb N^*.
\]
Hence, by the stopping rule (13), we have $\sum_{\x\in\mathcal X}p(\x)\,r^{L}_{\tau_{\alpha,\delta}^{\two,L}}(\x)\leq \delta$, and therefore on $\mathcal G^{\two,L}$,
\[
V(\pi^*)-V(\hat\pi_{\tau_{\alpha,\delta}^{\two,L}})
\le
\sum_{\x\in\mathcal X}p(\x)\,r^{L}(\x,\tau_{\alpha,\delta}^{\two,L})
\le
\delta.
\]
Thus, we have $\Pre2 = \mathbb P\big(V(\hat\pi_\tau)\ge V(\pi^*)-\delta\big) \ge \mathbb P(\mathcal G^{\two,L}) \ge 1-\alpha$.
\qed

\subsection{Proof of Theorem 3}
First we show that the GLR statistics increase with the sampling stage $t$ linearly. By Assumption 4, we have, for all $\x \in \mathcal{X}$ and $a \in \mathcal{A}(\x)$,
$$
\Sigma_t^{-1}(\x,a) \leq \frac{N_t(a)\,U}{\fxx^{\T} \fxx}.
$$
Let $\xi_1(\x)$ denote the constant $\frac{U}{\fxx^{\T} \fxx}$. Then we have, for all $a,a' \in \mathcal{A}$,
\begin{align*}
    \tilde Z^{L}_{a,a'}(\x,t,\delta)
    & \geq
    \frac{\big(\fxx^{\T} \hat\bb_t(a) - \fxx^{\T} \hat\bb_t(a') + \delta\big)^2}
    {2\Big(
    \frac{S_t^2(a)\,\xi_1(\x)}{N_t(a)}
    +
    \frac{S_t^2(a')\,\xi_1(\x)}{N_t(a')}
    \Big)} = 
    \frac{t}{k}
    \frac{\big(\fxx^{\T} \hat\bb_t(a) - \fxx^{\T} \hat\bb_t(a') + \delta\big)^2}
    {2\Big(
    S_t^2(a)\,\xi_1(\x)
    +
    S_t^2(a')\,\xi_1(\x)
    \Big)}.
\end{align*}

Define the constants $\Gamma_{a}(\x) \in (0,+\infty)$ for all $\x \in \mathcal{X}$ and $a\in\mathcal{A}(\x)\setminus \pi^*(\x)$ as
\begin{equation*}
    \Gamma_{a}(\x) := \frac{\big(\fxx^{\T} \bb(\pi^*(\x)) - \fxx^{\T} \bb(a) \big)^2}
    {2\Big(
    \sigma^2(a)\,\xi_1(\x)
    +
    \sigma^2(a')\,\xi_1(\x)
    \Big)}.
\end{equation*}

Let $\Delta_{\min} =  \min_{\x\in\mathcal{X}} \min_{a\in\mathcal{A}(\x)\setminus \pi^*(\x)} \left\{\fxx^{\T} \bb(\pi^*(\x)) - \fxx^{\T} \bb(a) \right\} > 0$. Fix a context $\x \in \mathcal{X}$. Given $\epsilon_1 > 0$, there exist $\epsilon_2 \in (0,\frac{\Delta_{\min} + \delta}{4}]$ and $T^{\epsilon_1}$ such that for all $t \geq k T^{\epsilon_1}$, $\big|\fxx^{\T} \hat\bb_t(a) - \fxx^{\T} \bb(a)\big| \leq \epsilon_2$ and $|S_t^2(a) - \sigma^2(a)| \leq \epsilon_2,\;\forall\,a\in\mathcal{A}(\x)$ can imply that 
\begin{equation*}
    \tilde Z^{L}_{a,\pi^*(\x)}(\x,t,\delta) \geq \frac{t}{k}\Big(\Gamma_{a}(\x) - \epsilon_1 \Big),\;\forall\,a\in\mathcal{A}(\x)\setminus \pi^*(\x).
\end{equation*}
Moreover, we have $\mathbb{E}[T^{\epsilon_1}] < \infty$.

Now let us consider the boundaries $\varphi_{a,a'}^{\one,L}$. Note that the function $\gamma ^{L}(t_1, t_2,\alpha)$ in (16) begins nontrivial when
\begin{equation*}
    \left(\frac{\alpha^2 }{t_2+1} \right)^{\frac{1}{t_1-d+1}} (t_2+1) - 1 > 0.
\end{equation*}

Define the function
\begin{equation*}
h\!\left(\Sigma_t^{-1}(\x,a)\right)
:=
\left(N_t(a)-d\right)
\log\left(\Sigma_t^{-1}(\x,a) + 1\right)
\end{equation*}
Therefore, the boundaries $\varphi_{a,a'}^{\one,L}$ become active when for all $\x\in\mathcal{X}$ and for all $a, a'\in\mathcal{A}(\x)$,
\begin{equation*}
    h\!\left(\Sigma_t^{-1}(\x,a)\right) > \log\left(\frac{(k-1)^2(mp(\x))^2}{\alpha^2 } \right) + \log\left(\Sigma_t^{-1}(\x,a') + 1\right).
\end{equation*}

Define the random initial stage 
\begin{equation*}
    T^0 := \inf \left\{t\in\mathbb{N}^*:\,h\!\left(\Sigma_t^{-1}(\x,a)\right) > \log\left(\frac{(k-1)^2(mp(\x))^2 (\Sigma_t^{-1}(\x,a') + 1)}{\alpha^2 } \right), \forall\,\x\in\mathcal{X}, \forall\,a\in\mathcal{A}(\x) \right\}. 
\end{equation*}
By Assumption 3, we have that, given $\epsilon_3 > 0$, there exists $T^{\epsilon_3}$ such that for all $t \geq k T^{\epsilon_3}$,
\begin{equation*}
    \Sigma_t^{-1}(\x,a) \geq \frac{t}{k}\,p^{s}(\x^{\circ}) \fxx^{\T} \Sigma^{-1} \fxx, \quad \forall\,\x\in\mathcal{X}, \forall\,a\in\mathcal{A}(\x).
\end{equation*}
Moreover, we have $\mathbb{E}[T^{\epsilon_3}] < \infty$. 

Let $\xi_2 := \min_{\x\in\mathcal{X}} \left\{p^{s}(\x^{\circ}) \fxx^{\T} \Sigma^{-1} \fxx \right\}$. Since $h(\cdot)$ is increasing, we obtain
\begin{equation*}
    T^0 \leq \max\big\{k T^{\epsilon_3}, T^{\gamma}\big\},
\end{equation*}
where $T^{\gamma}$ is defined by
\begin{equation}   \label{equa-appendix-T-gamma}
    \left(\frac{T^\gamma}{k} - d -1 \right)\log\left(\frac{\xi_2 T^\gamma}{k} + 1\right) = \log\left(\frac{(k-1)^2(mp(\x))^2}{\alpha^2 }\right).
\end{equation}

Now let us consider the two scenarios $\alpha \to 0$ and $k \to \infty$, respectively.

(1) As $\alpha \to 0$, from the equation (\ref{equa-appendix-T-gamma}), we have $T^\gamma \to \infty$ and $\lim\limits_{\alpha \to 0} \frac{T^\gamma}{\log(1/\alpha)} = 0$. Let $M^{\epsilon} = \max\Big\{T^\gamma, k T^{\epsilon_1}, k T^{\epsilon_3}, 1/\epsilon_1^2\Big\}$. Now we consider two cases:

\textit{(i)} For all $\x\in\mathcal{X}$ and for all $a\in\mathcal{A}(\x)$, there exists $t \leq M^{\epsilon}$ such that $\tilde Z^L_{a,\pi^*(\x)}(\cdot) > \varphi^{\one,L}_{a,\pi^*(\x)}(\cdot)$. 

In this case we have $\tau^{\one,L}_{\alpha,\delta} \le M^{\epsilon}$.

\textit{(ii)} There exist $\x\in\mathcal{X}$ and $a\in\mathcal{A}(\x)$ such that for all $t \leq M^{\epsilon}$, $\tilde Z^L_{a,\pi^*(\x)}(\cdot) \leq \varphi^{\one,L}_{a,\pi^*(\x)}(\cdot)$.

Consider two actions $a$ and $a'$. Since $M^{\epsilon} \geq T^{\gamma} = \infty$ when $\alpha \to 0$, for all $t\geq M^{\epsilon}$, we have
\begin{align*}
    & \gamma^{L}\!\left(N_t(a), \Sigma_t^{-1}(\x,a), \frac{\alpha}{(k-1)mp(\x)} \sqrt{\frac{1}{\Sigma_t^{-1}(\x,a')+1}}\right)  \\
    =\ & \frac{N_t(a)} {\left( \frac{\alpha^2 } {(k-1)^2 \Sigma_t^{-1}(\x,a) (mp(\x))^2 (\Sigma_t^{-1}(\x,a')+1)} \right)^{\frac{1}{N_t(a)}}} - N_t(a) \\
    \leq\ & \frac{N_t(a)}{1 - \frac{1}{N_t(a)} \log\!\left(\frac{(k-1)^2 \xi_1(\x) N_t(a) (mp(\x))^2 (\xi_1(\x) N_t(a')+1)}{\alpha^2 }\right)} - N_t(a) \\
    =\ & \frac{\log\!\left(\frac{(k-1)^2 \xi_1(\x) N_t(a) (mp(\x))^2 (\xi_1(\x) N_t(a')+1)}{\alpha^2 }\right)}{
    1 - \frac{1}{N_t(a)} \log\!\left(\frac{(k-1)^2 \xi_1(\x) N_t(a) (mp(\x))^2 (\xi_1(\x) N_t(a')+1)}{\alpha^2 }\right)},
\end{align*}
where the inequality is due to the fact that $\Sigma_t^{-1}(\x,a) \leq \xi_1(\x) N_t(a)$ and $e^{-u} = 1-u + \mathcal{O}(u^2)$ as $u\to0$.

For $(\x,a)$ with $\tilde Z^L_{a,\pi^*(\x)}(\cdot) \leq \varphi^{\one,L}_{a,\pi^*(\x)}(\cdot)$, let $T^r$ be the solution of
\begin{align*}
    (\Gamma_{a}(\x) - \epsilon_1) \frac{T^r}{k}
    =\max\Bigg\{~
    &\frac{1}{2}\,\gamma^{L} \!\left(
    \frac{T^r}{k}, ~\Sigma_{t}^{-1}(\x,a), ~
    \frac{\alpha}{(k-1)mp(\x)} \sqrt{\frac{1}{\Sigma_{t}^{-1}(\x,\pi^*(\x))+1}}
    \right),\\
    &\frac{1}{2}\,\gamma^{L} \!\left(
    \frac{T^r}{k}, ~\Sigma_{t}^{-1}(\x,\pi^*(\x)), ~
    \frac{\alpha}{(k-1)mp(\x)} \sqrt{\frac{1}{\Sigma_{t}^{-1}(\x,a)+1}}
    \right)
    ~\Bigg\}.
\end{align*}

Since $T^r > M^{\epsilon}$, we have
\begin{equation}   \label{equa-appendix-Tr-alpha}
\begin{aligned}
&\Big(\Gamma_{a}(\x) - \epsilon_1\Big)\frac{T^r}{k} 
- \frac{1}{2} \Big(\Gamma_{a}(\x) - \epsilon_1 + 1\Big)\log\Big(T^r\big(\xi_1(\x) T^r + k\big) \Big) \\
\leq&\quad \frac{1}{2} \Big(\Gamma_{a}(\x) - \epsilon_1 + 1\Big)
\left[\log\left(\xi_1(\x) (mp(\x))^2\right) 
+ \log\frac{(k-1)^2}{k^2} 
+ \log\left(\frac{1}{\alpha^2} \right) \right]
\end{aligned}
\end{equation}

Combining the two cases yields
$$
\mathbb{E}[T^L_{\alpha,\delta}] \le
\mathbb{E}[M^\varepsilon] + \mathbb{E}[T^r].
$$

Since $\lim\limits_{\alpha \to 0} \frac{T^\gamma}{\log(1/\alpha)} = 0$, $\mathbb{E}[T^{\epsilon_1}] < \infty$, $\mathbb{E}[T^{\epsilon_3}] < \infty$ and let $\epsilon_1 \to 0$, we have
\begin{equation*}
    \limsup_{\alpha \to 0}
    \frac{\mathbb{E}[\tau^{\one,L}_{\alpha,\delta}]}{\log(1/\alpha)} = \frac{k(\Gamma^*+1)}{\Gamma^*},
\end{equation*}
where $\Gamma^* = \min\limits_{\x\in\mathcal{X}, a\in\mathcal{A}(\x)} \Gamma_{a}(\x)$.

(2) As $k \to \infty$, the proof is similar to that when $\alpha \to 0$. We only list the changes here.

When $k \to \infty$, $T^{\gamma} \to \infty$ and $\lim\limits_{k \to \infty} \frac{T^\gamma}{k\log k} = 0$. The RHS of inequality in (\ref{equa-appendix-Tr-alpha}) becomes a constant denoted by $C$. Since $T^r \geq k$, (\ref{equa-appendix-Tr-alpha}) becomes
\begin{equation}   \label{equa-appendix solve Tr}
    \Big(\Gamma_{a}(\x) - \epsilon_1\Big)\frac{T^r}{k} 
- \Big(\Gamma_{a}(\x) - \epsilon_1 + 1\Big)\log(T^r) \leq B + C,
\end{equation}
where the constant $B = \frac{1}{2} \Big(\Gamma_{a}(\x) - \epsilon_1 + 1\Big) \log\Big(\xi_1(\x) + 1\Big)$. 

Solve the inequality (\ref{equa-appendix-Tr-alpha}). We have
\begin{equation*}
    T^r \leq - \frac{k\,\Big(\Gamma_{a}(\x) - \epsilon_1 + 1\Big)}{\Big(\Gamma_{a}(\x) - \epsilon_1\Big)} W_{-1} \left(- \frac{\Big(\Gamma_{a}(\x) - \epsilon_1\Big)}{k\,\Big(\Gamma_{a}(\x) - \epsilon_1 + 1\Big)} \exp\left(- \frac{B+C}{\Big(\Gamma_{a}(\x) - \epsilon_1 + 1\Big)}\right) \right),
\end{equation*}
where $W_{-1}(-x)$ for $0 < x \leq \frac{1}{e}$ is the Lambert W function and $W_{-1}(-x) = \log(x) - \log\log(x) + \mathcal{O}\Big(\frac{\log\log(x)}{\log(x)} \Big)$ as $x\to 0$. Therefore, as $k \to +\infty$, we have
$$
T^r = \mathcal{O}(k\log(k)).$$

Since $\lim\limits_{k \to \infty} \frac{T^\gamma}{k\log k} = 0$, $\lim\limits_{k \to \infty} \frac{\mathbb{E}[k T^{\epsilon_1}]}{k\log k} = 0$ and $\lim\limits_{k \to \infty} \frac{\mathbb{E}[k T^{\epsilon_3}]}{k\log k} = 0$, we have
\begin{equation*}
    \limsup_{k\to\infty}
    \frac{\mathbb{E}[\tau^{\one,L}_{\alpha,\delta}]}{k\log (k)}
    = \tilde{C},
\end{equation*}
where $\tilde{C}$ is a constant.

For $\Pre2$, we introduce the auxiliary stopping time $\bar\tau_{\alpha,\delta}^{\two,L}$ obtained by requiring the stronger condition $r^L_t(\x)\le \delta$ for every context $\x\in\mathcal X$. Since this implies $\sum_{\x\in\mathcal X}p(\x)\,r^L(\x,t)\le \delta$, we have $\tau_{\alpha,\delta}^{\two,L}\le \bar\tau_{\alpha,\delta}^{\two,L}$ and hence $T_{\two}\le \mathbb E[\bar\tau_{\alpha,\delta}^{\two,L}]$. The proof for $\mathbb E[\bar\tau_{\alpha,\delta}^{\two,L}]$ is identical to that for $T_{\one}$ after replacing the boundaries $\varphi_{a,\pi^*(\x)}^{\one,L}(\cdot)$ by $\varphi_{a,\pi^*(\x)}^{\two,L}(\cdot)$. Consequently, $\mathbb E[\bar\tau_{\alpha,\delta}^{\two,L}]$ satisfies the same order bounds as $T_{\one}$, which implies the stated bounds for $T_{\two}$.  \qed

\section{Details of Numerical Experiments}
This appendix presents the details of synthetic data used in Section 6.1.
\subsection{Benchmark functions used in Section 6.1.1}
The benchmark functions used for generate synthetic cases for the agnostic setting are defined as follows.
\begin{itemize}
    \item \textit{Toy function:} $y(\x^{j}, a^{i}) = |i-j|\big(0.1+0.1(j-1) \big), \,j\in\{1,...,m\}, \,i\in\{1,...,k\}$. For each action–context pair $(a^{i}, \x^{j})$, samples are  independently drawn from a normal distribution $\mathcal{N}(y(\x^{j}, a^{i}), \sigma^2 (\x^{j}, a^{i}))$, where $\sigma (\x^{j}, a^{i}) = 0.1 + 0.1(i-1)+ 0.1(j-1)$. We consider $k = 10$ actions and $m = 10$ uniformly distributed contexts. 
    \item \textit{Matyas function:} $y(x, a) = 0.26(x^2+a^2 )-0.48xa, \,x\in\mathcal{X}, \,a\in\mathcal{A}$. The sample standard deviation $\sigma (\x, a) = 1$. We consider the action space $\mathcal{A} = \{-10, -5, 0, 5, 10\}$ and the context space $\mathcal{X} = \{0, 0.5, 1, 1.5, 2, 2.5, 3\}$ such that $k = 5$ and $m = 7$. The context probabilities $p(x), x\in\mathcal{X}$ are randomly generated from $\mathcal{U}(0,1)$ and normalized to satisfy $\sum_{x\in\mathcal{X}} p(x) = 1$.
    \item \textit{Dixon-Price function:} $y(x, a) = (a_1-x_1)^2 + \sum_{l=2}^d l\Big(2(a_l - x_l)^2 -(a_{l-1} - x_{l-1})\Big)^2, \,x\in\mathcal{X}, \,a\in\mathcal{A}$. The sample standard deviations $\sigma (\x, a)$ are randomly generated from $\mathcal{U}(0.5, 2)$. We consider the two dimensional case ($d=2$) with $k = 9$ actions $\mathcal{A} = \{0,0.8,1.6\} \times \{0,0.8,1.6\}$ and $m = 25$ contexts $\mathcal{X} = \{-0.2,-0.1,0,0.1,0.2\} \times \{-0.2,-0.1,0,0.1,0.2\}$. The context probabilities $p(x), x\in\mathcal{X}$ are randomly generated from $\mathcal{U}(0,1)$ and then normalized.
\end{itemize}

\subsection{Random cases used in Section 6.1.2}
The settings used to generate the synthetic random cases under the linear setting are summarized in Table~EC.1. Across all cases, each dimension of the context vector is evenly spaced over the interval $[0,1]$, and the design points take values of 0 or 1 in each non-intercept dimension. Unless otherwise specified, the context probability distribution is uniform. In Cases~4 and~5, the context probabilities are instead randomly generated to introduce heterogeneity across contexts.

For each case, Table~EC.1 reports the number of actions ($k$), the context dimension ($d$), the distributions used to generate the coefficient vectors in the linear models, the noise standard deviations, the number of contexts ($m$), the number of design points ($p$), the initial sample size ($n_0$), and the minimum detection gap ($\delta$). The realized parameter values for each case are reported in Tables~EC.2--EC.6.

\begin{table}[h]   \label{tab-synthetic linear random setting}
\caption{Synthetic random case settings under the linear setting.}
\centering
\small
\begin{tabular}{ccccccccc}
\hline
Case & $k$  & $d$ & $\bb$    & $\sigma$     & $m$  & $p$ & $n_0$ & $\delta$ \\ \hline
1    & 20 & 2 & $\mathcal{U}(0, 5)$ & $\mathcal{U}(0.5, 2)$ & 6  & 2 & 10 & 0.1   \\
2    & 5  & 3 & $\mathcal{U}(0, 5)$ & $\mathcal{U}(0.5, 2)$ & 81 & 4 & 10 & 0.1   \\
3    & 10 & 4 & $\mathcal{U}(0, 5)$ & $\mathcal{U}(0.5, 2)$ & 64 & 8 & 20 & 0.1   \\
4    & 5  & 2 & $\mathcal{U}(0, 5)$ & $\mathcal{U}(0.5, 2)$ & 6  & 2 & 10 & 0.1   \\
5    & 10 & 3 & $\mathcal{U}(0, 5)$ & $\mathcal{U}(0.5, 2)$ & 36 & 4 & 10 & 0.1   \\ \hline
\end{tabular}
\end{table}

Table~EC.2 - EC.6 reports the realized parameter values generated for synthetic Case~1-5 under the linear setting. The context values are denoted by $X_d$, the coefficients for each context dimension by $\beta^i$, and $\sigma$ by the standard deviation of noise in each linear model.

\begin{table}[H]
\centering
\caption{Realized parameter values for synthetic Case~1 under the linear setting.}
\adjustbox{max width=\textwidth}{
\begin{tabular}{c|cccccccccccccccccccc}
\hline
$X_d$ & 0.000 & 0.200 & 0.400 & 0.600 & 0.800 & 1.000 & & & & & & & & & & & & & & \\ \hline
$\beta^1$ & 2.717 & 1.392 & 2.123 & 4.224 & 0.024 & 0.608 & 3.354 & 4.129 & 0.684 & 2.875 & 4.457 & 1.046 & 0.927 & 0.542 & 1.098 & 4.893 & 4.058 & 0.860 & 4.081 & 1.370   \\ 
$\beta^2$ & 2.159 & 4.700 & 4.088 & 1.681 & 0.877 & 1.864 & 0.028 & 1.262 & 3.978 & 0.076 & 2.994 & 3.019 & 0.526 & 1.910 & 0.182 & 4.452 & 4.905 & 0.300 & 4.453 & 2.885 \\
$\sigma$ & 1.614 & 1.445 & 1.373 & 0.531 & 0.815 & 1.317 & 1.654 & 0.876 & 0.929 & 1.779 & 1.963 & 1.827 & 1.039 & 1.398 & 1.032 & 1.010 & 0.767 & 0.857 & 0.567 & 1.258 \\
\hline
\end{tabular}}
\end{table}

\begin{table}[H]
\centering
\caption{Realized parameter values for synthetic Case~2 under the linear setting.}
{\fontsize{9pt}{11pt}\selectfont
\begin{tabular}{c|ccccccccc}
\hline
$X_d$ & 0.000 & 0.125 & 0.250 & 0.375 & 0.500 & 0.625 & 0.750 & 0.875 & 1.000 \\ \hline
$\beta^1$ & 0.785 & 3.162 & 0.538 & 1.948 & 3.539 & & & & \\
$\beta^2$ & 4.369 & 1.959 & 1.858 & 2.643 & 0.181 & & & & \\
$\beta^3$ & 4.460 & 3.745 & 4.487 & 4.483 & 3.611 & & & & \\
$\sigma$ & 0.807 & 1.516 & 1.910 & 1.884 & 0.964 & & & & \\
\hline
\end{tabular}}
\end{table}

\begin{table}[H]
\centering
\caption{Realized parameter values for synthetic Case~3 under the linear setting.}
{\fontsize{9pt}{11pt}\selectfont
\begin{tabular}{c|cccccccccc}
\hline
$X_d$ & 0.000 & 0.333 & 0.667 & 1.000 & & & & & & \\ \hline
$\beta^1$ & 1.188  & 0.010  & 1.007  & 4.670  & 3.887  & 0.887  & 3.029  & 3.469  & 3.439  & 2.691 \\
$\beta^2$ & 1.509  & 2.881  & 4.106  & 4.204  & 4.520  & 2.345  & 3.607  & 2.460  & 2.628  & 2.680 \\
$\beta^3$  & 3.375  & 0.053  & 4.399  & 2.955  & 3.696  & 4.677  & 2.927  & 2.872  & 0.758  & 3.023 \\
$\beta^4$  & 3.564  & 0.618  & 4.562  & 1.130  & 1.157  & 0.735  & 4.564  & 4.532  & 1.726  & 4.231 \\
$\sigma$ & 0.980  & 1.799  & 0.977  & 0.637  & 0.744  & 1.328  & 0.725  & 1.171  & 1.838  & 0.593 \\
\hline
\end{tabular}}
\end{table}

\begin{table}[H]
\centering
\caption{Realized parameter values for synthetic Case~4 under the linear setting.}
{\fontsize{9pt}{11pt}\selectfont
\begin{tabular}{c|cccccc}
\hline
$X_d$ & 0.000 & 0.200 & 0.400 & 0.600 & 0.800 & 1.000 \\ 
Weights & 0.262 & 0.260 & 0.162 & 0.198 & 0.092 & 0.025 \\ \hline
$\beta^1$ & 0.713 & 4.669 & 4.732 & 3.011 & 1.939 & \\
$\beta^2$ & 1.816 & 1.022 & 1.384 & 1.233 & 0.868 & \\
$\sigma$ & 1.195 & 1.263 & 0.633 & 1.292 & 1.988 & \\
\hline
\end{tabular}}
\end{table}

\begin{table}[H]
\centering
\caption{Realized parameter values for synthetic Case~5 under the linear setting.}
{\fontsize{9pt}{11pt}\selectfont
\begin{tabular}{c|cccccccccc}
\hline
$X_d$ & 0.000 & 0.200 & 0.400 & 0.600 & 0.800 & 1.000 & & & & \\ 
Weights & 0.201 & 0.277 & 0.180 & 0.324 & 0.014 & 0.004 & & & & \\ \hline
$\beta^1$ & 3.390 & 4.812 & 0.096 & 0.456 & 1.300 & 0.263 & 4.738 & 2.579 & 2.634 & 2.613 \\
$\beta^2$ & 4.322 & 0.566 & 0.041 & 1.714 & 1.319 & 4.937 & 3.112 & 4.805 & 3.875 & 3.636 \\
$\beta^3$  & 2.951 & 2.916 & 0.154 & 4.014 & 0.652 & 2.364 & 3.178 & 4.584 & 1.465 & 1.736 \\
$\sigma$ & 0.903 & 1.621 & 1.494 & 0.916 & 0.704 & 1.335 & 0.588 & 0.715 & 1.658 & 0.904 \\
\hline
\end{tabular}}
\end{table}







%

\end{document}